\newif\ifEJP
\newif\ifArxiv

\Arxivtrue

\ifEJP
\documentclass[EJP]{ejpecp}
\fi

\ifArxiv
\documentclass[reqno,a4paper]{amsart}
 \usepackage{bbding}
 
\providecommand{\AMSSUBJ}[1]
{
  \small	
  \textbf{\textbf{MSC2020 subject classifications:}} #1
} 
\usepackage{amsmath}
\numberwithin{equation}{section} 
\fi

\usepackage{lipsum}
\usepackage{amsmath}
\usepackage{amsfonts}
\usepackage{amssymb}
\usepackage{bbm}
\usepackage{graphicx}
\usepackage{multicol}
\usepackage{fancybox}
\usepackage{soul}
\usepackage{mathtools,leftindex}

\usepackage[utf8]{inputenc}
\usepackage[T1]{fontenc}
\usepackage{frcursive}
\usepackage{calligra}
\usepackage{marginnote}
\usepackage{dsfont}
\usepackage{float}
\usepackage{physics}
\usepackage[dvipsnames]{xcolor}
\usepackage{tikz}
\usepackage{comment}
\usepackage{hyperref}
\hypersetup{
    colorlinks=true,
    linkcolor=blue,
    filecolor=blue,      
    urlcolor=blue,
    citecolor=red,
    }
\usepackage{bm}
\usepackage[shortlabels]{enumitem}
\usepackage{makecell}
\usepackage{mathrsfs}
\usepackage{subfigure}
\DeclareMathAlphabet{\mathpzc}{OT1}{pzc}{m}{it}

\ifArxiv
\usepackage{amsthm}

\newtheorem{theorem}{Theorem}[section]

\newtheorem{lemma}[theorem]{Lemma}
\newtheorem{remark}[theorem]{Remark}
\newtheorem{proposition}[theorem]{Proposition}
\newtheorem{assumption}[theorem]{Assumption}
\newtheorem{definition}[theorem]{Definition}
\newtheorem{example}[theorem]{Example}
\newenvironment{acks}
{
  \noindent\textbf{Acknowledgements:}\begin{itshape}%
}%
{\end{itshape}}
\fi

\ifpdf
  \DeclareGraphicsExtensions{.eps,.pdf,.png,.jpg}
\else
  \DeclareGraphicsExtensions{.eps}
\fi

\newcommand{\ho}[1]{\textcolor{RedOrange}{(HO: #1)}}

\newcommand{\nl}[1]{{\overline #1}}
\newcommand{\mom}[2]{\langle #1 \rangle_{#2}}

\newcommand{\upx}{\mathbf{u}^{\inf}}
\newcommand{\upy}{\mathbf{u}^{\sup}}

\newcommand{\N}     {\mathbb{N}}        
\newcommand{\Q}     {\mathbb{Q}}        
\newcommand{\R}     {\mathbb{R}}        
\newcommand{\prob}     {\mathbb{P}}        

\newcommand{\E}{\mathbb{E}}
\newcommand{\Pb}{\mathbb{P}}

\newcommand{\FF}    {\mathcal{F}}
\newcommand{\GG}    {\mathcal{G}}

\newcommand{\Ps}    {\mathcal{P}}

\newcommand{\Lip}{\textbf{Lip}}
\newcommand{\Couplings} {\mathbf{\Pi}}

\newcommand{\inner}[2] { \left \langle #1 , #2  \right \rangle }

\DeclareMathOperator{\Law}{\textbf{Law}}
\DeclareMathOperator{\Was} {\textbf{Was}}

\newcommand{\wZ} {{Z}^{\downarrow}}

\newcommand{\Zbo}{\mathbf{Z}}

\newcommand{\dif} {X}
\newcommand{\nlDif}{\nl{X}}
\newcommand{\nlLaw}{\nl{\mu}}
\newcommand{\difMu} {X^{(\mu)}}

\newcommand{\difNu} {X^{(\nu)}}
\newcommand{\self} {\overline{X}}

\newcommand{\NLchain}{\nl{Z}} 
\newcommand{\NLchainP}{\nl{Z}^{(\NLpf)}} 
\newcommand{\NLchainQ}{\nl{Z}^{(\NLqf)}} 
\newcommand{\sNLchainN}{\nl{X}^{(N)}} 

\newcommand{\wM} {{M}^{\downarrow}}
\newcommand{\sPSsM}{\mathbb{X}^{(N)}} 

\newcommand{\bmath}[1]{\boldsymbol{#1}}

\newcommand{\pf}{{\bf p}}

\newcommand{\pfn}{{\bf p}_N}
\newcommand{\NLpf}{{\bf \nl{p}}}
\newcommand{\NLqf}{{\bf \nl{q}}}
\newcommand{\NLpfn}{{\bf \nl{p}}_N}
\newcommand{\qf}{{\bf q}}

\newcommand{\p}{{p}}
\newcommand{\NLb}{{\bf \nl{b}}}

\newcommand{\bb}{{\bf b}}

\newcommand{\psd}[1]{\mathbf{Z}^{(\delta,\NLpf,{#1})}}
\newcommand{\psdi}[1]{{Z}^{(\delta,\NLpf,{#1})}}
\newcommand{\psod}[1]{\mathbf{Z}^{(\NLpf,{#1})}}
\newcommand{\psodi}[1]{{Z}^{(\NLpf,{#1})}}
\newcommand{\psdN}[1]{\mathbf{Z}^{(\delta_N,\NLpf_N,{#1})}}

\newcommand{\nlZn}{\nl{Z}^{(\delta_N,\NLpfn)}}

\ifEJP
\newtheorem*{theorem*}{Theorem}
\fi

\newcommand{\domby}{\le_{\textbf{st}}}
\renewcommand{\dd}{{\rm d}}
\newcommand{\unif}{{\rm Unif}([0,1])}
\newcommand{\betadist}{{\rm Beta}}
\newcommand{\gammadist}{{\rm Gamma}}

\usepackage{framed}

\newcommand{\articleName}{Wright--Fisher kernels: from linear to non-linear dynamics, ergodicity and McKean--Vlasov scaling limits}

\newcommand{\articleAbstract}{We study the evolution of a pathogen with two allelic types infecting a population of hosts, where within-host type frequencies evolve in discrete time. Our framework is built on a two-parameter family of transition kernels on $[0,1]$, which describe one-step updates of type frequencies.
In the absence of host interaction, the single-host type-frequency process admits, for a broad class of parameters, a moment dual with a branching–coalescing structure reminiscent of the \emph{Ancestral Selection Graph}. Under suitable parameter and time scalings, it converges to a Wright--Fisher diffusion with drift. To incorporate interactions among hosts, we introduce a mean-field mechanism whereby within-host dynamics depend on the empirical type distribution across the population. We prove uniform-in-time propagation of chaos, comparing the dynamics in a typical host with a corresponding non-linear Markov chain. Under appropriate scaling, this non-linear chain converges to a McKean–Vlasov Wright–Fisher diffusion. As an illustration, we analyse a model where mutation rates depend on the current type distribution across hosts and establish uniform-in-time propagation of chaos together with ergodicity of the limiting McKean–Vlasov equation.
}

\ifEJP

\SHORTTITLE{Wright--Fisher kernels: from linear to non-linear dynamics}

\TITLE{\articleName \support{F.C. was funded by the Deutsche Forschungsgemeinschaft (DFG, German Research Foundation) --- Project-ID 317210226 --- SFB 1283; C.J. was funded by Universidad de Santiago de Chile, DMCC --- Beca de Trabajo POC 2022-DMCCUSA215; H.O. was partially funded by ANID Exploraci\'on, project number 13220168-2023 \emph{Biological and quantum Open System Dynamics: evolution, innovation and mathematical foundations} and FONDECYT Regular Nº1242001 \emph{Propagation of chaos for mean-field interacting particle systems in mathematical physics and mathematical biology}; L.V. was funded by FONDECYT Iniciaci\'on, project number 11240158-2024 \emph{Adaptive behavior in stochastic population dynamics and non-linear Markov processes in ecoevolutionary modeling}.
    }} 

\AUTHORS{ 
Fernando Cordero \footnote{BOKU University, Institute of Mathematics, Vienna, Austria. \BEMAIL{fernando.cordero@boku.ac.at}}\orcid{0000-0001-5959-0787}
    \and
Christian Jorquera\footnote{Universidad de Santiago de Chile, DMCC, Chile. 
\BEMAIL{christian.jorquera@usach.cl}.}    
    \and
H\'ector Olivero \footnote{Universidad de Valpara\'iso, CIMFAV and Instituto de Ingenier\'ia Matem\'atica, Chile.\BEMAIL{hector.olivero@uv.cl}}\orcid{0000-0003-3589-0130}
    \and
Leonardo Videla\footnote{Universidad de Santiago de Chile, DMCC, Chile.\BEMAIL{leonardo.videla@usach.cl}}\orcid{0000-0003-3423-6508}
    }


\AMSSUBJ{60K35; 92D10; 92D15} 

\SUBMITTED{2024} 
\ACCEPTED{December 20, 2025} 

\VOLUME{0}
\YEAR{2024}
\PAPERNUM{0}
\DOI{0}

\ABSTRACT\articleAbstract
\fi

\ifArxiv
	\title[Wright--Fisher kernels: from linear to non-linear dynamics]{\articleName}
	
    \author{F. Cordero\affmark[1] \and C. Jorquera \affmark[2] \and H. Olivero\affmark[3]\and  L. Videla \affmark[4]}
\address{\newline\noindent\affmark[1]BOKU University, Institute of Mathematics, Vienna, Austria.\newline
\affmark[2]Universidad de Santiago de Chile, DMCC, Chile.\newline 
\affmark[3]Universidad de Valpara\'iso, CIMFAV and Instituto de Ingenier\'ia Matem\'atica, Chile.\newline
\affmark[4]Universidad de Santiago de Chile, DMCC, Chile.}
\email{\affmark[1]fernando.cordero@boku.ac.at, \affmark[2] christian.jorquera@usach.cl, \affmark[3]hector.olivero@uv.cl,\newline \affmark[4]leonardo.videla@usach.cl}

\date{}                     
\newcommand*{\affmark}[1][*]{\textsuperscript{#1}}
\fi

\begin{document}

\ifArxiv
	\maketitle
	\begin{abstract}
	    \articleAbstract
	\end{abstract}
	
	\bigskip
{\small MSC: Primary 60K35; secondary 92D10; 92D15 

}
\fi

\tableofcontents

\ifArxiv

\fi
\section{Introduction}
Consider a large population of host individuals infected by a contagious pathogen. 
The pathogen reproduces much faster than its host; thus, over the evolutionary 
timescale of the pathogen, the host population size may be regarded as constant.  
We focus on a diallelic locus in the pathogen genome (with alleles $0$ and $1$) 
that determines two distinct phenotypic strains. Within each host, the allele 
frequency evolves under mutation and selection. In addition, we assume that the pathogen dynamics inside each host are influenced by the global state of the population. More precisely, mutation and/or selection 
rates may depend not only on the local allele frequency within a host but also on 
the empirical distribution of allele frequencies across all hosts. This type of 
weak, population-level coupling is captured mathematically through a mean-field 
interaction and gives rise to a non-linear evolution at the host level.

This dynamic can be illustrated through the following experimental scenario. Consider many identical bacterial cultures (hosts), each infected 
by bacteriophages (the pathogen) of two monitored strains. Strain~$0$ is highly 
virulent but environmentally sensitive, whereas Strain~$1$ is less virulent but 
more robust. Within each culture, phage replication is subject to mutation between 
strains and to local selection. At regular times, the experimenter estimates the 
overall frequency of Strain~$1$ across all cultures. This empirical average then 
determines an environmental adjustment (e.g.\ a $pH$ shift) that subsequently 
affects selection locally in every culture. For instance, if resistance becomes 
widespread, the induced environmental change may intensify selection against 
Strain~$0$ globally. Thus, the selective pressure experienced by pathogen populations 
within individual cultures depends, in part, on the collective state of the system. 
This feedback produces a mean-field regime and leads to non-linear evolutionary 
dynamics.

The setting described above is characteristic of current research at the interface of ecology and evolution (see, e.g., \cite{urban2008} and, in a closely related context, \cite{stubbendieck2016}). Multi-level eco-evolutionary models, here involving a host level and a community level, have been considered previously. For instance, \cite{luo2014, luo2017} study an urn–ball scheme in which a trait evolves locally according to a Moran model, being advantageous at the individual level but disadvantageous at the group level.
Here, we focus on large communities of hosts whose size remains approximately constant over relevant time scales. Our goal is to quantify how interactions between individuals, the local community, and possibly more distant communities influence pathogen evolution, with particular emphasis on the long-term dynamics.
We introduce a simple modeling framework addressing the scenario above at two levels of complexity:
\begin{itemize}
\item \textbf{Single-host level (one-level dynamics):} evolution of the frequency of allele~$0$ in the pathogen population within a single host, ignoring inter-host interactions.
\item \textbf{Deme level (two-level dynamics):} evolution under non-negligible interactions among hosts within a deme.
\end{itemize}
The framework naturally extends to a third \emph{meta-community} level in which hosts from different demes interact, allowing one to explore how abiotic or biotic heterogeneity affects pathogen evolution and coinfection dynamics. To keep the presentation focused, we restrict attention here to the first two levels; the meta-community extension will be pursued in future work.

In our model the fundamental entities are the hosts (as in mathematical epidemiology) rather than the pathogen particles (as in classical population genetics). This choice is particularly appropriate when empirical data are collected at the host level. The pathogen’s genetic evolution is represented in a mesoscopic, mass-action fashion, abstracting from the behavior of individual alleles. Nevertheless, we show that this formulation is consistent with familiar allele-level models such as the Wright–Fisher, Moran, and Cannings models. Moreover, it provides a tractable means of incorporating host-level dynamics into intra-deme interactions.

Our construction begins with a two-parameter family $(P_{\delta,p})_{\delta,p\in[0,1]}$ of kernels on $[0,1]$, which govern one-step transitions in both one- and two-level dynamics. We call these \emph{Wright--Fisher kernels}. They exhibit strong ergodicity and monotonicity properties that carry over to the models built from them.

Building on these kernels, the paper develops two modeling layers. At the 
single-host level, we introduce a discrete-time Wright--Fisher--type 
model in which within-host evolutionary forces are encoded by a function~$\pf$. 
For suitable choices of~$\pf$, the resulting dynamics admit a moment dual. In the \emph{neutral case} (i.e.\ when $\pf$ is the identity), the dual corresponds to a discrete-time analogue of the 
$\Lambda$-coalescents from population genetics (see \cite{Ber09} for a review). More generally, when $\pf$ additionally incoporates mutation and a specific form of frequency-dependent selection, the dual process takes the form of a branching--coalescing system with killing. Finally, under diffusive rescaling, the model converges to a family of 
Wright--Fisher diffusions with drift.

We then consider a nonlinear extension in which the local evolutionary parameters depend not only on the within-host composition but also on the ambient composition of the host population (a mean-field effect). Under mild regularity, this nonlinear Markov chain admits a unique stationary distribution and converges to it.

The second layer consists of $M$ interacting hosts coupled through their empirical distribution. Conditional on this empirical measure, pathogen populations evolve independently within hosts. We prove \emph{uniform-in-time propagation of chaos}: when suitably initialized, the law of a typical host remains uniformly close to that of the nonlinear single-host model. In a diffusive scaling regime, the nonlinear model converges to a McKean--Vlasov--type Wright--Fisher diffusion, yielding a continuous-time description of the two-level dynamics.

The paper is organised as follows. Section~\ref{sec:overview} outlines the main results. Section~\ref{ss3.1} develops the Wright--Fisher kernels and their core properties. Section~\ref{section:host_level} presents the single-host model, its duality, and diffusion limits. Section~\ref{section:deme_level} introduces the nonlinear mean-field model and establishes convergence towards a McKean--Vlasov limit. Section \ref{sec:prop_chaos_results} introduces the interacting particle system modeling the two-layers dynamics and establishes a uniform propagation of chaos result. Section~\ref{sec:case_study} provides a case study, and Section~\ref{section:discussion} discusses open directions. Technical material and deferred proofs are collected in the Appendices \ref{section:appendix} and \ref{ss:existence_MV}.

\section{Preliminaries and main results}\label{sec:overview}
\subsection{Notations and preliminaries}\label{sec:notation}
Let $(E,d)$ be a metric space and $\mathcal{B}(E)$ its Borel $\sigma$-algebra; when no confusion arises, we also write $\mathcal{B}(E)$ for the space of real-valued Borel measurable functions.  
We denote by $\mathcal{P}(E)$ the space of probability measures on $(E,\mathcal{B}(E))$, endowed with the topology of weak convergence; for $E=[0,1]$ we write simply $\mathcal{P}$.  
For $\mu \in \mathcal{P}(\R)$, we write $\mom{\mu}{k}$ for its $k$-th moment and $\mathrm{var}(\mu)$ for its variance.

Given a probability space $(\Omega,\mathcal{F},\mathbb{P})$, a measurable space $(E,\mathcal{E})$, and an $E$-valued random variable $X$, we denote the law of $X$ by $\Law_\mathbb{P}(X)$, or simply $\Law(X)$ when unambiguous.  
We write $U \sim \mathrm{Unif}([0,1])$ if $U$ is uniformly distributed on $[0,1]$.

A \emph{transition kernel} on $(E,\mathcal{E})$ is a map
$Q : E \times \mathcal{E} \to \R_+$
such that $A \mapsto Q(x,A)$ is a probability measure for every $x \in E$ and $x \mapsto Q(x,A)$ is measurable for every $A \in \mathcal{E}$.  
Its action on bounded measurable $f$ and on $\mu \in \mathcal{P}(E)$ is given by
\[
Qf(x) \coloneqq \int_E f(y)\,Q(x,\mathrm{d}y), 
\qquad 
\mu Q(A) \coloneqq \int_E Q(x,A)\,\mu(\mathrm{d}x).
\]

For $\mu,\nu \in \mathcal{P}(\R)$, we say that $\mu$ is \emph{stochastically dominated} by $\nu$, written $\mu \domby \nu$, if
\[
\mu((t,\infty)) \le \nu((t,\infty)) \quad \text{for all } t \in \R.
\]
Equivalently, there exists a coupling $(X,Y)$ with $\Law(X)=\mu$, $\Law(Y)=\nu$, and $X \le Y$ a.s.; see~\cite{lindvall2002}.

\medskip
Assume $(E,d)$ is Polish with bounded metric.  
For $\mu,\nu \in \mathcal{P}(E)$, let $\Couplings(\mu,\nu)$ denote the set of couplings of $(\mu,\nu)$.  
The $d$-Wasserstein distance is
\[
\Was_d(\mu,\nu)
  \coloneqq \inf_{\pi \in \Couplings(\mu,\nu)} \int_{E^2} d(x,y)\,\pi(\mathrm{d}x,\mathrm{d}y)
  = \inf_{\pi \in \Couplings(\mu,\nu)} \mathbb{E}_\pi[d(X,Y)],
\]
where $(X,Y) \sim \pi$.  
It defines a metric on $\mathcal{P}(E)$, satisfies that $\Was_d$--convergence implies weak convergence, and the infimum is attained~\cite[Thm.~4.1]{villani2009}.  
Moreover, $(\mathcal{P}(E),\Was_d)$ is Polish.  
The dual representation
\begin{equation}\label{eq:duality_was}
\Was_d(\mu,\nu)=\sup_{\Lip(f)\le 1}\bigl(\mu f - \nu f\bigr)
\end{equation}
holds, where $\Lip(f)$ is the minimal $d$-Lipschitz constant of $f$.  
When $E=[0,1]$ with its usual metric, we often write $\Was$ for~$\Was_d$.

\medskip
The total variation distance is
\[
d_{TV}(\mu,\nu) \coloneqq \inf_{\pi \in \Couplings(\mu,\nu)} \mathbb{P}_\pi(X \neq Y).
\]
If ${\rm diam}(E)$ denotes the diameter of $(E,d)$, then
\begin{equation}\label{eq:wass-TV-inequality}
\Was_d(\mu,\nu) \le {\rm diam}(E)\, d_{TV}(\mu,\nu).
\end{equation}

If $E$ is complete and separable, we denote by $\mathcal{C}([0,T];E)$ the space of continuous paths, equipped with the supremum norm $\lVert\cdot\rVert_\infty$.  
We write $\mathbb{D}([0,T];E)$ for the space of c\`adl\`ag paths, equipped with the Billingsley metric generating the $J_1$--Skorokhod topology; this space is also Polish~\cite{billingsley1999}.  
In both cases, the Borel $\sigma$-algebra is generated by the coordinate maps $x \mapsto x_t$, ($t \in [0,T]$). These maps are continuous for the supremum norm on $\mathcal{C}([0,T];E)$, but not for the $J_1$--Skorokhod topology.

\subsection{Main results}\label{sec:overview-results}

The aim of this section is to provide a more detailed preview of the article and to guide the reader through the various developments. We refrain from giving a comprehensive summary of our findings, as this would require the introduction of a substantial amount of notation.

As stated in the introduction, the one-level and two-level models are based on a two-parameter family $(P_{\delta,\p})_{\delta,\p \in [0,1]}$ of kernels on $[0,1]$. These kernels act on bounded measurable functions $h : [0,1] \mapsto \R$ via
\begin{align*}
	P_{\delta, \p} h (x) \coloneqq \p\,\E\big[h \big(x+(1-x)\delta U\big)\big] + (1-\p)\, \E\big[h \big(x(1-\delta U)\big)\big], \quad \delta,\p,x \in [0,1],
\end{align*}
where $U \sim \unif$. We refer to these kernels as \emph{Wright--Fisher kernels}; see Fig.~\ref{fig:transport} for an illustration of the action of the kernel. 
For fixed \(\delta,\p \in [0,1]\), the discrete-time Markov chain driven by
\(P_{\delta,\p}\) describes the evolution of the fraction of type-\(0\)
(pathogen) particles within a host. If the current proportion of type-\(0\)
particles is \(x \in [0,1]\), then a random fraction \(\delta U\), with
\(U \sim \unif\) of the population dies and is replaced by clonal offspring:
with probability \(\p\) from a type-\(0\) particle, and with probability
\(1-\p\) from a type-\(1\) particle. The boundary states \(x=0\) and \(x=1\)
correspond to extinction and fixation of type~\(0\), respectively. Note that for initial
states in \((0,1)\) and fixed \(\p\in(0,1)\), neither boundary is reached in finite time.

This mechanism is reminiscent of \(\Lambda\)–Wright–Fisher dynamics, in which
reproduction events occur in continuous time at Poisson points \((t,r)\):
with probability \(x\), a type-\(0\) individual replaces a random fraction \(r\)
of the population by its offspring, and with probability \(1-x\), a type-\(1\)
individual does so. Our discrete-time model can thus be viewed as a simplified
analogue in which reproduction events occur at regular time steps and a random
fraction \(\delta U\) of the population is replaced at each update (see
\cite{Bertoin2003, Bertoin2005}). We return to this connection in
Section~\ref{ss:pp_repres}.

\begin{figure}[t!]
\centering
\begin{tikzpicture}[scale=0.9]  
    \draw[color=black, fill = Cyan] (3, 5) rectangle (6,5.5) node [pos=.5] {$x$};
    \draw[color=black, fill =black!20!white] (10,5) rectangle (6,5.5) node [pos=.5] {$1-x$};
    
    \draw [thick, ->] (6,4.9) -- (2,3.1) node[midway,above left] {$p$};
    \draw [thick, ->] (7,4.9) -- (11,3.1) node[midway,above right] {$1-p$};
    
    \draw[color=black, fill = Cyan] (-1,2.5) rectangle (2,3) node [pos=.5] {$x$};
    \draw[color=black, fill = Cyan] (4,2.5) rectangle (2,3) node [pos=.5] {$\scriptstyle \delta U (1-x)$};
    \draw[color=black, fill = black!20!white] (6,2.5) rectangle (4,3) node [pos=.5] {$\scriptstyle (1-x)(1-\delta U)$};
    
    \draw[color=black, fill = Cyan] (7,2.5) rectangle (8.5,3) node [pos=.5] {$\scriptstyle x(1-\delta U)$};
    \draw[color=black, fill = black!20!white] (10,2.5) rectangle (8.5,3) node [pos=.5] {$\scriptstyle \delta U x$};
    \draw[color=black, fill = black!20!white] (14,2.5) rectangle (10,3) node [pos=.5] {$1-x$};
  \end{tikzpicture}
  \caption{The action of the Wright--Fisher kernel $P_{\delta, p}$.}
\label{fig:transport}\end{figure}
\subsubsection{Ergodic properties of the Wright--Fisher kernels}
In Section~\ref{ss3.1}, we establish key properties of the Wright--Fisher kernels, culminating in the following result.
\begin{theorem}[Ergodicity of Wright--Fisher kernels]\label{cor:unique}
	For any $\delta \in (0,1]$ and $\p \in [0,1]$, the kernel $P_{\delta, \p}$ admits a unique invariant distribution, denoted by $\beta_{\delta, \p}$. Moreover, for any $\mu\in\mathcal{P}$, $\mu (P_{\delta,p})^n$ converges in distribution as $n \to \infty$ to $\beta_{\delta, \p}$, and the rate of convergence is given by  
	\begin{align}\label{eq:convergence_rate0}
		\Was (\mu (P_{\delta,p})^n, \beta_{\delta, \p}) 
		\le  \bigg(1 - \frac{\delta}{2}\bigg)^{\!n} 
		\Was (\mu, \beta_{\delta, \p}).
	\end{align}   
\end{theorem}
Further properties of $\beta_{\delta,\p}$, for $\delta \in (0,1]$ and $\p \in [0,1]$, are summarized in Proposition~\ref{prop:resumen}. The proof of Theorem~\ref{cor:unique} relies on a coupling argument comparing the one-step transitions of the kernels $P_{\delta,\p}$ and $P_{\delta,q}$, with $\p,q \in [0,1]$, starting from different initial distributions. This coupling, along with several others used throughout the paper, is based on Lemma~\ref{uple}, which collects fundamental properties of two families of functions encoding the transitions of our Wright--Fisher kernels.

\subsubsection{The single host level}
In Section~\ref{section:host_level}, we focus on the one-level dynamics, i.e. the
evolution of pathogen alleles within a single host, ignoring interactions
between hosts. The model is specified by a constant $\delta \in (0,1]$ and a
function $\pf : [0,1] \to [0,1]$. Given these parameters, the fraction of
allele-$0$ pathogens evolves according to the transition kernel
$x \mapsto P_{\delta,\pf(x)}(x,\cdot)$; we denote the resulting Markov chain by
$Z^{(\delta,\pf)} \coloneqq (Z^{(\delta,\pf)}_n)_{n \in \N}$.

The function $\pf$ encodes the within-host bias in reproduction, which may either favor or disfavor
allele~$0$, and may depend on its current frequency $x$. This allows
within-host selective effects or replication rates to vary with the local
composition of the pathogen population, while the random replacement mechanism
(through $\delta$ and $U$) preserves a Wright--Fisher-type stochastic structure.

After establishing general properties of $Z^{(\delta,\pf)}$, we examine in
detail the example $\pf = \pf^{\downarrow}_\delta$, where
\[
\pf^{\downarrow}_\delta(x)
 \coloneqq x + \frac{2\delta}{3}\!\left(
   -x(1-x)\sum_{i=0}^\infty \sigma_i x^i
   + \theta_0(1-x) - \theta_1 x
 \right),
\]
with $\delta>0$, $\theta_0,\theta_1 \ge 0$, and 
$(\sigma_k)_{k \ge 0}$ a sequence of nonnegative coefficients satisfying the 
additional conditions in Assumption~\ref{eq:non_increasing}.  
The neutral case $\pf^\downarrow(x)=x$ for all $x\in[0,1]$ corresponds to 
$\sigma_i = \theta_0 = \theta_1 = 0$ for all $i \ge 0$. In the general case, the deviation from neutrality is modulated by $\delta$ and 
arises from two contributions: (i) a frequency-dependent selection term 
$s(x) = -x(1-x)\sum_{i\ge0} \sigma_i x^i$, which acts against type~$0$, and 
(ii) a mutation term $\vartheta(x) = \theta_0(1-x) - \theta_1 x$.  
The condition imposed on $(\sigma_i)_{i\ge0}$ corresponds to the 
\emph{fittest-type-wins} mechanism of~\cite{GS18}, and is precisely what allows 
us to show that $Z^{(\delta,\pf^\downarrow_\delta)}$ admits a moment dual 
$M^{(\delta,\pf^\downarrow_\delta)} = (M^{(\delta,\pf^\downarrow_\delta)}_n)_{n\in\N}$, 
with transition rates given in~\eqref{ratedual}.  The precise statement is:
\begin{theorem}[Moment duality]\label{th:moment_duality_II}
For every $n,m \ge 0$ and $z \in [0,1]$,
\begin{equation}\label{momdi2}
    \E_z\!\left[\!\left(Z^{(\delta,\pf^{\downarrow}_\delta)}_n\right)^{\!m}\right]
    = \E_m\!\left[z^{M_n^{(\delta,\pf^{\downarrow}_\delta)}}\right].
\end{equation}
\end{theorem}
In Section~\ref{sec:appli}, we apply Theorem~\ref{th:moment_duality_II} to determine the long-term behavior of the chain $Z^{(\delta,\pf)}$ in two settings: 
(i) the case without mutation (see Proposition~\ref{pr-abs}), and 
(ii) the case with mutation but without selection (see Proposition~\ref{nomut}).

Section~\ref{section:host_level} concludes by considering the scaling limit of the chain $Z^{(\delta_N,\pfn)}$, with 
\begin{align} \label{eq:f_N}
\delta_N \coloneqq \sqrt{\frac3N}
\quad \text{and} \quad 
\pfn(x) \coloneqq x + \frac{2}{\sqrt{3N}}\,b(x), \quad x \in [0,1],
\end{align}
where the function $b$ satisfies the following assumption.
\begin{assumption}\label{eq:lip_bound}
The function $b : [0,1] \mapsto \R$ is $L_1$-Lipschitz on $[0,1]$ and satisfies 
$b(0) \ge 0$ and $b(1) \le 0$.
\end{assumption}
We note that, under Assumption~\ref{eq:lip_bound}, there exists $N_0 \coloneqq N_0(L_1) \in \N$ such that $\pfn(x) \in [0,1]$ for all $x \in [0,1]$ and all $N \ge N_0$. We therefore assume in all what follows that $N \ge N_0$.
The scaling in~$\delta_N$ and~$\pfn$ corresponds to a weak selection and small random replacement regime. The next result shows that under this assumption, the discrete Wright--Fisher-type dynamics converge to the classical Wright--Fisher diffusion with drift~$b$. 
\begin{theorem}[Diffusion limit]\label{thm:connection-WF}
Assume that $b : [0,1] \mapsto \R$ satisfies Assumption~\ref{eq:lip_bound}. Let 
$X^{(N)}_{t} \coloneqq Z^{(\delta_N,\pfn)}_{\lfloor Nt \rfloor}$, $t \ge 0$, and assume that $X_0^{(N)}$ converges in distribution to $x_0 \in [0,1]$. Then the process $X^{(N)}$ converges in distribution as $N \to \infty$ to the pathwise unique strong solution of the SDE 
\begin{equation}\label{eq:WF-Phi-Psi}
\tag{WF}
X_t = X_0 + \int_0^t b(X_s)\,\dd s + \int_0^t \sqrt{X_s(1-X_s)}\,\dd W_s,
\end{equation}
with $X_0 = x_0$, where $W$ is a standard Brownian motion.
\end{theorem}
\begin{remark}
Assumption~\ref{eq:lip_bound} also implies that the solution of the SDE~\eqref{eq:WF-Phi-Psi} with $X_0 = x_0 \in [0,1]$ remains in $[0,1]$ for all times.
\end{remark}
Further connections to classical population genetics will be discussed throughout Section \ref{section:host_level}.
\subsubsection{Non-linear chains and McKean--Vlasov limits}
In Section~\ref{section:deme_level}, we set up the framework needed to later analyze the second-level dynamics. As discussed in the introduction, the feedback mechanisms act at the level of the entire community and then influence within-host selection uniformly across hosts. When focusing on a single host, this results in a reproduction bias that depends on its own state together with the distribution of states across all hosts. As an intermediate step toward the full multi-host model, we therefore introduce a class of non-linear Markov chains whose transitions depend not only on the current state but also on the current law of the process itself.

The parameters of the non-linear dynamics are a constant $\delta \in (0,1)$ and a function $\nl \pf : [0,1] \times \mathcal{P} \to [0,1]$. The non-linear Wright--Fisher Markov chain ${\nl Z}^{(\delta,\nl\pf)} \coloneqq ({\nl Z}^{(\delta,\nl\pf)}_n)_{n \in \N}$ is defined by specifying an initial distribution and the following updating scheme. If ${\nl Z}_n^{(\delta,\nl\pf)} \sim \nu_n$, then 
\begin{equation}\label{nl-trans}
{\nl Z}_{n+1}^{(\delta,\nl\pf)} \sim P_{\delta, \nl\pf({\nl Z}_n^{(\delta,\nl\pf)}, \nu_n)}({\nl Z}_n^{(\delta,\nl\pf)}, \cdot).
\end{equation}
To establish our next result, we introduce the following notion.
\begin{definition}[$(L_1,L_2)$-Lipschitz]\label{ass:Lipschitz} 
Let $L_1,L_2 > 0$ and $\NLpf : [0,1] \times \mathcal{P} \mapsto [0,1]$. We say that $\NLpf$ is $(L_1,L_2)$-Lipschitz if for every $x,y \in [0,1]$ and $\mu,\nu \in \mathcal{P}$,
	\begin{align*}
		\vert \NLpf (x, \mu) - \NLpf(y, \nu) \vert 
		\le L_1 \vert x - y \vert + L_2 \Was (\mu, \nu). 
	\end{align*}
\end{definition}
The next result shows the existence of and convergence to the unique invariant distribution for the non-linear chain ${\nl Z}^{(\delta,\nl\pf)}$.
\begin{theorem}[Ergodicity of the non-linear dynamics] \label{th:non_linear_ergodic}
	Let $\NLpf : [0,1] \times \mathcal{P} \mapsto [0,1]$ be a $(L_1,L_2)$-Lipschitz function with $L_1 + L_2 < 1$. Then the non-linear Wright--Fisher Markov chain $(\nl Z^{(\NLpf,\delta)}_n)_{n \ge 0}$ admits a unique invariant distribution $\eta \in \Ps$. Moreover, $\nl Z^{(\NLpf,\delta)}_n$ converges in distribution as $n \to \infty$ to $\eta$, and the rate of convergence is given by  
	\begin{align}\label{eq:convergence_rate}
		\Was (\Law(\nl Z^{(\NLpf,\delta)}_n), \eta) 
		\le  \Bigg(1 - \frac{\delta (1 - L_1 - L_2)}{2}\Bigg)^{\!n} 
		\Was (\Law (\nl Z^{(\NLpf,\delta)}_0), \eta).
	\end{align} 
	Furthermore, in the particular case where $\NLpf(x, \mu)$ does not depend on $x$, we have $\eta = \beta_{\delta,\bmath{\bar{s}}}$, where $\bmath{\bar{s}}$ is the unique solution of the equation $\NLpf(0, \beta_{\delta,s}) = s$ (recall the definition of $\beta_{\delta,s}$ in Theorem~\ref{cor:unique}).
\end{theorem}
Next, we consider scaling limits of these non-linear chains. For this, set
\begin{equation}\label{scale-setting}
\delta_N \coloneqq \sqrt{\frac3N}
\quad \text{and} \quad
{\nl \pf}_N (x,\mu) \coloneqq x + \frac{2}{\sqrt{3N}}\,{\NLb}(x,\mu), 
\quad x \in [0,1],\, \mu \in \mathcal{P},
\end{equation}
where the function $\NLb : [0,1] \times \mathcal{P} \mapsto \R$ satisfies the following.
\begin{assumption}\label{ass:on_b}
$\NLb$ is $(L_1,L_2)$-Lipschitz, and for all $\mu \in \mathcal{P}([0,1])$, 
$\NLb(0,\mu) \ge 0$ and $\NLb(1,\mu) \le 0$.
\end{assumption} 
Note that, under Assumption~\ref{ass:on_b}, there exists $N_1 \coloneqq N_1(L_1,L_2) \in \N$ such that ${\nl\pf}_N(x) \in [0,1]$ for all $x \in [0,1]$ and all $N \ge N_1$. We therefore assume in all what follows that $N \ge N_1$. The following result formalizes the scaling limit of the non-linear Wright--Fisher chains.

\begin{theorem}[McKean--Vlasov limit]\label{th:convergence_theorem}
Assume that Assumption~\ref{ass:on_b} holds and that the McKean--Vlasov Wright--Fisher (MVWF) stochastic differential equation 
\begin{align}\label{eq:MV}
\tag{MVWF}
 \dd {\nl X}_t = \NLb({\nl X}_t,\nl{\mu}_t)\,\dd t 
 + \sqrt{{\nl X}_t(1 - {\nl X}_t)}\,\dd W_t, 
 \quad \nl{\mu}_t = \Law({\nl X}_t), \quad {\nl X}_0 \sim \overline \mu_0,
\end{align}
has a unique weak solution $\nl X$ on $[0,\infty)$ (see Definition~\ref{def:weak_solution}). Let ${\nl X}^{(N)}_{t} \coloneqq {\nl Z}^{(\delta_N,{\nl\pf}_N)}_{\lfloor Nt \rfloor}$, $t \ge 0$, and assume that ${\nl X}^{(N)}_0 \to \mu_0 \in \mathcal{P}$. Then the sequence $(\sNLchainN)_{N \ge 1}$ converges in distribution to the unique solution of the McKean--Vlasov Wright--Fisher diffusion $\nl X$.
\end{theorem}

Although these non-linear chains can be interpreted as a more elaborate version of the one-level dynamics, they arise naturally as approximations to the within-host evolution when hosts interact through population-level processes such as infection or migration. The resulting non-linearity thus encodes the feedback between individual-level dynamics and population-level composition.
\subsubsection{The deme level}

In Section~\ref{sec:prop_chaos_results}, we set up the second-level dynamics. We now assume that there are $M$ hosts interacting within a deme; the interaction between hosts occurs only through their empirical measure. For instance, as illustrated in the introduction, global environmental adjustments may be triggered by the empirical distribution of pathogen types across all hosts; these adjustments then modify within-host selection uniformly throughout the community. This creates a feedback loop in which individual-level reproduction biases depend on the collective state of the population. 

The parameters of the model are, as for the non-linear dynamics, a constant $\delta \in (0,1)$ and a function $\nl \pf : [0,1] \times \mathcal{P} \to [0,1]$. More precisely, we define $\Zbo^{(\delta,\nl\pf,M)}_n = (Z^{(\delta,\nl\pf,M)}_{n,1}, \ldots, Z^{(\delta,\nl\pf,M)}_{n,M})$, where $Z^{(\delta,\nl\pf,M)}_{n,i}$ denotes the fraction of allele-$0$ pathogens in the $i$-th host, $i \in \{1, \ldots, M\}$, at generation~$n$. The process $\Zbo^{(\delta,\nl\pf,M)} \coloneqq (\Zbo^{(\delta,\nl\pf,M)}_n)_{n \in \N}$ evolves according to the following Markovian transitions. If $\Zbo^{(\delta,\nl\pf,M)}_n = {\bf{z}} = (z_1, \ldots, z_M)$ and ${\hat{\mu}_M}({\bf{z}}) \coloneqq \frac{1}{M}\sum_{i=1}^M \delta_{z_i}$, then
$$
\Zbo^{(\delta,\nl\pf,M)}_{n+1} \sim 
P_{\delta, \NLpf(z_1,\hat{\mu}({\bf{z}}))}(z_1,\cdot) \otimes \cdots \otimes 
P_{\delta, \NLpf(z_M,\hat{\mu}({\bf{z}}))}(z_M,\cdot).
$$

The main results of this section come in the form of uniform-in-time propagation of chaos results. The first one compares the laws of $Z^{(\delta,\nl\pf,M)}_{n,1}$ (i.e., the fraction of allele-$0$ pathogens in the first host in the two-level model with $M$ hosts) and ${\nl Z}^{(\delta,\nl\pf)}_n$ (the fraction of allele-$0$ pathogens in the corresponding non-linear model). 

\begin{theorem}[Propagation of chaos I]\label{propchaos}
Let $\NLpf:[0,1]\times\Ps\to[0,1]$ be a $(L_1,L_2)$-Lipschitz function with $L_1+L_2<1$ and let $\mu_0\in\Ps$. Consider the Markov chain $\Zbo^{(\delta,\nl\pf,M)}$ with initial condition $\Zbo^{(\delta,\nl\pf,M)}_0\sim\mu_0^{\otimes M}$ and the non-linear chain $\nl Z^{(\delta,\NLpf)}$ defined via~\eqref{nl-trans} and initial distribution $\mu_0$. Then, there exists $C>0$ (which may depend on $L_1$ and $L_2$, but not on $\delta$) such that, for every $n \ge 0$,
	\begin{align*}
		\Was\Big(\Law (Z^{(\delta,\nl\pf,M)}_{n,1}), \Law (\nl Z^{(\delta,\NLpf)}_n )\Big)
		\le \dfrac{C}{\sqrt{M}}.
	\end{align*}
\end{theorem}

Next, we introduce a scaling parameter $N\in\N$, and let the number of host particles $M(N)$ depend on it. We consider $\delta_N$ and $\NLpfn$ as in~\eqref{scale-setting} and set
\begin{align}\label{def:systems}
\sPSsM_t \coloneqq  \psdN{M(N)}_{\lfloor Nt\rfloor}
\quad \text{and} \quad
\sNLchainN_t \coloneqq \nlZn_{\lfloor Nt\rfloor}, \quad t \ge 0.
\end{align}
In this setting, we show the following propagation of chaos result comparing the laws of $\sPSsM_{t,i}$ and $\nl X_t^{(N)}$.

\begin{theorem}[Propagation of chaos II]\label{th:prop_chaos_2}
Assume that $\delta_N$ and $\NLpfn$ are as in~\eqref{scale-setting}
for a function $\NLb$ satisfying Assumption \ref{ass:on_b}.  
Let $\mu_0^{(N)}$ denote the law of $\sNLchainN_0$ and suppose that  
$\sPSsM_0\sim (\mu_0^{(N)})^{\otimes M(N)}$.  
Then, for all $N$ sufficiently large, we have
\begin{align*}
\Was\big(\Law(\sPSsM_{t,1}), \Law(\nl X^{(N)}_t)\big)
\le \dfrac{C_0L_2}{(L_1+L_2)\sqrt{M(N)}} 
\left(e^{(L_1+L_2)t}-1\right),
\end{align*}
where $C_0>0$ is a constant.
\end{theorem}
If the number $M(N)$ of hosts grows sufficiently fast, and $\NLpf_N$ satisfies suitable additional assumptions, we obtain a uniform-in-time version of the previous result; see Proposition~\ref{prop:uniform_prop_chaos}.
\subsubsection{A case study: self-stabilizing MVWF diffusion}
We conclude with a case study that illustrates the general results by considering a McKean--Vlasov Wright--Fisher diffusion whose drift incorporates both mutation and self-stabilizing effects. This model exemplifies how feedback mechanisms at the population level can induce additional stabilizing forces in the allele-frequency dynamics.

Specifically, we study the McKean--Vlasov model associated with the drift
$$
\nl{\bb}(x,\mu) \coloneqq -\theta_1 x + \theta_0(1-x) - \gamma \left(x - \int_{[0,1]} y\,\mu(\dd y)\right),
$$
where the parameter $\theta_1>0$ (resp.\ $\theta_0>0$) models mutations to type~$1$ (resp.\ type~$0$). The term $x - \int_{[0,1]} y\,\mu(\dd y)$ represents a self-stabilizing effect, and $\gamma>0$ determines its strength. Intuitively, when the state $x$ of an individual deviates from the population mean, the drift term drives it back toward that mean, thereby stabilizing the overall composition.  

The corresponding McKean--Vlasov Wright--Fisher SDE is given by
\begin{equation}\label{eq:wf_mut}
\begin{aligned}
\dd\self_t &= \big(-\theta_1 \self_t + \theta_0(1-\self_t) - \gamma(\self_t-\E[\self_t])\big)\,\dd t  
  + \sqrt{\self_t (1-\self_t)}\,\dd W_t,\\
\self_0 &\sim \overline \mu_0,
\end{aligned}
\end{equation}
where $W$ is a standard Brownian motion.  

Using the Wright--Fisher kernels, the moment estimates from Section~\ref{section:host_level}, and the convergence results from Section~\ref{section:deme_level}, we obtain the following quantitative ergodicity estimate.

\begin{theorem}[Ergodicity]\label{prop:invariant-self}
The process $\self$, solution to~\eqref{eq:wf_mut}, admits a unique stationary distribution $\eta$, which is
the $\betadist \big( 2 \theta_0 (1+ \gamma/(\theta_1+ \theta_0)),\, 2\theta_1 (1+ \gamma/(\theta_1+\theta_0)) \big)$ distribution. Moreover,
\begin{align*}
\Was (\Law (\self_t), \eta) 
\le e^{-(\theta_1+\theta_0)t} 
\,\Was (\Law (\self_0), \eta).
\end{align*}
\end{theorem}
Note that the rate of convergence in Theorem \ref{prop:invariant-self} does not depend on $\gamma$. 

We close Section~\ref{sec:case_study} with the following result, which shows that any individual particle of the Markovian system $\sPSsM$ approximates the dynamics of the McKean--Vlasov Wright--Fisher diffusion~\eqref{eq:wf_mut}.

\begin{theorem}[Propagation of chaos, from discrete-time Markov chains to non-linear diffusions]\label{prop_chaos_sswf}
Let $M(N) = N$, $\delta_N = \sqrt{3/N}$, and $\NLpf_N$ as in Theorem~\ref{th:prop_chaos_2} with our choice of $\nl{\bb}$. Consider the particle system $(\sPSsM_t)_{t \ge 0}$ as in display~\eqref{def:systems}. Then
\begin{align*}
\lim_{N \to \infty} \sup_{t \ge 0} 
\Was \big( \Law(\sPSsM_{t,1}), \Law(\self_t) \big) = 0.
\end{align*}
\end{theorem}
\begin{figure}[ht!]
\centering
\includegraphics[scale=0.5]{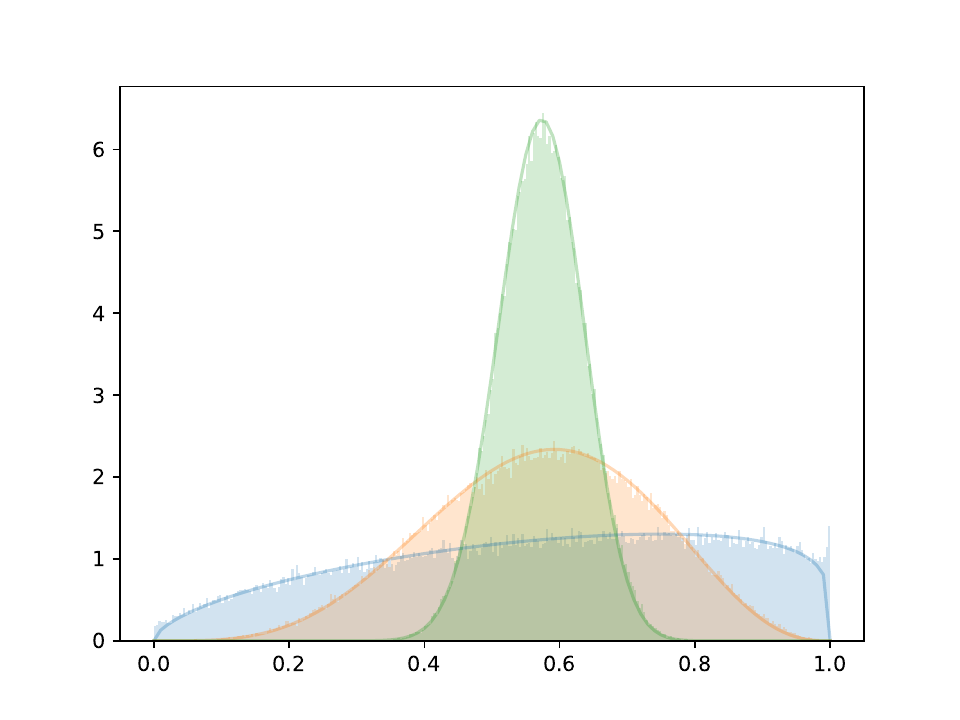}
\caption{Invariant densities for the self-stabilizing MVWF model ($\theta_0=0.8$, $\theta_1=0.6$) for several values of $\gamma$: $\gamma=0$ (light blue), $\gamma=3$ (pink), and $\gamma=30$ (green). 
Darker curves show the theoretical invariant densities. 
Histograms are based on $10^5$ simulations of the rescaled Markovian particle system with scale parameter $N=600$, started from i.i.d.\ uniform initial conditions and sampled at time $T=10$.}
\label{fig:fig2}
\end{figure}
 
\section{Ergodic properties of the Wright--Fisher kernels}\label{ss3.1}

In this section, we establish ergodic properties of the kernels $P_{\delta,p}$ for $\delta, p \in [0,1]$. In particular, we prove Theorem~\ref{cor:unique}, which addresses the case $\delta \in (0,1]$. Note that the kernel $P_{0,p}$ is the identity for every $p \in [0,1]$. A more interesting case is $\delta = 1$, studied in \cite{videla2023}. There it was shown that for $p \in (0,1)$, if $\beta_p$ denotes the law of a $\betadist(p,1-p)$ random variable, then $\beta_p P_{1,p} = \beta_p$, and for every $\nu \in \mathcal{P}$, the sequence $\nu P_{1,p}^n$ converges in total variation to $\beta_p$ as $n \to \infty$.

We next lay the groundwork for proving Proposition~\ref{cor:unique}. The results presented in the next subsection will be useful all along the paper.
\subsection{Updating functions}
As seen in Section \ref{sec:overview-results}, the transitions of the processes we will consider in this article are based on the Wright--Fisher kernels $(P_{\delta,p})_{\delta,p\in[0,1]}$.  
At various stages, we will compare their transitions for different initial values and/or parameters. The common element in all instances will be a coupling argument based on the \emph{updating functions} $\upx_{p,q,\delta}:[0,1]^4\to[0,1]$ and $\upy_{q,\delta}:[0,1]^3\to[0,1]$, $0\leq p\leq q\leq 1$, defined via 
  \begin{align}
    \upx_{p,q,\delta}(x,u,v,w) & \coloneqq
    \begin{cases*}
      x (1-\delta u), & \text{if $v \leq 1-q$ or $v>1-q$ and $w>p/q$,} \\
      x+ \delta u (1-x),        & \text{if $v>1-q$ and $w\leq p/q$,}
    \end{cases*} \label{upd1}\\
    \upy_{q,\delta}(y,u,v)&  \coloneqq
    \begin{cases*}
      y (1-\delta u), & \text{if $v \leq 1-q$,} \\
      y+ \delta u (1-y),       & \text{if $v>1-q$}.
    \end{cases*} \label{upd2}
  \end{align} 
 The following result summarizes some basic properties of these updating functions. 
\begin{lemma}\label{uple}
The families of updating functions $\{\upx_{p,q,\delta}:0\leq p\leq q\leq 1,0<\delta<1\}$ and $\{\upy_{q,\delta}:0\leq q\leq 1,0<\delta<1\}$ satisfy the following properties:
\begin{enumerate}
    \item For any $0\leq p\leq q\leq 1$, $0<\delta<1$, $u,v,w\in[0,1]$, and $0\leq x\leq y\leq 1$,  we have 
    \begin{equation}\label{eq:stochastic-orders-updates}
       \upx_{p,q,\delta}(x,u,v,w)\leq \upy_{q,\delta}(y,u,v). 
    \end{equation}
    \item For any $0\leq p\leq q\leq 1$, $0<\delta<1$, $u,v,w\in[0,1]$, and $0\leq x, y\leq 1$,  we have 
    \begin{equation}\label{eq:absolute-value-updates-difference}
        \begin{aligned}
            |\upx_{p,q,\delta}(x,u,v,w)- \upy_{q,\delta}(y,u,v)| 
            &\leq |x-y|(1-\delta u) +  \mathbf{1}_{\{v> 1-q\}} \mathbf{1}_{\{w> p/q\}}\delta u.
        \end{aligned}
    \end{equation}
    \item Let $U, U_1,U_2\sim\unif$ be independent. Then, for any $0\leq p\leq q\leq 1$, $0<\delta<1$, $u,v,w\in[0,1]$, and $0\leq x, y\leq 1$,  we have 
    \begin{equation}\label{eq:conditional-expected-value-difference-updates}
        \begin{aligned}
            \E\left[|\upx_{p,q,\delta}(x,U,U_1,U_2)- \upy_{q,\delta}(y,U,U_1)| \right]
                &\leq |x-y|\bigg(1-\frac{\delta}{2} \bigg) +  (q-p)\frac{\delta}{2}.
        \end{aligned}
    \end{equation}
    Moreover, if $X\sim \mu \in \mathcal{P}$ and $Y\sim \nu \in \mathcal{P}$, and $U, U_1, U_2$ are independent uniform random variables, also independent from $X$ and $Y$, we have: 
    \begin{equation}\label{eq:couplig-update-functions}
         \begin{aligned}
        X' &=\upx_{p,q,\delta}(X,U,U_1,U_2) \sim \mu P_{\delta,p},\qquad
        Y' =\upy_{q,\delta}(Y,U,U_1)\sim \nu P_{\delta,q}
    \end{aligned}
    \end{equation}
\end{enumerate}
\end{lemma}
We will use Lemma \ref{uple} at various stages of the paper to couple processes constructed on the basis of the Wright--Fisher kernels (see Remarks \ref{stcoupling}, \ref{nlstcoupling} and \ref{app:particle-system-coupling}) . 
\begin{proof}
 Note that the update functions can be written as
 \begin{align*}
  \upx_{p,q,\delta}(x,u,v,w)
  &= x(1-\delta u)+\mathbf{1}_{\{v>1-q\}}\mathbf{1}_{\{w\leq p/q\}}\delta u \\
 \upy_{q,\delta}(y,u,v)
 &= y(1-\delta u)  +  \mathbf{1}_{\{v> 1-q\}}\delta u,
 \end{align*}
and hence,
 \begin{align*}
            \upx_{p,q,\delta}(x,u,v,w)- \upy_{q,\delta}(y,u,v) 
            &= (x-y)(1-\delta u)  -\mathbf{1}_{\{v> 1-q\}}\mathbf{1}_{\{w> p/q\}}\delta u.
        \end{align*}
Inequalities \eqref{eq:stochastic-orders-updates} and \eqref{eq:absolute-value-updates-difference} follow directly. Taking expectations in \eqref{eq:absolute-value-updates-difference} yields \eqref{eq:conditional-expected-value-difference-updates}. Finally, for any measurable and bounded function $h:[0,1]\to[0,1]$, we have
 \begin{align*}
     \E[h(X')]  &= \E\bigg[\E\Big[h\big( (\mathbf{1}_{\{U_1\leq 1-q\}}+ \mathbf{1}_{\{U_1> 1-q,\, U_2> p/q\}})X(1-\delta U)\\
     &\qquad\qquad\qquad\qquad\qquad\qquad\qquad\quad+  \mathbf{1}_{\{U_1> 1-q,\, U_2\leq p/q\}}(X+\delta U(1-X) )\big) \mid X\Big]\bigg]\\
                &= (1-q+q-p)\E\Big[\E\big[h(X(1-\delta U)\big)\mid X\big]+ q\frac{p}{q}\E\big[h\big(X+ \delta U(1-X)\big)\mid X\big]\Big]\\
                &= \E\left[P_{\delta,p}h(X)\right]=\mu P_{\delta,p}h.
 \end{align*}
 An analogous computation for $Y'$ shows $\E[h(Y')] = P_{\delta,q}h$, achieving the proof. 
\end{proof}

    To simplify the notation we will omit the subscript $\delta$ from the update functions $\upx_{p,q,\delta}$ and $\upy_{q,\delta}$ if the dependency is clear from the context.
\subsection{Proof of Theorem \ref{cor:unique} and properties of the invariant measures}
The proof of Theorem~\ref{cor:unique}, which establishes existence, uniqueness, and convergence to the stationary distribution $\beta_{\delta,p}$ of the kernel $P_{\delta,p}$, will be based on the following result.
\begin{proposition} \label{lemma:contraction}
The following properties hold.
	\begin{enumerate}
		\item For every $\mu, \nu$ in $\mathcal{P}$, $\p,q$ in $[0,1]$ and $\delta \in (0,1]$:
		\begin{align*}
			\Was (\mu P_{\delta, \p}, \nu P_{\delta, q}) \le \left (1-\dfrac{\delta}{2}\right ) \Was (\mu, \nu) + \dfrac{\delta}{2}\vert \p-q\vert.
		\end{align*}
		\item For every $\delta,p \in [0,1]$, the map $\mu\mapsto\mu P_{\delta, \p}$ is increasing with respect to the stochastic order $\domby$.
	\end{enumerate}
\end{proposition}
\begin{proof}
	Let $\mu, \nu \in \mathcal{P}$ and assume without loss of generality that $0 \le \p \le q$. Let $(X_0, Y_0)$ be an arbitrary coupling of $\mu$ and $\nu$ defined on a probability space $(\Omega_0, \FF_0, \mathbb{P}_0)$. Consider independent $U, U_1, U_2\sim\unif$ on a different probability space $(\Omega_1, \FF_1, \mathcal{P}_1)$. Consider the probability space $(\Omega, \FF, \mathbb{P})\coloneqq(\Omega_0 \times \Omega_1, \FF_0 \otimes \FF_1, \mathbb{P}_0 \otimes \mathbb{P}_1)$. We denote by $\E_0, \E_1$ and $\E$ the expectation operators with respect to $\Pb_0$, $\Pb_1$ and $\Pb$, respectively. Define $(X_1, Y_1)$ via
  \begin{equation*}
    X_1 =\upx_{p,q}(X_0,U,U_1,U_2)\quad\text{and}\quad Y_1=\upy_q(Y_0,U,U_1).
  \end{equation*}
From \eqref{eq:couplig-update-functions}, $X_1 \sim \mu P_{\delta, \p}$ and $Y_1 \sim \nu P_{\delta, q}$. Hence applying the bound \eqref{eq:conditional-expected-value-difference-updates} we obtain:
	\begin{align*}
	    \E [\vert X_1-Y_1\vert ] &\leq  
            \bigg (1- \frac{\delta}{2}\bigg )\E_0 [\vert X_0-Y_0\vert] + \left(q-p\right)\frac{\delta}{2}.
	\end{align*}
The first part follows by choosing a coupling $(X_0, Y_0)$ such that $\E_0 (\vert X_0-Y_0\vert)=\Was(\mu, \nu)$.
	
Notice that if $\mu \domby \nu$, one may choose $(X_0, Y_0)$ such that $X_0 \leq Y_0$ almost surely. The second assertion then follows directly from inequality~\eqref{eq:stochastic-orders-updates} in Lemma~\ref{uple}. \end{proof}
We now use the previous result to prove Theorem~\ref{cor:unique}.
\begin{proof}[Proof of Theorem \ref{cor:unique}]
According to the first point of Proposition~\ref{lemma:contraction}, the map $\mu \mapsto \mu P_{\delta,\p}$ is a strict $\Was$-contraction. Since $(\mathcal{P}, \Was)$ is a Polish space, Banach’s fixed point theorem ensures that this map admits a unique fixed point, which yields the existence and uniqueness of the stationary distribution. Applying again the first part of Proposition~\ref{lemma:contraction} with $q=p$, $\nu=\beta_{\delta,p}$, and $\mu (P_{\delta,p})^{n-1}$ in place of $\mu$, we obtain
\begin{align*}
\Was \bigl(\mu (P_{\delta,\p})^n, \beta_{\delta,p}\bigr) 
    \le \left(1-\frac{\delta}{2}\right) 
        \Was \bigl(\mu (P_{\delta,\p})^{n-1}, \beta_{\delta,p}\bigr).
\end{align*}
The result follows by iterating this inequality.
\end{proof}
We establish in Appendix~\ref{subsection:auxiliary_results} additional properties of the invariant distributions $\beta_{\delta,p}$, for $\delta\in(0,1]$ and $p\in[0,1]$, including absolute continuity with respect to the Lebesgue measure, recursive relations for their moments, and integral equations for their densities.
\section{The single host level: duality and scaling limits} 
This section is devoted to the proof of Theorem~\ref{th:moment_duality_II} and Theorem~\ref{thm:connection-WF}, as well as to the presentation of additional results of independent interest that shed light on the dynamics of Markov chains driven by Wright--Fisher kernels, in particular on their long-term behavior.

\label{section:host_level}
\subsection{Some aspects of the dynamics}
Fix $\delta\in(0,1)$ and a measurable function $\pf: [0,1] \mapsto [0,1]$. Consider a two-types ($0$ and $1$) infinite population evolving in discrete generations. At each generation a fraction $\delta U$, with $U$ being uniformly distributed in $[0,1]$, of the population dies and is replaced by the offspring of an individual. The latter is chosen uniformly at random from the type-$0$ (resp. type-$1$) sub-population with probability $\pf(x)$ (resp. $1-\pf(x)$), where $x$ denotes the current fraction of type $0$ individuals. Let $Z^{(\delta,\pf)}_n$ denote the fraction of type-$0$ individuals in the population at generation $n$. The process $Z^{(\delta,\pf)}\coloneqq(Z^{(\delta,\pf)}_n)_{n \ge 0}$ is a Markov chain with transition  kernel $(x, A) \mapsto P_{\delta, \pf(x)} (x, A)$. Since $\delta$ remains fix until Section \ref{sec:connection-WF}, we will write until then $Z^{(\pf)}$ instead of $Z^{(\delta,\pf)}$. In contrast, in Section \ref{sec:connection-WF} it will be important to make the dependency in $\delta$ explicit.

An important example arises when $\pf$ is the identity function, 
which we refer to as the \emph{neutral case}. In contrast, a general function $\pf$ introduces, at each generation, 
a bias toward one of the two types, depending on the current type composition. 
The next proposition characterizes the neutral case together with two extreme scenarios 
in which one of the types is always favored. In the latter situation, we refer to the advantaged type as the \emph{fit} type.

\begin{proposition}\label{lemma:smartingales}
Assume that $\pf(x) \le x$ (resp. $\pf(x)\ge x$) for all $x\in[0,1]$. Then the Markov chain $Z^{(\pf)}$ driven by $P_{\pf}$ is a bounded supermartingale (resp. a bounded submartingale). In particular, if $\pf$ is the identity function, $Z^{(\pf)}$ is a martingale.
\end{proposition}

\begin{proof}
	By construction $Z_n^{(\pf)}\in[0,1]$ for all $n\geq 0$ and
	\begin{align*}
\E\!\left[ Z_{n+1}^{(\pf)} \mid Z_n^{(\pf)} \right]
&= \pf\!\left(Z_n^{(\pf)}\right)
    \Bigg( Z_n^{(\pf)} + \big(1 - Z_n^{(\pf)}\big)\frac{\delta}{2} \Bigg)
   + \big(1 - \pf(Z_n^{(\pf)})\big)
     \Bigg( Z_n^{(\pf)} \Big(1 - \frac{\delta}{2}\Big) \Bigg)\\
&= \pf\!\left(Z_n^{(\pf)}\right)
    \Bigg( Z_n^{(\pf)} + (1 - Z_n^{(\pf)})\frac{\delta}{2}
          - Z_n^{(\pf)} + Z_n^{(\pf)}\frac{\delta}{2} \Bigg)
   + Z_n^{(\pf)} \Big(1 - \frac{\delta}{2}\Big)\\
&= \big(\pf(Z_n^{(\pf)}) - Z_n^{(\pf)}\big)\,\frac{\delta}{2}
   + Z_n^{(\pf)}.
\end{align*}
and the result follows.
\end{proof}
In the setting of Proposition~\ref{lemma:smartingales}, we conclude that $Z_n^{(\pf)}$ converges almost surely, as $n \to \infty$, to a random variable $Z_\infty$. The next result provides a necessary condition for this limit to take values only in $\{0,1\}$, i.e. for type~$0$ to eventually fixate or become extinct.
\begin{proposition}\label{lemma:aslimit}
Let $\pf$ be a continuous function such that $\pf(x)=0 \Leftrightarrow x=0$ and $\pf(x)=1 \Leftrightarrow x=1$. Assume also that for some initial condition $x \in [0,1]$, the Markov chain $Z^{(\pf)}$ driven by $P_{\pf}$ has an a.s. limit $Z_\infty$. Then $Z_\infty \in \{0,1\}$ a.s. 
\end{proposition}
\begin{proof}
    For every $n \ge 2$, let $I_n\coloneqq[1/n, 1-1/n]$. Clearly,
\begin{align*}
\{Z_\infty \in (0,1)\} 
= \bigcup_{n \ge 2} \{Z_\infty \in I_n\} 
\subseteq \bigcup_{n \ge 2} \bigcup_{m \ge 1} \bigcap_{k \ge m} \{Z_k^{(\pf)} \in I_n\}.
\end{align*}
Define $A^{(n)}_k \coloneqq \{Z_k^{(\pf)} \in I_n\}$ for $n\geq 2$.
It suffices to show that for each $n\geq 2$, we have $\prob [\liminf_k A^{(n)}_k]=0$,
to achieve the proof.

We claim that, whenever $Z^{(\pf)}_0 \in I_n$, the chain $Z^{(\pf)}$ has probability at least $r_n$, for some $r_n>0$, of exiting $I_n$ through its left endpoint within a fix number of steps (possibly depending on $n$). To see this, let $L^{(n)}$ be the minimal integer $L$ such that
\[
\Big(1 - \frac{\delta}{2}\Big)^{L}\Big(1 - \frac{1}{n}\Big) < \frac{1}{n}.
\]
Further, define $q_n \coloneqq 1 - \sup\{\pf(x) : x \in I_n\}$. Note that $q_n > 0$ by the continuity of $\pf$.

Note that, if $(U_i)_{i=1}^{L^{(n)}}$ and $(V_i)_{i=1}^{L^{(n)}}$ are independent families of independent uniform random variables on $[0,1]$, we can construct $Z^{(\pf)}$ iteratively by setting
\begin{align*}
Z^{(\pf)}_{i+1}=\upy_{\pf(Z_i)}(Z^{(\pf)}_i,U_{i+1},V_{i+1}),\qquad i<L^{(n)}.
\end{align*}
Conditional on $Z^{(\pf)}_{0} \in I_n$, on the event $E_n \coloneqq \{\, U_i \ge 1/2,\; V_i > q_n,\; i = 1,\ldots, L^{(n)} \,\}$, the chain exits the interval $I_n$ via its left endpoint at or before time $L^{(n)}$. Moreover, $\mathbb{P}(E_n) = \big(2^{-1}(1-q_n)\big)^{L^{(n)}}\eqqcolon r_n > 0.$
Hence, by the Markov property, for any $m \ge 0$,
\begin{align} \label{eq:bound_markov}
\prob (Z^{(\pf)}_{m+i} \notin I_n \text{ for some } i=1, 2, \ldots, L^{(n)}\vert Z^{(\pf)}_m\in I_n) \ge r_n.  
\end{align}
Now, for every $m \ge 1$:
\begin{align*}
\bigcap_{k \ge m}\{Z_k^{(\pf)} \in I_n\} &= \bigcap_{k \ge 0} \bigcap_{i=0}^{L^{(n)}-1}\{Z_{m+kL^{(n)}+i}^{(\pf)} \in I_n\}.
\end{align*}
Set
\begin{align*}
B^{(n)}_k\coloneqq \bigcap_{i=0}^{L^{(n)}-1}\{Z_{m+kL^{(n)}+i}^{(\pf)} \in I_n\}.
\end{align*}
Using the Markov property, we obtain, for every $K \ge 1$, that
\begin{align*}
\prob \bigg [\bigcap_{k = 0}^{K} B^{(n)}_k\bigg ] &\le \prob (B^{(n)}_0)\prod_{k=1}^{K} \prob (B^{(n)}_k \vert B^{(n)}_{k-1})\le (1-r_n)^K,
\end{align*}
where in the last line we used \eqref{eq:bound_markov}. Letting $K\to\infty$, we conclude that $\prob [\liminf_k A^{(n)}_k]=0$ for every $n$, ending the proof.
\end{proof}
\begin{remark}
Note that if $\pf(x) \le x$ for all $x \in [0,1]$ or $\pf(x) \ge x$ for all $x \in [0,1]$, Proposition~\ref{lemma:smartingales} implies that $Z^{(\pf)}$ is either a bounded supermartingale or a bounded submartingale. In both cases, we may conclude that $Z_n^{(\pf)}$ converges almost surely and in $L^1$ to a random variable $Z_\infty$. If, in addition, $\pf$ is continuous and satisfies $\pf(x) = 0 \Leftrightarrow x = 0$ and $\pf(x) = 1 \Leftrightarrow x = 1$, then Proposition~\ref{lemma:aslimit} implies that $Z_\infty$ is a Bernoulli random variable. 
\end{remark}
We conclude this section with a result that allows us to compare Wright--Fisher chains associated with different functions $\pf$. Before stating it precisely, we present in the following remark a suitable coupling for this setting.

\begin{remark}[The standard coupling]\label{stcoupling}
Let $\pf,\qf:[0,1]\to[0,1]$ be measurable functions. Lemma \ref{uple} provides a natural way to couple the chains $Z^{(\pf)}$ and $Z^{(\qf)}$, provided that their initial values $Z^{(\pf)}_0$ and $Z^{(\qf)}_0$ are already coupled. To see this, assume that we have already coupled the chains up to generation $n$ and let $U,U_1,U_2\sim\unif$ be independent between them, and independent from the two processes up to generation $n$. Set $p\coloneqq\pf(Z^{(\pf)}_n)$ and $q\coloneqq\pf(Z^{(\qf)}_n)$ and we define
$$Z^{(\pf)}_{n+1}=\upx_{p,q}(Z^{(\pf)}_n,U,U_1,U_2)\quad\textrm{and}\quad Z^{(\qf)}_{n+1}=\upy_{q}(Z^{(\qf)}_n,U,U_1),$$
if $p\leq q$, or
$$Z^{(\qf)}_{n+1}=\upx_{q,p}(Z^{(\qf)}_n,U,U_1,U_2)\quad\textrm{and}\quad Z^{(\pf)}_{n+1}=\upy_{p}(Z^{(\pf)}_n,U,U_1),$$
otherwise. Thanks to part 3 of Lemma \ref{uple}, the chains $Z^{(\pf)}$ and $Z^{(\qf)}$ constructed by this procedure have the desired distributions and satisfy, for all $n\geq 0$, 
	\begin{align} \label{eq:pre_bound_00}
		\E[\vert Z^{(\pf)}_{n+1}-Z^{(\qf)}_{n+1}] & \le \Bigg (1-\dfrac{\delta}{2}\Bigg) \E[\vert Z^{(\pf)}_{n}-Z^{(\qf)}_{n} \vert ]+\dfrac{\delta}{2}\E[\vert \pf(Z^{(\pf)}_{n})-\qf( Z^{(\qf)}_{n})\vert].
	\end{align}
\end{remark}
We can now state the announced comparison result, which will prove useful in Section~\ref{sec:appli}.
\begin{lemma}\label{lemma:comparison_theorem}
Let $\pf, \qf: [0,1]\to[0,1]$ be two measurable functions. Assume that $\qf\geq\pf$ (pointwise) and that either $\pf$ or $\qf$ is monotone increasing. Assume also that  $Z^{(\pf)}_0$ is stochastically dominated by $Z^{(\qf)}_0$. Then $Z^{(\pf)}$ is stochastically dominated (at every generation) by $Z^{(\qf)}$.
\end{lemma}

\begin{proof}
To prove the result it suffices to construct a coupling of the processes $Z^{(\pf)}$ and $Z^{(\qf)}$ such that $Z^{(\pf)}_n\leq Z^{(\qf)}_n$ for all $n\geq 0$. By assumption, we know we can couple the initial values so that  $Z^{(\pf)}_0\leq Z^{(\qf)}_0$. 
Moreover, note that if $\pf$ is increasing and $x\leq y$, then $\pf (x) \le \pf (y) \le \qf(y)$. Similarly, if $\qf$ is increasing and $x\leq y$, then $\pf (x) \le \qf (x) \le \qf(y)$. In both cases, if $x\leq y$, then $\pf (x) \leq \qf(y)$.
Using this and the first part of Lemma \ref{uple}, we see that the standard coupling
starting from $Z^{(\pf)}_0\leq Z^{(\qf)}_0$ satisfy the desired property.
\end{proof}
\subsection{Moment duality in fittest-type wins kernels with mutation}\label{ss:duality_fittest_type_wins}
In this section we focus on a particular case of the models introduced in the previous section, namely when $\pf$ is of the form
$$\pf^{\downarrow}_\delta(x)=x+\frac{2\delta}{3}\left(-x(1-x)\sum_{i=0}^\infty\sigma_i x^i+\theta_0(1-x)-\theta_1 x\right),$$
with $\delta>0$, $\theta_0,\theta_1\geq 0$ and $(\sigma_k)_{k\geq 0}$ satisfying the following assumption.
\begin{assumption}\label{eq:non_increasing}
The coefficients $(\sigma_k)_{k \ge 0}$ are non-negative, non-increasing and their sum is finite. Moreover, $2\delta(\sigma_0+\theta_0+\theta_1)/3<1$.
\end{assumption}
From now onwards, we write $\pf^{\downarrow}$ instead of $\pf_\delta^{\downarrow}$; making the dependency on $\delta$ explicit will be only relevant in Remark \ref{ftw-rem}. Note that $\pf^{\downarrow}$ is the identity when $\sigma_0=0=\theta_0=\theta_1$. Under Assumption \ref{eq:non_increasing}, the function $\pf^{\downarrow}$ satisfies $\pf^{\downarrow}(x)\in[0,1]$ for all $x\in[0,1]$. The first term inside the brackets can be interpreted as frequency dependent selection and favors always type-$1$ individuals. It coincides with the drift term in the \emph{fittest-type wins} Wright--Fisher processes introduced in \cite{GS18} and further studied in \cite{BEH23,CHS19}. The second term inside the brackets can be interpreted as mutation; $\theta_0$ and $\theta_1$ can be seen as the relative strength of mutations to type $0$ and to type $1$, respectively. For this reason, we refer to the kernel $P_{\pf^{\downarrow}}$ as a \emph{fittest type wins kernel with mutation}. In Section \ref{sec:connection-WF} we will see a more precise connection to the fittest-type wins Wright--Fisher processes studied in \cite{GS18}, but let us already mention that those models are tailored to admit a moment dual. 

Note also that under Assumption \ref{eq:non_increasing}, if $\sigma_0>0$, the vector $(\rho_i)_{i\geq 1}$ defined via
$$\rho_i\coloneqq\frac{\sigma_{i-1}-\sigma_i}{\sigma_0},\qquad i\in\N,$$
is a probability measure on $\N$. Moreover, letting $\varphi:[0,1]\to[0,1]$ be the probability generating function of $(\rho_i)_{i\geq 1}$, i.e.
$\varphi(x)\coloneqq \sum_{i=1}^\infty \rho_i x^i$,  $x\in[0,1],$
we can express the selective term in $\pf^{\downarrow}$ as
$$-x(1-x)\sum_{i=0}^\infty\sigma_i x^i=\sigma_0 x(\varphi(x)-1).$$
To avoid case distinctions, we set $\rho_i\coloneqq 0$ for all $i\geq 1$ when $\sigma_0=0$.

In population genetics, individual-based models typically allow for a graphical representation as an interacting particle system that couples both forward and backward evolution. From such a graphical representation, one can read duality relations characterizing functionals of the forward evolution, as for example type frequencies, in terms of appropriate functionals of the backward evolution, as the number of ancestors of a given sample (see \cite{jansen2014} for a nice survey on the various notions of duality associated with Markov processes). One of the simplest and most useful examples is moment duality, which unfortunately only holds under restrictive conditions, such as neutrality. In our case, we don't have a particle picture at hand for our process $Z^\downarrow$, but nevertheless, as stated in Theorem \ref{th:moment_duality_II}, we have a moment dual. We would like to point out that this is rather an exception in comparison with other classical discrete-time models (e.g. Wright--Fisher or Cannings models), where the more involved notion of \emph{factorial or sampling duality} arises in a natural way (see \cite{Cannings74,Glad77,Glad78,Martin99, BEH23}).

We define  
\begin{align}\label{eq:lambdas}
\lambda_{m,k}\coloneqq \dfrac{1}{\delta} \int_{0}^{\delta} x^{k}(1-x)^{m-k} \dd x,\quad m\in\N,\, k\in\N_0,\, 0\leq k\leq m.
\end{align}
Notice that $\sum_{k=0}^m \binom{m}{k} \lambda_{m,k} = 1$, hence $\{\binom{m}{k} \lambda_{m,k}\}$ defines a probability over $\{0,\ldots,m\}$. Furthermore, set $\delta_0\coloneqq\frac{2\delta}{3}$ and $\theta\coloneqq\theta_0+\theta_1$ and define the Markov chain $\wM \coloneqq M^{(\pf^\downarrow)}\coloneqq(\wM_n)_{n\geq 0}$ on  $\overline {\mathbb{N}}_0\coloneqq \N_0 \cup\{ \partial\}$, where $\partial$ denotes a cemetery state, and transitions 
\begin{equation}\label{ratedual}
\begin{split}
    m&\to m \qquad\qquad\text{with prob.}\quad \lambda_{m,0}+(1-\delta_0(\theta+\sigma_0))m\lambda_{m,1}\\
    &\qquad \qquad \qquad \qquad \qquad \qquad \quad +\delta_0\sigma_0\sum_{k=2}^m\binom{m}{k}\lambda_{m,k}\rho_{k-1},\\
    m& \to m-j \qquad\text{with prob.}\quad\delta_0\theta_0\binom{m}{j}\lambda_{m,j}+(1-\delta_0(\theta+\sigma_0)) \binom{m}{j+1}\lambda_{m,j+1},\\
    &\phantom{\quad\text{with prob.}}\qquad\qquad\qquad+\delta_0\sigma_0\sum_{k=j+2}^m\binom{m}{k}\lambda_{m,k}\rho_{k-j-1},\, \quad 1\leq j\leq m, \\
    m&\to m+j \qquad\text{with prob.}\quad\delta_0\sigma_0 \sum_{k=1}^m \binom{m}{k}\lambda_{m,k}\rho_{j+k-1},\, \qquad \quad j\geq 1,\\
    m& \to \partial \qquad\qquad\text{with prob.}\quad \delta_0\theta_1(1-\lambda_{m,0}),
\end{split}
\end{equation}
with the usual assumptions that an empty sum equals $0$ and that $\binom{i}{j}=0$ if $j>i$. Note that when $\theta_0=\theta_1=0$, the state space can be reduced to $\N$.

Now we have all the ingredients to prove Theorem \ref{th:moment_duality_II}.
\begin{proof}[Proof of Theorem \ref{th:moment_duality_II}]
The result is clearly true for $n=0$. In the discrete time setting, proving the inductive step, is equivalent to proving the result for $n=1$. By construction, we have 
\begin{align}
\E_z [(\wZ_1)^{m}]&= (1-\pf^{\downarrow}(z)) \E [z^m(1-\delta U)^m] + \pf^{\downarrow}(z )\E\bigg [ (z(1-\delta U) + \delta U)^m \bigg ] \nonumber\\
&=(1-\pf^{\downarrow}(z))z^{m}\lambda_{m,0} + \pf^{\downarrow}(z) \sum_{j=0}^{m} z^{m-j} \binom{m}{j} \lambda_{m,j}.\label{zEZ}
\end{align}
Using the definition of $\pf^{\downarrow}$ yields
\begin{align*}
 1-\pf^{\downarrow}(z)&= 1-\delta_0\theta_0+ \Big(-1+\delta_0(\sigma_0+\theta)\Big)z-\delta_0\sigma_0\sum_{i=1}^\infty\rho_iz^{i+1},\\
 \pf^{\downarrow}(z) \sum_{j=0}^{m} z^{m-j} \binom{m}{j} \lambda_{m,j}&=\left(\sum_{j=1}^m\alpha_j^{-}z^{-j}+\alpha_0+\sum_{j=1}^\infty \alpha_j^{+}z^j\right)z^m,
\end{align*}
with 
\begin{align*}
   \alpha_j^{-}&=\delta_0\theta_0\binom{m}{j}\lambda_{m,j}+(1-\delta_0(\theta+\sigma_0)) \binom{m}{j+1}\lambda_{m,j+1}+\delta_0\sigma_0\sum_{k=j+2}^m\binom{m}{k}\lambda_{m,k}\rho_{k-j-1},\\
   \alpha_0&=\delta_0\theta_0\lambda_{m,0}+(1-\delta_0(\theta+\sigma_0))m\lambda_{m,1}+\delta_0\sigma_0\sum_{k=2}^m\binom{m}{k}\lambda_{m,k}\rho_{k-1},\\
   \alpha_j^{+}&=1_{\{j=1\}}(1-\delta_0(\theta+\sigma_0))\lambda_{m,0}+\delta_0\sigma_0 \left(\sum_{k=1}^m \binom{m}{k}\lambda_{m,k}\rho_{j+k-1}+1_{\{j\geq 2\}}\lambda_{m,0}\rho_{j-1}\right).
\end{align*}
Hence,
$$\E_z [(\wZ_1)^{m}]=\sum_{j=1}^m\alpha_j^{-}z^{m-j}+\widehat{\alpha}_0z^m+\sum_{j=1}^\infty \widehat{\alpha}_j^{+}z^{m+j},$$
with
\begin{align*}
   \widehat{\alpha}_0&=(1-\delta_0\theta_0)\lambda_{m,0}+\alpha_0=\lambda_{m,0}+(1-\delta_0(\theta+\sigma_0))m\lambda_{m,1}+\delta_0\sigma_0\sum_{k=2}^m\binom{m}{k}\lambda_{m,k}\rho_{k-1},\\
   \widehat{\alpha}_j^{+}&=\alpha_j^{-}-\lambda_{m,0}(1_{\{j=1\}}(1-\delta_0(\theta+\sigma_0))+1_{\{j\geq 2\}}\delta_0\sigma_0\rho_{j-1})=\delta_0\sigma_0 \sum_{k=1}^m \binom{m}{k}\lambda_{m,k}\rho_{j+k-1}.
\end{align*}
Since $$\sum_{j=1}^m\alpha_j^{-}+\widehat{\alpha}_0+\sum_{j=1}^\infty \widehat{\alpha}_j^{+}=1-(1-\lambda_{m,0})\delta_0\theta_1,$$
the result follows.
\end{proof}

\subsection{A particle picture interpretation of the dual process}\label{ss:pp_repres}
Although we do not have a particle-system representation for the forward dynamics at hand, the process $\wM$
admits a particle  interpretation, in close analogy with classical models in population genetics. This ancestral picture is not derived from the forward dynamics; it merely serves as a heuristic device to visualize and organize the transition rates. This viewpoint will nonetheless prove valuable when establishing comparison results; see Lemma \ref{compa} and Remark \ref{compare} below.

\paragraph{Neutral case ($\sigma_0=\theta_0=\theta_1=0$).}
As in the interpretation of the Wright--Fisher kernels in
Section~\ref{sec:overview-results}, in the neutral case one may picture each
generation as replacing a fraction $\delta U$ of the population, where
$U\sim\unif$, by the offspring of a uniformly chosen individual.  
More precisely, this mechanism is the discrete-time analog of the
$\Lambda$--Wright--Fisher process described at the beginning of Section \ref{sec:overview-results} with
\[
\Lambda(\dd u)=\delta^{-1}\,\mathbf{1}_{[0,\delta]}(u)\,u^{2}\,\dd u,
\]
whose reproduction events occur in continuous time at Poisson points with
intensity $\dd t\,\dd v\,\Lambda(\dd u)\,u^{-2}$ (see \cite{Bertoin2003,
Bertoin2006}).

Further, in this neutral setting, the transitions of the dual $\wM$ simplify to
\begin{align*}
m &\to m
&& \text{with probability }\ \lambda_{m,0}+m\,\lambda_{m,1},\\
m &\to m-k+1
&& \text{with probability }\ \binom{m}{k}\,\lambda_{m,k},
\qquad 1<k\le m.
\end{align*}

In analogy with the forward picture, it is convenient to think of $\wM_n$ as
the number of distinct ancestors, $n$ generations in the past, of a present-day
sample of size $\wM_0=m$. At each backward step, every ancestor is independently marked with probability \(u\delta\), where \(u\) is a realization of an independent uniform random variable on \([0,1]\). All marked ancestors are then merged, meaning they share a common parent in the previous generation; this parent is chosen uniformly at random from the full pool of individuals present at that time.
For readers familiar with exchangeable coalescents, this leads to a discrete-time
analogue of a $\Lambda$–coalescent with the above measure $\Lambda$.
$\Lambda$–coalescents were introduced independently in
\cite{DK99, Pitman1999, Sa99} (see \cite{Ber09} for a survey) and have since
been extensively studied and used as prototypical models for genealogies in neutral
populations. Theorem~\ref{th:moment_duality_II} can thus be interpreted as a discrete-time analogue 
of the classical moment duality between $\Lambda$–coalescents and $\Lambda$–Wright--Fisher processes.

\paragraph{Fittest type wins with mutation.} Equipped with the intuition from the neutral case, we now turn to the process $\wM$ and the interpretation of the duality relation in the general setting.
Suppose that at time~$n$ we sample $m$ individuals and wish to compute the 
probability that all of them are of type~$0$. Conditioned on $Z_n^\downarrow$, 
this probability is simply $(Z_n^\downarrow)^m$, and taking expectations 
yields the right-hand side of the duality relation~\eqref{momdi2}. Ideally, 
we would like to recover the same quantity through an ancestral process whose 
line--counting process is~$\wM$. The idea is to trace backward in time the 
individuals, \emph{potential ancestors}, whose types must be specified 
in order to determine whether the sampled individuals are all of type~$0$. 
We would like the ancestral system to start with the $m$ sampled lines, and 
for the sample to be entirely of type~$0$ precisely when, at every backward 
time~$k$, all $\wM_k$ ancestral lines present at that time are of type~$0$. 
Consequently, if we sample at forward time~$n$ and follow the genealogy back 
$n$ generations, then, conditional on $\wM_n$, the probability that the sample 
contains only type~$0$ individuals is $z^{\wM_n}$. Taking expectations then yields the right-hand side of \eqref{momdi2}, and hence the duality relation.

Since we do not have a particle picture for the forward dynamics, we instead 
construct a branching--coalescing system that matches the rates of the dual 
process, we describe how types propagate along its lines, and we provide a 
genealogical interpretation of its various transitions. It is important to 
stress, however, that this construction is obtained from the duality relation 
rather than from a forward particle model that is naturally compatible with 
the backward process. In particular, it does not give a new proof of the 
duality; it only offers a way to visualize the backward process.

As usual in genealogical constructions, we work in an untyped manner: we trace
ancestral lineages without assigning them types, except when a mutation occurs,
at which point the type of the affected lineage is revealed. As in the neutral case, at each generation we independently mark each potential ancestor with probability $\delta u$, where $u$ is the realization of an independent $\mathrm{Unif}[0,1]$ random variable. All marked individuals share
the same parent in the previous generation, and we assume they inherit that
parent's type; hence, their lineages can be merged. This coalescence event  may be
immediately followed (in the backward direction) by one of the
following events.

\begin{itemize}

\item[\textbf{$\bullet$}] \emph{Neutral events.}
With probability $1 - \delta_0(\theta + \sigma_0)$, the event is neutral and is
not coupled to any further update.

\item[\textbf{$\bullet$}] \emph{Deleterious mutation.}
With probability $\delta_0 \theta_0$, the merger event is followed by a mutation
to type~$0$ (although seen backward in time, its type effect is forward in time).
This mutation may occur in any of the lineages present after the coalescence event. The affected lineage is removed from further consideration, since all of its descendants will necessarily be of
type~$0$, and its lineage no longer needs to be traced into the past.

\item[\textbf{$\bullet$}] \emph{Beneficial mutation.}
With probability $\delta_0 \theta_1$, the merger of the marked individuals is
followed by a mutation to type~$1$. This mutation may occur in any of the lineages present after the coalescence event. Thus, at least one potential ancestor of type~$1$ appears, which will
propagate forward in time and produce at least one sampled individual of
type~$1$. Thus, the event ``all sampled individuals are of type~$0$'' becomes impossible,
and we terminate the genealogical exploration by sending the process to the
cemetery state~$\partial$.

\item[\textbf{$\bullet$}] \emph{Selective events.}
With probability $\delta_0 \sigma_0 \rho_j$, the merger is followed by an event in which one of the lineages present after the coalescence event becomes the child
of a group of $j+1$ \emph{potential parents} from the previous generation.
The true parent is selected uniformly among the potential parents of type~$1$,
if any are present, and uniformly among all $j+1$ potential parents otherwise.
Since type information (other than mutation) is not yet known, we must trace
back the lineages of all $j+1$ potential parents.
\end{itemize}
\begin{figure}[t!]\label{particles}
\centering

\begin{minipage}{0.25\textwidth}
\centering
\begin{tikzpicture}
\node[draw] at (0.5,5.5) {Neutral};
 \draw [thick, ->] (1,-1) -- (0,-1) node[midway,below] {$t$};
 \shade[ball color=white, opacity =1] (0,0) circle (0.1);
 \shade[ball color=white, opacity =1] (0,1) circle (0.1);
 \shade[ball color=white, opacity =1] (0,1.5) circle (0.1);
 \shade[ball color=white, opacity =1] (0,2) circle (0.1);
 \shade[ball color=white, opacity =1] (0,3) circle (0.1);
 \shade[ball color=white, opacity =1] (0,3.5) circle (0.1);

 \shade[ball color=white, opacity =0.2] (0.5,0) circle (0.1);
 \shade[ball color=white, opacity =0.2] (0.5,1) circle (0.1);
 \shade[ball color=white, opacity =0.2] (0.5,1.5) circle (0.1);
 \shade[ball color=white, opacity =0.2] (0.5,2) circle (0.1);
 \shade[ball color=white, opacity =0.2] (0.5,3) circle (0.1);
 \shade[ball color=white, opacity =0.2] (0.5,3.5) circle (0.1);

  \draw[]        (0.4,0)  to[out=-20,in=180] (0.1,0);
  \draw[]        (0.4,1)  to[out=-20,in=180] (0.1,1);
  \draw[]        (0.4,1.5)   to[out=-20,in=180] (0.1,1.5);
  \draw[]        (0.4,2)   to[out=-20,in=180] (0.1,2);
  \draw[]        (0.4,3)   to[out=-20,in=180] (0.1,3);
  \draw[]        (0.4,3.5)  to[out=-20,in=180] (0.1,3.5);

  \draw[]        (0.6,0)  to[out=-20,in=180] (0.9,0);
  \draw[]        (0.6,3.5)   to[out=-20,in=180] (0.9,0.5);
  \draw[]        (0.6,1)  to[out=-20,in=180] (0.9,1);
  \draw[]        (0.6,1.5)   to[out=-20,in=180] (0.9,1.5);
  \draw[]        (0.6,2)   to[out=-20,in=180] (0.9,2);
  \draw[]        (0.6,3.5)   to[out=-20,in=180] (0.9,2.5);
  \draw[]        (0.6,3)   to[out=-20,in=180] (0.9,3);
  \draw[]        (0.6,3.5)  to[out=-40,in=-80] (0.9,3.5);
  \draw[]        (0.6,3.5) to[out=-20,in=180] (0.9,4);
  
 \shade[ball color=white, opacity =1] (1,0) circle (0.1);
 \shade[ball color=black, opacity =1] (1,0.5) circle (0.1);
 \shade[ball color=white, opacity =1] (1,1) circle (0.1);
 \shade[ball color=white, opacity =1] (1,1.5) circle (0.1);
 \shade[ball color=white, opacity =1] (1,2) circle (0.1);
 \shade[ball color=black, opacity =1] (1,2.5) circle (0.1);
 \shade[ball color=white, opacity =1] (1,3) circle (0.1);
 \shade[ball color=black, opacity =1] (1,3.5) circle (0.1);
 \shade[ball color=black, opacity =1] (1,4) circle (0.1);
 
 \end{tikzpicture}
 \end{minipage}\begin{minipage}{0.25\textwidth}
\centering
\begin{tikzpicture}
\node[draw] at (0.5,5.5) {+Mut. to $0$};
 \draw [thick, ->] (1,-1) -- (0,-1) node[midway,below] {$t$};
 shade[ball color=white, opacity =1] (0,0) circle (0.1);

 \shade[ball color=white, opacity =1] (0,1) circle (0.1);
 \shade[ball color=white, opacity =1] (0,1.5) circle (0.1);

 \shade[ball color=white, opacity =1] (0,3) circle (0.1);
 \shade[ball color=white, opacity =1] (0,3.5) circle (0.1);
 \shade[ball color=white, opacity =1] (0,0) circle (0.1);

 \draw [thick, ->] (1,-1) -- (0,-1) node[midway,below] {$t$};
 \shade[ball color=white, opacity =.2] (0.5,0) circle (0.1);
 \shade[ball color=white, opacity =.2] (0.5,1) circle (0.1);
 \shade[ball color=white, opacity =.2] (0.5,1.5) circle (0.1);
 \shade[ball color=blue, opacity =.2] (0.5,2) circle (0.1);

 \shade[ball color=white, opacity =.2] (0.5,3) circle (0.1);
 \shade[ball color=white, opacity =.2] (0.5,3.5) circle (0.1);

  \draw[]        (0.1,0)  to[out=-20,in=180] (0.4,0);

  \draw[]        (0.1,1)  to[out=-20,in=180] (0.4,1);
  \draw[]        (0.1,1.5)   to[out=-20,in=180] (0.4,1.5);

  \draw[]        (0.1,3)   to[out=-20,in=180] (0.4,3);
  \draw[]        (0.1,3.5)  to[out=-20,in=180] (0.4,3.5);

  \draw[]        (0.6,0)  to[out=-20,in=180] (0.9,0);
  \draw[]        (0.6,3.5)   to[out=-20,in=180] (0.9,0.5);
  \draw[]        (0.6,1)  to[out=-20,in=180] (0.9,1);
  \draw[]        (0.6,1.5)   to[out=-20,in=180] (0.9,1.5);
  \draw[]        (0.6,2)   to[out=-20,in=180] (0.9,2);
  \draw[]        (0.6,3.5)   to[out=-20,in=180] (0.9,2.5);
  \draw[]        (0.6,3)   to[out=-20,in=180] (0.9,3);
  \draw[]        (0.6,3.5)  to[out=-20,in=180] (0.9,3.5);
  \draw[]        (0.6,3.5) to[out=-20,in=180] (0.9,4);
  
 \shade[ball color=white, opacity =1] (1,0) circle (0.1);
 \shade[ball color=black, opacity =1] (1,0.5) circle (0.1);
 \shade[ball color=white, opacity =1] (1,1) circle (0.1);
 \shade[ball color=white, opacity =1] (1,1.5) circle (0.1);
 \shade[ball color=white, opacity =1] (1,2) circle (0.1);
 \shade[ball color=black, opacity =1] (1,2.5) circle (0.1);
 \shade[ball color=white, opacity =1] (1,3) circle (0.1);
 \shade[ball color=black, opacity =1] (1,3.5) circle (0.1);
 \shade[ball color=black, opacity =1] (1,4) circle (0.1);
 \end{tikzpicture}
 \end{minipage}\begin{minipage}{0.25\textwidth}
\centering
\begin{tikzpicture}
\node[draw] at (0.5,5.5) {+Mut. to $1$};
 \draw [thick, ->] (1,-1) -- (0,-1) node[midway,below] {$t$};

 \shade[ball color=white, opacity =.2] (0.5,0) circle (0.1);
 \shade[ball color=white, opacity =.2] (0.5,1) circle (0.1);
 \shade[ball color=white, opacity =.2] (0.5,1.5) circle (0.1);
 \shade[ball color=red, opacity =.2] (0.5,2) circle (0.1);

 \shade[ball color=white, opacity =.2] (0.5,3) circle (0.1);
 \shade[ball color=white, opacity =.2] (0.5,3.5) circle (0.1);

\node[draw] at (0,2) {$\partial$};

  \draw[]        (0.6,0)  to[out=-20,in=180] (0.9,0);
  \draw[]        (0.6,3.5)   to[out=-20,in=180] (0.9,0.5);
  \draw[]        (0.6,1)  to[out=-20,in=180] (0.9,1);
  \draw[]        (0.6,1.5)   to[out=-20,in=180] (0.9,1.5);
  \draw[]        (0.6,2)   to[out=-20,in=180] (0.9,2);
  \draw[]        (0.6,3.5)   to[out=-20,in=180] (0.9,2.5);
  \draw[]        (0.6,3)   to[out=-20,in=180] (0.9,3);
  \draw[]        (0.6,3.5)  to[out=-20,in=180] (0.9,3.5);
  \draw[]        (0.6,3.5) to[out=-20,in=180] (0.9,4);
  
 \shade[ball color=white, opacity =1] (1,0) circle (0.1);
 \shade[ball color=black, opacity =1] (1,0.5) circle (0.1);
 \shade[ball color=white, opacity =1] (1,1) circle (0.1);
 \shade[ball color=white, opacity =1] (1,1.5) circle (0.1);
 \shade[ball color=white, opacity =1] (1,2) circle (0.1);
 \shade[ball color=black, opacity =1] (1,2.5) circle (0.1);
 \shade[ball color=white, opacity =1] (1,3) circle (0.1);
 \shade[ball color=black, opacity =1] (1,3.5) circle (0.1);
 \shade[ball color=black, opacity =1] (1,4) circle (0.1);
 \end{tikzpicture}
 \end{minipage}\begin{minipage}{0.25\textwidth}
\centering
\begin{tikzpicture}
\node[draw] at (0.5,5.5) {+Selection};
 \draw [thick, ->] (1,-1) -- (0,-1) node[midway,below] {$t$};
 \shade[ball color=white, opacity =1] (0,0) circle (0.1);
 \shade[ball color=white, opacity =1] (0,0.5) circle (0.1);
 \shade[ball color=white, opacity =1] (0,1) circle (0.1);
 \shade[ball color=white, opacity =1] (0,1.5) circle (0.1);
 \shade[ball color=white, opacity =1] (0,2) circle (0.1);
 \shade[ball color=white, opacity =1] (0,2.5) circle (0.1);
 \shade[ball color=white, opacity =1] (0,3) circle (0.1);
 \shade[ball color=white, opacity =1] (0,3.5) circle (0.1);
 \shade[ball color=white, opacity =1] (0,4) circle (0.1);
  \shade[ball color=white, opacity =1] (0,-0.5) circle (0.1);
   \shade[ball color=white, opacity =1] (0,4.5) circle (0.1);

 \shade[ball color=white, opacity =.2] (0.5,0) circle (0.1);
 \shade[ball color=white, opacity =.2] (0.5,1) circle (0.1);
 \shade[ball color=white, opacity =.2] (0.5,1.5) circle (0.1);
 \shade[ball color=black, opacity =.2] (0.5,2) circle (0.1);

 \shade[ball color=white, opacity =.2] (0.5,3) circle (0.1);
 \shade[ball color=white, opacity =.2] (0.5,3.5) circle (0.1);

  \draw[]        (0.1,0)  to[out=-20,in=180] (0.4,0);

  \draw[]        (0.1,1)  to[out=-20,in=180] (0.4,1);
  \draw[]        (0.1,1.5)   to[out=-20,in=180] (0.4,1.5);

  \draw[]        (0.1,3)   to[out=-20,in=180] (0.4,3);
  \draw[]        (0.1,3.5)  to[out=-20,in=180] (0.4,3.5);
  
  \draw[]        (0.1,-0.5)  to[out=10,in=190] (0.4,2);
    \draw[]        (0.1,0.5)  to[out=10,in=190] (0.4,2);
      \draw[]        (0.1,2)  to[out=10,in=190] (0.4,2);
        \draw[]        (0.1,2.5)  to[out=10,in=190] (0.4,2);
             \draw[]        (0.1,4)  to[out=10,in=190] (0.4,2);
        \draw[]        (0.1,4.5)  to[out=10,in=190] (0.4,2);

  \draw[]        (0.6,0)  to[out=-20,in=180] (0.9,0);
  \draw[]        (0.6,3.5)   to[out=-20,in=180] (0.9,0.5);
  \draw[]        (0.6,1)  to[out=-20,in=180] (0.9,1);
  \draw[]        (0.6,1.5)   to[out=-20,in=180] (0.9,1.5);
  \draw[]        (0.6,2)   to[out=-20,in=180] (0.9,2);
  \draw[]        (0.6,3.5)   to[out=-20,in=180] (0.9,2.5);
  \draw[]        (0.6,3)   to[out=-20,in=180] (0.9,3);
  \draw[]        (0.6,3.5)  to[out=-20,in=180] (0.9,3.5);
  \draw[]        (0.6,3.5) to[out=-20,in=180] (0.9,4);
  
 \shade[ball color=white, opacity =1] (1,0) circle (0.1);
 \shade[ball color=black, opacity =1] (1,0.5) circle (0.1);
 \shade[ball color=white, opacity =1] (1,1) circle (0.1);
 \shade[ball color=white, opacity =1] (1,1.5) circle (0.1);
 \shade[ball color=white, opacity =1] (1,2) circle (0.1);
 \shade[ball color=black, opacity =1] (1,2.5) circle (0.1);
 \shade[ball color=white, opacity =1] (1,3) circle (0.1);
 \shade[ball color=black, opacity =1] (1,3.5) circle (0.1);
 \shade[ball color=black, opacity =1] (1,4) circle (0.1);
 \end{tikzpicture}
 \end{minipage}
 \caption{Building blocks of the particle representation of the dual process $\widehat M$.
Each panel depicts a one-step transition of the dual, with an intermediate
configuration included to illustrate that every transition can be interpreted as
the outcome of two successive events.  In the dual picture, coalescence events may be followed by selective branchings or by killings of ancestral lines (or of the entire process).  Forward in time,
these correspond respectively to mutation or selective reproduction events
followed by neutral reproduction events.
}
\end{figure}
\noindent
Summarizing, the transitions of this ancestral system, and hence of $\wM$, consist of
(i) pure mergers (neutral), 
(ii) mergers with deletion (neutral + mutation to type $0$),
(iii) branching--merger events (neutral + selection),
(iv) killing events (neutral + mutation to type $1$); see Fig. \ref{particles} for an illustration of these four types of events. Thus, $\wM$ may be viewed as a generalization of the \emph{killed Ancestral Selection Graph} (see~\cite{BCH17,BW18}).

The particle construction of the chain $\wM$ described above is a key ingredient in the proof of the following comparison principle. This principle will subsequently be used to establish the asymptotic behavior of the dual chains in Proposition~\ref{pr-abs} of the next section.
\begin{lemma}[The effect of mutations]\label{compa}
There is a coupling between the chain $\wM$ and the chain $M^{\downarrow, 0}$ without mutations such that $\wM_n\leq M^{\downarrow,0}_n$ until the first time $\wM$ absorbs at $\partial$.
\end{lemma}
\begin{proof}
Set $\wM_0=m_0= M_0^{\downarrow,0}$ and assume that at generation $n$, we have $M^{\downarrow,0}_n\geq \wM_n\neq \partial$. In that case, we make both processes transition as follows. Let $S_0$ and $S$ denote systems of particles with $M^{\downarrow,0}_n$ and $\wM_n$ particles, respectively. Let us identify each particle in $S$ to a unique particle in $S_0$. Now we mark independently each particle in $S_0$ with probability $u$, where $u$ is a realization of a uniform random variable on $[0,\delta]$; the particles in $S$ copy the marks of their partners in $S_0$. With probability $1-\delta_0(\theta+\sigma_0)$ we only merge the marked particles in both systems. With probability $\delta_0\sigma_0\rho_j$ the merger event is coupled to a selective event of range $j$ in both systems. With probability $\delta_0\theta_0$ the merger event is coupled to a mutation event to type $0$ in $S$ leading to the additional loss of one particle; nothing happens in $S_0$. With probability $\delta_0\theta_1$ the merger event is coupled to a mutation event to type $1$ in $S$ sending the system to $\partial$; nothing happens in $S_0$. Hence, either $\wM_{n+1}=\partial$, or, in all the other cases, the order is preserved, i.e.
$M^{\downarrow,0}_{n+1}\geq \wM_{n+1}\neq \partial$. The results follows via a simple iteration of the argument.
\end{proof}
\begin{remark}[Comparison results]\label{compare}
The coupling technique used in the proof of Lemma \ref{compa} can be adapted to compare instances of the dual process $M^{\downarrow}$ with different branching rates. For example, consider $M^{\downarrow,1}$ and $M^{\downarrow,2}$ as instances of the dual process $M^{\downarrow}$ with the same initial value and mutation parameters $\theta_0,\theta_1$, but branching rates $(\sigma_j)_{j=0}^\infty$ and $({\sigma}_j^*)_{j=0}^\infty$, respectively. If $\sigma_j\leq \sigma_j^*$ for all $j\geq 0$, then one can couple $M^{\downarrow,1}$ and  $M^{\downarrow,2}$, so that $M^{\downarrow,1}_n\leq M^{\downarrow,2}_n$ for all $n$.\end{remark}
\subsection{Application: long-term behavior via moment duality}\label{sec:appli}
In this section we investigate the long–term behavior of the process
$Z^\downarrow$ introduced in Section~\ref{ss:duality_fittest_type_wins}.  
Our approach is as follows. We first analyze the asymptotic behavior of the process $\wM$, and then,
by means of Theorem~\ref{th:moment_duality_II}, we deduce the corresponding
long–term behavior of $Z^\downarrow$ in two specific situations:  
(i) the case without mutation, and  
(ii) the case with mutation but without selection.  
The results presented in this section require the following additional
assumption on the parameters $(\sigma_k)_{k \geq 0}$.

\begin{assumption}\label{eq:non_explosion}
The coefficients $(\rho_k)_{k\geq 1}$ satisfy 
\begin{align*}
\sum_{k \ge 1} k \rho_k \eqqcolon \mathfrak{b} < \infty.
\end{align*}
\end{assumption}
Note that, by the definition of the coefficients $(\rho_k)_{k \geq 1}$, 
under Assumption~\ref{eq:non_increasing},  
Assumption~\ref{eq:non_explosion} is equivalent to the existence of the limit 
$\lim_{k \to \infty} k \sigma_k$.
\begin{proposition}\label{pr-abs}
Under Assumptions \ref{eq:non_increasing} and \ref{eq:non_explosion} the following assertions hold.
\begin{enumerate}
\item If $\theta_0=\theta_1=0$, then $\wM$ is positive recurrent and admits a unique invariant distribution with support equal to $\N$.  
\item If $\theta_0,\theta_1>0$, then $\wM$ is almost surely absorbed in $\{0,\partial\}$ in finite time.
\end{enumerate}
\end{proposition}

\begin{proof}
Assume first that $\theta_0=\theta_1=0$. In this case, $\pf^{\downarrow}$ is differentiable and
$$\lim_{z \to 1}(\pf^{\downarrow})' (z)=1+ \dfrac{2\delta}{3} \sigma_0 \mathfrak{b}<\infty.$$ 
Using this and Eq. \eqref{zEZ}, we deduce that
$z\mapsto\E_z[(\wZ_1)^m]$ is also differentiable. Moreover, differentiating this function first using \eqref{zEZ} and then using
the moment duality \eqref{momdi2} with $n=1$, and letting $z \to 1$ in both expressions, we infer that
\begin{align*}
\E_m [\wM_1] -m =  - \dfrac{\delta}{2} m + \left (1+ \dfrac{2\delta}{3}\sigma_0 \mathfrak{b}\right) (1- \lambda_{m,0}),
\end{align*}
which can be made $< -1$ uniformly on $m \ge m_0$ for some $m_0 \ge 1$ (which may depend on $\delta$). Hence, we deduce from  \cite[Thm. 2.19]{benaimhurth2022} that any finite subset of $\N$ is positive recurrent for the chain $\wM$. Since, in addition,  $\wM$ is irreducible, the result follows. 

Assume now that $\theta_0,\theta_1>0$. As seen in Lemma \ref{compa} one can couple the chain $\wM$ to a copy $M^{\downarrow,0}$ of the chain without mutations in a way that $\wM\leq M^{\downarrow,0}$ until the first time $\wM$ absorbs at $\partial$. By construction, every time $M^{\downarrow,0}$ is at $1$, either $\wM$ is already absorbed at $0$ or at $\partial$, or it will absorb at $0$ with a constant positive probability in the next generation. So $\wM$ will absorb at $0$ or $\partial$ at the latest after a geometric number of excursions of $M^{\downarrow,0}$ (away from $1$). The result follows from the positive recurrence of $M^{\downarrow,0}$.
\end{proof}
\subsubsection{The case without mutations}
In this section, we briefly discuss the case where $\theta_0=\theta_1=0$, i.e. when there are no mutations.
In this setting, under Assumption \ref{eq:non_increasing}, we have $\pf^{\downarrow}(x)\leq x$. Thus, Proposition \ref{lemma:smartingales} implies that $\wZ$ is a bounded supermartingale. Therefore, $\wZ_n$ converges a.s. and in $L^p$ as $n\to\infty$ to a limit $\wZ_\infty$. Note also that $\pf^{\downarrow}$ satisfies the hypotheses of Proposition \ref{lemma:aslimit}, and thus $\wZ_\infty \in \{0,1\}$ almost surely; alternatively, under Assumption  \ref{eq:non_explosion}, this also follows from the duality relation (Thm. \ref{th:moment_duality_II}) and the positive recurrence of $\wM$ (Prop. \ref{pr-abs}). From the biological point of view, this means that there will be fixation of one of the two types almost surely. Moreover, using again that $\pf^{\downarrow}(x) \le x$, we deduce from Lemma \ref{lemma:comparison_theorem} that we can couple the chain $\wZ$ and the neutral chain $Z$ starting from $\wZ_0=Z_0=z \in [0,1]$ so that $\wZ_n \le Z_n$ for every $n \ge 0$. As a consequence,
\begin{align*}
\prob (\wZ_\infty=0) \ge \prob (Z_\infty=0)= 1-z.
\end{align*}
This already implies that if the population starts with an arbitrary, but positive fraction of type $1$ individuals, then the probability that type $1$ fixates (and hence, type $0$ gets extinct) in the population is strictly positive. The next result states that under Assumption \ref{eq:non_explosion} this also holds for type $0$.
\begin{proposition}
Assume that \ref{eq:non_increasing} and \ref{eq:non_explosion} hold. 
Then, we have for all $z\in(0,1)$
$$\Pb_z(\wZ_\infty=1)=1-\Pb_z(\wZ_\infty=0)=\E[z^{\wM_\infty}]\in(0,1),$$
where $\wM_\infty$ is a random variable distributed as the stationary distribution of $\wM$.
\end{proposition}
\begin{proof}
 Recall from Proposition \ref{pr-abs} that under Assumptions \ref{eq:non_increasing} and \ref{eq:non_explosion} $\wM$ is positive recurrent and has a unique stationary distribution $\hat{\pi}$. Since, in addition, $\wM$ is irreducible and aperiodic, $\wM_n$ converges in distribution, independently of the initial condition, to a random variable $\wM_\infty\sim\hat{\pi}$. Thus, letting $n\to\infty$ in the duality relation in Theorem \ref{th:moment_duality_II}, we deduce that
$$\Pb_z[\wZ_\infty=1]=\E[z^{\wM_\infty}],\quad z\in[0,1].$$
The result follows noting that for $z \in (0,1)$: 
\begin{align*}
0<\Pb[\wM_\infty=1]z\leq \sum_{k \ge 1} \Pb[\wM_\infty=k]z^{k}=\E[z^{\wM_\infty}]< 1.
\end{align*}
The first inequality is due to the first point at Proposition \ref{pr-abs}, the second inequality is obvious, the equality is just a definition, and  the last inequality is again a consequence of Proposition \ref{pr-abs}. 
\end{proof} 

\subsubsection{The case with mutations and no selection}
Let us now consider the case with mutations to type $0$ and $1$, i.e. when $\theta_0,\theta_1>0$. The aim of this section is to describe the long-term behavior of the process $\wZ$. As in the case without mutations, the moment duality in Theorem \ref{th:moment_duality_II} will play an important role.

Recall that, thanks to Proposition  \ref{pr-abs}, $\wM$  almost surely gets absorbed at $\{0,\partial\}$ in finite time. Let us then define for $r \in  \overline{\N}_0$
\begin{align*}
\tau_r\coloneqq\inf\{k\ge 0: \wM_k=r\}.  
\end{align*}
The next result shows that in the presence of bidirectional mutations the process $\wZ$ is ergodic, and characterizes the moments of its stationary distribution.
\begin{proposition}\label{nomut}
Assume that \ref{eq:non_increasing} and \ref{eq:non_explosion} hold. 
Then, the $\wZ_n$ converges in distribution as $n\to\infty$, independently of the initial condition, to a random variable $\wZ_\infty$ whose moments are given by
$$\E[(\wZ_\infty)^m]=\Pb_m(\tau_0<\tau_\partial),\quad m\in \overline{\N}_0.$$
\end{proposition}
\begin{proof}
Recall that under the assumptions of the proposition, the process $\wM$ almost surely gets absorbed at $\{0,\partial\}$ in finite time. Hence, letting $n\to\infty$ in the duality relation in Theorem \ref{th:moment_duality_II} we deduce that, for any $z\in [0,1]$
\begin{align*}\label{eq:lim_moments}
\lim_{n\to\infty}\E_z[(\wZ_n)^m]=\Pb_m(\tau_0<\tau_\partial).
\end{align*}
Note that for any $n\in\mathbb{N}_0$ and $z\in[0,1]$, the sequence 
$(\E_z[(\wZ_n)^m])_{m\ge 0}$ consists of the moments of a $[0,1]$-valued 
random variable, and is therefore completely monotone. Since 
$(\Pb_m(\tau_0<\tau_\partial))_{m\ge 0}$ arises as the pointwise limit of 
this sequence as $n\to\infty$, it is also completely monotone. Moreover, 
because $\Pb_0(\tau_0<\tau_\partial)=1$, the Hausdorff moment problem (see \cite{hau23}) implies 
that $(\Pb_m(\tau_0<\tau_\partial))_{m\ge 0}$ is the moment sequence of a
$[0,1]$-valued random variable. Thus, the result follows recalling that probability measures on $[0,1]$ are completely determined by their positive integer moments and that convergence of positive integer moments implies convergence in distribution. 
\end{proof}

 \begin{remark}[The case without selection and small $\delta$]\label{rem:moment_nutation} Assume now that $\sigma_i=0$ for $i\in\N_0$. Observe that by Proposition \ref{nomut}, $\wZ$ admits a unique invariant distribution. Denote by
 $\beta_{\theta_0, \theta_1; \delta}$ this invariant law; this time we make explicit the dependency on $\delta$, as we are interested in this law for small values of $\delta$. Denote the $m$-th moment of $\beta_{\theta_0, \theta_1;\delta}$ by $\mom{\beta_{\theta_0, \theta_1;\delta}}{m} $. According to Proposition \ref{nomut}
\begin{align*}
\mom{\beta_{\theta_0, \theta_1;\delta}}{m}= \prob_m (\tau_0 < \tau_\partial).
\end{align*}
Clearly, $\mom{\beta_{\theta_0, \theta_1;\delta}}{0}=1$. A short calculation shows as well that $\mom{\beta_{\theta_0, \theta_1;\delta}}{1}= \theta_0/\theta$, where $\theta=\theta_0+\theta_1$. More generally, a first-step decomposition yields for $n \ge 2$:
\begin{align*}
\mom{\beta_{\theta_0, \theta_1;\delta}}{m} &= \sum_{j=0}^{m} \prob_{m} (\tau_0 < \tau_\partial\vert \wM_1=j) \prob_m (\wM_1=j).
\end{align*}
Now, recall the definition of the $\lambda_{m,k}$ in \eqref{eq:lambdas}. A straightforward computation yields:
\begin{align*}
\lambda_{m,0}     &= 1-\dfrac{m}{2} \delta +\dfrac{m(m-1)}{6}\delta^{2} +o(\delta^{2}),\\
m\lambda_{m,1}   &= \dfrac{m}{2} \delta - \dfrac{m(m-1)}{3}\delta^{2}+ o(\delta^{2}),\\
\frac{m(m-1)}{2}\lambda_{m,2}   &= \dfrac{m(m-1)}{6} \delta^2 + o(\delta^2), \\
\lambda_{m, k} &= o (\delta^2), \quad k > 2.
\end{align*}
Thus, after inspection of the transition rates of $\wM$ in \eqref{ratedual}, we get
\begin{align*}
    \prob_m (\wM_1=m) &= 1- \left( \frac{m(m-1)}{6} + \frac{2\theta m}{6} \right)\delta^{2} +o(\delta^{2}) \\
    \prob_m (\wM_1=m-1)&= \left( \frac{m(m-1)}{6} + \frac{2\theta_0 m}{6} \right)\delta^{2} + o(\delta^{2})\\
    \prob_m (\wM_1=m-k)&=o (\delta^2).
\end{align*}
Hence, 
 \begin{align*}
 \mom{\beta_{\theta_0, \theta_1;\delta}}{m} &= 
 \mom{\beta_{\theta_0, \theta_1;\delta}}{m-1} \prob_{m} (\wM_1=m-1) +\mom{\beta_{\theta_0, \theta_1;\delta}}{m} \prob_m (\wM_1=m)  + o (\delta^2).
\end{align*}
and so: 
\begin{align*}
 \mom{\beta_{\theta_0, \theta_1;\delta}}{m} = \left ( \dfrac{m-1+2\theta_0}{m-1+2\theta}\right) \mom{\beta_{\theta_0, \theta_1;\delta}}{m-1} + o(1).
\end{align*}
Let $\beta_{\theta_0, \theta_1}$ be the $\betadist (2 \theta_0, 2 \theta_1)$ distribution. It is well known that the moments of $\beta_{\theta_0, \theta_1}$ satisfy the recurrence relation
\begin{align*}
 \mom{\beta_{\theta_0, \theta_1}}{m} = \left ( \dfrac{m-1+2\theta_0}{m-1+2\theta}\right) \mom{\beta_{\theta_0, \theta_1}}{m-1}.
\end{align*}
Since this recurrence relation determines the moments and 
$\beta_{\theta_0,\theta_1}$ is supported on $[0,1]$, then $\beta_{\theta_0,\theta_1}$ is the unique 
distribution satisfying the relation (see, e.g. \cite{ethierkurtz86}, 
Theorem~3.4.5). By these recurrences, the moments of 
$\beta_{\theta_0,\theta_1;\delta}$ converge, as $\delta \to 0$, to those of 
$\beta_{\theta_0,\theta_1}$. Because the convergence of moments for 
$[0,1]$-valued random variables implies convergence in distribution, we 
conclude that $\beta_{\theta_0,\theta_1;\delta}$ converges 
to $\beta_{\theta_0,\theta_1}$ as $\delta \to 0$.

\end{remark}
\subsection{Scaling limit}\label{sec:connection-WF}

In this section, we provide the proof of Theorem~\ref{thm:connection-WF}, which establishes the convergence of the Markov chain $Z^{(\delta_N,\pfn)}$, with $\delta_N$ and $\pfn$ satisfying~\eqref{eq:f_N}, to a Wright--Fisher diffusion with drift as $N \to \infty$ under the appropriate time scaling $t\mapsto\lfloor Nt\rfloor$. 
\begin{proof}[Proof of Theorem \ref{thm:connection-WF}]
Pathwise uniqueness for SDE \eqref{eq:WF-Phi-Psi} follows from \cite[Chap. 5, Thm. 3.8]{ethierkurtz86} (see also \cite[Chap IX, Thm. 3.5]{revuzyor1999}).
Existence of weak solutions for any initial condition follows from \cite[Chap. 5, Thm. 3.10]{ethierkurtz86}. Existence of strong solutions follows then from the theorem of Yamada-Watanabe. Furthermore, a standard application of It\^o's formula, shows that $X$ solves the martingale problem associated to the operator $\mathcal{L}:\mathcal{C}^{2} ([0,1]) \mapsto \mathcal{C} ([0,1])$ acting on $h \in \mathcal{C}^2 ([0,1])$ via 
\begin{align*}
 \mathcal{L} h (x) = b(x)\dfrac{\dd }{\dd x} h (x) + \dfrac{1}{2} x(1-x) \dfrac{{\dd}^2}{\dd x^2} h (x).
\end{align*}
We deduce from \cite[Chap. 8, Thm 2.1]{ethierkurtz86} that $X$ induces a Feller semigroup with infinitesimal generator $\mathcal{L}$ and $\mathcal{C}^\infty([0,1])$ as a core.
Therefore, according to \cite[Thm. 17.28]{kallenberg2002} it suffices to show that for every $f\in \mathcal{C}^{\infty}([0,1])$
$$\lim_{N\to\infty}\sup_{x\in[0,1]}\lvert N\E_x[f(Z_1^{(\delta_N,\pfn)})-f(x)]-\mathcal{L}(x)\rvert=0.$$
Setting for $k\in\N$  
\begin{align*}
	a^{(k)}_N (x)\coloneqq N \E^{(N)}_x [(Z_1^{(\delta_N,\pfn)}-x)^k],
\end{align*}
and using a Taylor expansion of $f$, we see that it suffices to show that
$$a_N^{(1)}(x)\to b(x),\quad a_N^{(2)}(x)\to x(1-x),\quad a_N^{(3)}(x)\to 0$$
as $N\to\infty$ uniformly in $x\in[0,1]$.
A  straightforward calculation shows that
\begin{align*} 
a^{(k)}_N (x) = \dfrac{3^{k/2}N^{1-k/2}}{(k+1)} \left ( (-1)^{k}(1-\pfn(x))x^{k}+ \pfn(x) (1-x)^{k}\right).
\end{align*}
Using this for $k=1$, $k=2$ and $k=3$ we see that
\begin{align*}
a^{(1)}_N(x) &= b (x),\quad a^{(2)}_N(x)=x(1-x)+O (N^{-1/2}),\quad a^{(3)}(x)=O(N^{-1/2}),
\end{align*}
where the error terms in the second and third identities are uniform in $x$. The result follows. 
\end{proof}

\begin{remark}[Fittest-type-wins revisited]\label{ftw-rem}
Note that choosing $b:[0,1]\to\R$ as
\begin{align*}
b(x) \coloneqq -\theta_1 x + \theta_0(1-x)-x(1-x)\sum_{j=0}^\infty \sigma_j x^j  ,\quad x\in[0,1],
\end{align*}    
and setting $\delta_N\coloneqq\sqrt{3/N}$, yields $\pfn=\pf_{\delta_N}^{\downarrow}$ (see Section \ref{ss:duality_fittest_type_wins} to recall the definition of $\pf_\delta^{\downarrow}$). Assume that Assumptions \ref{eq:non_increasing} and \ref{eq:non_explosion} are satisfied.
Then, the process $Z^{(\delta_N,\pfn)}$ admits a moment dual $M^{(\delta_N,\pfn)}$ (see Section \ref{ss:duality_fittest_type_wins} for the transitions).
According to Theorem \ref{thm:connection-WF} the process $(Z^{(N,\pfn)}_{\lfloor Nt\rfloor})_{t \geq 0}$ converges as $N\to\infty$ to the solution of the SDE
\begin{equation*}
        \dd X_t = \bigg(-\theta_1 X_t + \theta_0(1-X_t) -X_t(1-X_t)\sum_{j=0}^\infty \sigma_j X_t^j \bigg)\, \dd t + \sqrt{X_t(1-X_t)}\dd W_t,
    \end{equation*}
which is referred in the literature as Wright--Fisher diffusion with fittest-type-wins selection and bi-directional mutation.
An asymptotic analysis of the transition probabilities of the dual process $M^{(\delta_N,\pfn)}$ (similar to the one performed in the proof of Theorem \ref{thm:connection-WF}) allows one to show that 
$({M}^{(\delta_N,\pfn)}_{\lfloor N t\rfloor})_{t\geq 0}$ converges as $N\to\infty$ in distribution to the continuous-time Markov chain with generator acting on $h:\overline{\N}_0\to\R$ with $h(\partial)=0$ via
\begin{align*}\label{eq:generator}
\mathcal{\hat{L}} h (n) =  \bigg(\binom{n}{2}+n\theta_0\bigg) (h(n-1)-h(n)) + n \sigma_0 \sum_{j \ge 1 } \rho_j (h (n+j)-h(n))-n\theta_1 h(n).
\end{align*}
This limiting branching-coalescing process can be stochastically dominated by a branching process transitioning from $n\to n+j$ at rate $n\sigma_0\rho_j$, $j\geq 1$, which is non-explosive under Assumption \ref{eq:non_explosion}. Hence, our limiting process is also non-explosive (see \cite[Chap. V, Thm. 9.1]{Harris63}). In particular, the moment duality stated in Theorem \ref{momdi2} is preserved in the limit, and we recover the moment duality stated in \cite{GS18} (see also \cite{BEH23}).
\qed
\end{remark}

\section{Non-linear chains} \label{section:deme_level}

In this section we introduce the notion of \emph{non-linear Markov chain} (see \cite{kolokoltsov2010nonlinear} for a detailed exposition) and consider a particular class constructed on the basis of Wright--Fisher kernels; we will refer to them as \emph{non-linear Wright--Fisher chains}. We show that these chains admit a unique invariant distribution, which arises as their limit when time tends to infinity (see Section \ref{ss:ergodic}). Section \ref{section:MVWF} is devoted to the study of its scaling limits. More precisely, we construct, via an appropriate scaling of time and of parameters, a sequence of non-linear Wright--Fisher chains converging, as the scaling parameter tends to infinity, to a McKean--Vlasov diffusion. 

The results presented in this section generalize some of the results obtained in Section \ref{sec:connection-WF} to the non-linear setting. Moreover, as we will see in Section \ref{sec:prop_chaos_results}, the non-linear Wright--Fisher chains introduced in this section approximate the evolution of the type composition of pathogens within a single host, which interacts in a mean-field fashion with a myriad of (identically distributed) hosts. Therefore, we can see this section as a step towards the modeling of a two-level population (i.e. of the population at the deme level).

\subsection{Non-linear Markov chains and their invariant distributions}\label{ss:non_linear_markov}
 Roughly speaking, a non-linear Markov chain is a sequence of random variables where the transition to the next state does not only depend on the current state, but also on its distribution. This concept bears resemblance to the notion of \emph{non-linear Markov process}, a term initially introduced by McKean in the context of non-linear parabolic partial differential equations. For a detailed exposition on this topic, we refer to the seminal article by McKean \cite{mckean66}. 
\begin{definition}[Non-linear Markov chain]
	Let $(E, \mathcal{E})$ be a measurable space, $\mathbb {Q}\coloneqq(Q_\nu)_{\nu \in \mathcal{P}(E)}$ be a family of Markov transition operators on $(E, \mathcal{E})$ and $\mu_0 \in \mathcal{P}(E)$. We say that a sequence of random variables $(\nl Y_n)_{n \ge 0}$ on $(E, \mathcal{P}(E))$ is a \emph{non-linear Markov chain driven by $(\mu_0, \mathbb{Q})$} if $\nl{Y_0}\sim\mu_0$ and, for all $n\geq 1$, 
	\begin{equation}\label{eq:def-non-linear-MC}
	    \nl Y_{n} \sim Q_{\mu_{n-1}} (\nl Y_{n-1}, \cdot),\quad\textrm{where} \quad \mu_{n-1}\coloneqq \Law (\nl Y_{n-1}).
	\end{equation}
 If the initial distribution is not specified we say that the non-linear Markov chain is driven by $\mathbb{Q}$.
\end{definition}

Note that the one-dimensional distributions $(\mu_n)_{n\geq 0}$ of the non-linear chain $(\nl Y_n)_{n \ge 0}$ can be computed via the recursion 
\begin{align*}
\quad \mu_{n+1}&= \mu_n Q_{\mu_n},\quad n\geq 0. 
\end{align*}
More generally, the finite dimensional distributions can be computed in a recursive way as follows. Let $\vec{\mu}_{n}$ denote the law of $(\nl Y_0,\ldots,\nl Y_n)$, $n\geq 0$. Then, for any $n\geq 0$ and any bounded measurable function $f:E^{n+1}\to\R$, we have 
$$\vec{\mu}_{n+1}f=\int_{E^n}\vec{\mu}_{n}(\dd x_0,\ldots,\dd x_n)\int_E Q_{\mu_n}(x_n,\dd x_{n+1})f(x_0,\ldots,x_n,x_{n+1}).$$

In particular, if there is a transition kernel $Q$ on $(E, \mathcal{E})$ such that $Q_\nu=Q$ for all $\nu \in \mathcal{P}(E)$, then $(\nl Y_n)_{n\geq 0}$ is a Markov chain in the traditional sense. However, this is not true in general. Indeed, if $\mathcal{B}_b$ is the set of real-valued, bounded Borel measurable functions, the operators $T_n:\mathcal{B}_b \mapsto  \mathcal{B}_b$, $n\geq0$, defined through $T_nf(x) \coloneqq \mathbb{E}_x [f(\nl Y_n)]$, do not satisfy the semigroup property, i.e. $T_{n+m}$ is not necessarily equal to $T_n\circ T_m$. Nevertheless, if for $\mu \in \mathcal{P}(E)$ we denote $\mu \mathbb{Q}^{(n)}$ as the law of $\nl Y_n$ when $\nl Y_0$ has law $\mu$ (with the conventions $\mathbb{Q}=\mathbb{Q}^{(1)}$ and $\mathbb{Q}^{(0)}$ is the identity on $\mathcal{P}(E)$), then the semigroup property holds in the sense that $(\mu \mathbb{Q}^{(n)}) \mathbb{Q}^{(m)}=\mu \mathbb{Q}^{(n+m)}$. This motivates the following notion.

\begin{definition}[Invariant distributions]
	An invariant distribution for $\mathbb {Q}\coloneqq(Q_\nu)_{\nu \in \mathcal{P}(E)}$, or equivalently for the non-linear Markov chain driven by $\mathbb {Q}$, is a probability distribution $\eta\in\mathcal{P}(E)$ satisfying $\eta \mathbb{Q}= \eta$, where $\eta \mathbb{Q}:= \eta Q_\eta$. 
\end{definition}
\begin{remark}[Non-linear chains are linear under invariant distributions]\label{nlvsli}
Note that if $\eta$ is an invariant distribution for $\mathbb {Q}\coloneqq(Q_\nu)_{\nu \in \mathcal{P}(E)}$, then $\eta$ is an invariant distribution for the Markov chain driven by the kernel $Q_\eta$. Moreover, the non-linear Markov chain driven by $(\eta,\mathbb {Q})$ and the Markov chain driven by the the kernel $Q_\eta$ and initial distribution $\eta$ are equally distributed.
\end{remark}

The following result about existence and uniqueness of invariant distributions for non-linear Markov chains will be useful in the forthcoming sections.  

\begin{proposition} \label{th:existence}
	Assume that $(E, d)$ is a compact metric space. Let $\mathbb {Q}\coloneqq(Q_\nu)_{\nu \in \mathcal{P}(E)}$ be a family of Markov transition operators on $(E, \mathcal{B}(E, d))$. Assume that for every $\nu \in \mathcal{P}(E)$ there exists a unique stationary distribution $T \nu \in \mathcal{P}(E)$ for the kernel $Q_\nu$. Assume further that the map $\nu \mapsto T \nu$ is a $\Was$-strict contraction (i.e. it is a $\Was$-Lipschitz map with Lipschitz constant smaller than $1$).  Then, $\mathbb{Q}$ admits a unique invariant distribution.
\end{proposition}

\begin{proof}
	Since $(\mathcal{P}(E), \Was)$ is complete and $T$ is a contraction, by Banach's fixed point theorem there exists a unique element $\eta \in \mathcal{P}(E)$ such that $T \eta=\eta$. Note that by definition of $\eta \mathbb{Q}$ we have
 \begin{align*}
		\eta \mathbb{Q}&= \eta Q_\eta = (T\eta)Q_\eta
		= T \eta= \eta.
	\end{align*}
Hence, $\eta$ is invariant for the non-linear chain driven by $\mathbb{Q}$. Moreover, if $\eta_1$ is another invariant distribution for the non-linear chain, then $\eta_1 =\eta_1 \mathbb{Q}= \eta_1 Q_{\eta_1}$. Hence, $\eta_1$ is a fixed point of $T$. From the uniqueness of fixed points for $T$ we conclude that $\eta_1=\eta$. 
\end{proof}

\subsection{Ergodic properties of the Non-linear Wright--Fisher chains}\label{ss:ergodic}

In this section we introduce the non-linear Wright--Fisher chains and prove their ergodicity (Theorem \ref{th:non_linear_ergodic}). Before proceeding to the proof of this result, we say some words regarding its motivation and relevance. First, we remark the fact that, in general, it is far from trivial to guarantee that a given McKean-Vlasov equation admits unique invariant distribution, even in very simple cases (see e.g. \cite{HERRMANN20101215}). Then, even if one is able to guarantee the existence and uniqueness of an invariant distribution, the standard machinery of the ergodic theory of Markov process is not immediately available to  prove the attraction of any initial distribution to the invariant law, let alone to estimate the rates of this convergence. Now, as we will see in section \ref{section:MVWF} below, for a class of McKean-Vlasov equations  inspired by population genetic models, the non-linear chains we are dealing with in this section (after some scaling procedure) are good proxies of the diffusion models, and in fact, as we will see in Section \ref{sec:case_study}, the convergence rates of these non-linear chains (given by Theorem \ref{th:non_linear_ergodic}) can be lifted up to the limiting model. 

The non-linear chains we are going to deal with are constructed on the basis of the two-parameters family of kernels $(P_{\delta, p})_{\delta,p \in [0,1]}$ on $[0.1]$ introduced above. 

\begin{definition}[Non-linear Wright--Fisher chains]\label{nonlinearWF} 
Let $\NLpf: [0,1] \times \mathcal{P} \mapsto [0,1]$ be a measurable function and $\delta>0$. For $x \in [0,1]$ and $\nu \in \mathcal{P}$, set  $$Q_\nu^{(\delta,\NLpf)} (x, \cdot)\coloneqq P_{\delta, \NLpf(x, \nu)}(x, \cdot)\quad\textrm{and}\quad\mathbb{Q}^{(\delta,\NLpf)}\coloneqq (Q_\nu^{(\delta,\NLpf)})_{\nu \in \mathcal{P}}.$$
We call the non-linear Markov chain $\nl Z^{(\delta,\NLpf)}$ driven by $\mathbb{Q}^{(\delta,\NLpf)}$ a \emph{non-linear Wright--Fisher chain}.
\end{definition}
Since in this section $\delta>0$ will remain fix, we will write $\mathbb{Q}^{(\NLpf)}$ and $\nl Z^{(\NLpf)}$ instead of $\mathbb{Q}^{(\delta,\NLpf)}$ and $\nl Z^{(\delta,\NLpf)}$. Whenever the dependency in $\delta$ is relevant, we will make it explicit again. In  what follows we will work with a special class of functions $\NLpf$.


The next result paves the way to the existence and uniqueness of invariant distributions for the non-linear dynamics. 
\begin{lemma}\label{lemma:unique_invariant} Let $\NLpf: [0,1] \times \mathcal{P} \mapsto [0,1]$ be a $(L_1,L_2)$-Lipschitz function with $L_1<1$ (recall Definition \ref{ass:Lipschitz}). Then the following assertions hold.
	\begin{enumerate}
		\item For every $\nu \in \mathcal{P}$, $Q_\nu^{(\NLpf)}$ admits a unique invariant distribution $T\nu$.
		\item The map $\nu\mapsto T\nu$ is $\dfrac{L_2}{1-L_1}$-Lipschitz.
	\end{enumerate}
\end{lemma}
The proof of this lemma is based on the estimate \eqref{eq:pre_bound_00} for the standard coupling (see Remark \ref{stcoupling}) and we postpone it to Appendix \ref{app:to:ss:ergodic}.

Besides the previous Lemma, the proof of Theorem \ref{th:non_linear_ergodic} requires an extension of the standard coupling introduced in Remark \ref{stcoupling}. To do this, let $\NLpf,\NLqf:[0,1]\times \Ps\to[0,1]$ be measurable functions and $\mu,\nu\in\Ps$. Assume that $\left(\nl Z^{(\NLpf)}_0, \nl Z^{(\NLqf)}_0\right)$ is any coupling of $\mu$ and $\nu$, for example the coupling that attains the  Wasserstein distance. Denote by $\mu_n^{(\NLpf)}$ and $\mu_n^{(\NLqf)}$ the laws of 
$\nl Z^{(\NLpf)}_n$ and $\nl Z^{(\NLqf)}_n$, respectively (their definition doesn't depend on the realization of the processes). Assume now that we have already coupled the chains up to generation $n$ and let $U,U_1,U_2\sim\unif$ be independent between them, and independent from the two processes up to generation $n$. Set $p\coloneqq\pf(\nl Z^{(\NLpf)}_n,\mu_n^{(\NLpf)})$ and $q\coloneqq\pf(\nl Z^{(\NLqf)}_n,\mu_n^{(\NLqf)})$ and define

\begin{equation}\label{nlstcoupling}
\left\{
\begin{aligned}
\nl Z^{(\NLpf)}_{n+1}&=\upx_{p,q}(\nl Z^{(\NLpf)}_n,U,U_1,U_2)\quad\textrm{and}\quad\nl Z^{(\NLqf)}_{n+1}=\upy_{q}(\nl Z^{(\NLqf)}_n,U,U_1), \text{ if } p\leq q  \\ \\
\nl Z^{(\NLqf)}_{n+1}&=\upx_{q,p}(\nl Z^{(\NLqf)}_n,U,U_1,U_2)\quad\textrm{and}\quad\nl Z^{(\NLpf)}_{n+1}=\upy_{p}(\nl Z^{(\NLpf)}_n,U,U_1),  \text{ otherwise. }
\end{aligned}
\right.
\end{equation}
Thanks to part 3 of Lemma \ref{uple}, the processes $\nl Z^{(\NLpf)}$ and $\nl Z^{(\NLqf)}$ constructed by this procedure have the desired distributions and satisfy, for all $n\geq 0$, 
	\begin{align} \label{eq:nlpre_bound_00}
		\E[\vert\nl Z^{(\NLpf)}_{n+1}-\nl Z^{(\NLqf)}_{n+1}\vert] & \le \Bigg (1-\dfrac{\delta}{2}\Bigg) \E[\vert\nl Z^{(\NLpf)}_{n}-\nl Z^{(\NLqf)}_{n} \vert ]+\dfrac{\delta}{2}\E[\vert \NLpf(\nl Z^{(\NLpf)}_{n},\mu_n^{(\NLpf)})-\NLqf(\nl Z^{(\NLqf)}_{n},\mu_n^{(\NLqf)})\vert].
	\end{align}

\begin{proof}[Proof of Theorem \ref{th:non_linear_ergodic}]
The existence and uniqueness of the invariant distribution $\eta$ follows from  
	Lemma \ref{lemma:unique_invariant} and Proposition \ref{th:existence}. To prove the bound \eqref{eq:convergence_rate}, and hence the convergence towards the invariant distribution, we use a coupling argument. More precisely, we consider two copies $\nl Z^{(\NLpf)}$ and $\widetilde{Z}^{(\NLpf)}$ constructed via the non-linear standard coupling (with $\NLpf=\NLqf$), the first one with initial distribution $\mu^0$ and the second one with initial distribution $\eta$. Moreover, we assume that the initial values are optimally coupled (with respect to the Wasserstein distance). Since, $\NLpf$ is $(L_1,L_2)$-Lipschitz, Eq. \eqref{eq:nlpre_bound_00} delivers
 	\begin{align*} 
		\E[\vert\nl Z^{(\NLpf)}_{n+1}-\widetilde{Z}^{(\NLpf)}_{n+1}] & \le \Bigg (1-\dfrac{\delta(1-L_1)}{2}\Bigg) \E[\vert\nl Z^{(\NLpf)}_{n}-\widetilde{Z}^{(\NLpf)}_{n} \vert ]+\dfrac{\delta L_2}{2}\Was(\mu_n^{(\NLpf)},\eta)\\
  &\le \Bigg (1-\dfrac{\delta(1-L_1-L_2)}{2}\Bigg) \E[\vert\nl Z^{(\NLpf)}_{n}-\widetilde{Z}^{(\NLpf)}_{n} \vert ].	
  \end{align*}
  Iterating this inequality yields Eq. \eqref{eq:convergence_rate}.
  For the last claim, recall that $T \eta$ is the invariant distribution of the kernel $P_{\delta, \bmath{\bar{s}}}$ with $\bmath{\bar{s}}=\NLpf(0,\eta)$, and hence, by Theorem \ref{cor:unique}, $T\eta=\beta_{\delta,\bmath{\bar{s}}}$. 
 But according to Proposition \ref{th:existence}, $\eta$ is the unique solution of the equation $T\eta=\eta$. This implies that $\bmath{\bar{s}}=\NLpf(0,\beta_{\delta,\bmath{\bar{s}}})$. The uniqueness of $\bmath{\bar{s}}$ follows from Proposition \ref{prop:distance-invartiant}.
\end{proof}

\begin{remark}
Theorem \ref{th:non_linear_ergodic} is similar to the results in \cite[Section 2]{butkovsky2014}, but the results therein are stated in total variation distance instead of Wasserstein distance (other works addressing the long-term behavior of non-linear Markov processes are \cite{neumann2023nonlinear} and \cite{saburov2016ergodicity}).
\end{remark}

\subsection{Scaling limits: the McKean--Vlasov Wright--Fisher process}\label{section:MVWF}

In Section~\ref{sec:connection-WF}, we proved that, as $\delta \to 0$, Markov chains based on Wright--Fisher kernels of the form $P_{\delta,\pf_\delta}$ converge to a broad class of Wright--Fisher diffusions. The structure of these processes allowed us to apply well-known results on the weak convergence of Markov chains to diffusion processes. Our goal here is to prove Theorem~\ref{th:convergence_theorem}, which establishes the convergence of the non-linear chains introduced in the previous section to the McKean--Vlasov Wright--Fisher diffusion. In this setting, the techniques used for classical Markov processes are no longer directly applicable. Instead, our approach relies on convergence results for semimartingales together with uniqueness for an associated non-linear martingale problem.

The remainder of the section is organized as follows. We begin by proving Theorem~\ref{th:convergence_theorem} based on three auxiliary propositions, each addressing a key component of the argument: tightness, characterization of accumulation points, and the connection between uniqueness of solutions to the McKean--Vlasov diffusion and an associated martingale problem. We then establish these auxiliary propositions in separate subsections.

\subsubsection{The skeleton of the proof}\label{ss:convergence}
In what follows we will consider a function $\NLb:[0,1] \times \mathcal{P} \mapsto \mathbb{R}$ satisfying Assumption \ref{ass:on_b}.
As announced in display \eqref{scale-setting} in Section \ref{sec:overview-results}, we will consider the following scaling of parameters:
\begin{align*}
\delta_N= \sqrt{\frac{3}{N}}\quad\textrm{and}\quad	\NLpfn(x, \mu) = x + \frac{2}{\sqrt{3N}}\NLb(x,\mu).
\end{align*} 
Under Assumption \ref{ass:on_b}, the function $\NLpfn(x, \mu)$ is $(1+2L_1/\sqrt{3N},2L_2/\sqrt{3N})$-Lipschitz, and there exists $N_1>0$, such that, for any $N\geq N_1$ and any $(x,\mu)\in[0,1] \times \mathcal{P} ([0,1])$, $\NLpfn(x, \mu)\in[0,1]$.

We now consider the non-linear chain $\nlZn$ with $N \ge N_1$. Our goal is to
prove Theorem~\ref{th:convergence_theorem}; that is, we show that for any
$T \ge 0$ the sequence of semimartingales $(\sNLchainN_{[0,T]} : N \ge N_1)$,
defined in~\eqref{def:systems} by
\(
\sNLchainN_t = \nlZn_{\lfloor Nt \rfloor}, \, t \ge 0,
\)
converges to the unique solution of a McKean--Vlasov equation. More precisely,
we prove that the laws of these rescaled processes, viewed as probability
measures on the Skorokhod space equipped with the $J_1$ topology, converge weakly
to the unique solution of a non-linear martingale problem, which coincides with
the unique solution of the McKean--Vlasov diffusion~\eqref{eq:MV}. By a
\emph{solution} of this equation we mean the following.

\begin{definition}\label{def:weak_solution}
We say that the MV equation \eqref{eq:MV} has a weak solution on $[0, T]$ (resp. $[0, \infty)$) if there exists a filtered probability space $(\Omega, \FF, (\FF_t), \prob)$, an $\FF_t$-standard brownian motion and a continuous $[0,1]$-valued $\FF_t$-adapted process such that $\prob$-a.s., \eqref{eq:MV} holds on $t \in [0, T]$ (resp. $t \ge 0$). We say that weak uniqueness hold on $[0, T]$ (resp. $[0, \infty)$) if any two weak solutions induce the same law on $\mathcal{P} (\mathcal{C}([0, T]; [0,1]))$ (resp. $\mathcal{P} (\mathcal{C}([0, \infty); [0,1]))$).
\end{definition}

The proof of Theorem \ref{th:convergence_theorem} is presented in three steps, each of them encapsulated in a Proposition. First we state the tightness of the sequence of laws of $(\sNLchainN)_{N \ge 1}$ (Proposition \ref{prop:tightness-scaled-non-linear-markov-chain}), then we identify the limits points of this sequence as solutions of a \emph{nonlinear martingale problem}, which we introduce below  (Proposition \ref{prop:limits-of-non-linear-markov-chain-solve-NLMP}). Finally, we prove the uniqueness of the this limit (Proposition \ref{prop:uniqueness_NLMP}). The proofs of these propositions are deferred to the following subsections.

\begin{proposition}[Tightness]\label{prop:tightness-scaled-non-linear-markov-chain} Under Assumption \ref{ass:on_b}, the sequence of laws $(P^{(N)})_N$ of the semimartingales $(\sNLchainN)_N$, as probability measures on the Skorokhod space $\mathbb{D}([0, +\infty); [0,1])$, is a tight sequence. 
\end{proposition}

To identify the limits of the sequence  $(P^{(N)})_N$, we introduce 
a \textit{nonlinear martingale problem.} 

Consider the space $\Omega= (\mathbb{D} ([0, +\infty); [0,1]), \mathcal{B}  (\mathbb{D} ([0, +\infty); [0,1])))$, where the underlying topology is the $J_1$-Skorokhod topology, and endow it with the canonical filtration $(\FF_t)_{t \ge 0}$. For $\mu \in \mathcal{P}$ and $\varphi$ regular enough we put:
    \begin{align*}
    L_\mu \phi (x):= \NLb(x, \mu) g' (x)+\frac{1}{2}x(1-x) g'' (x).
    \end{align*}
The \emph{non-linear} martingale problem reads as follows:
\smallskip
	\begin{center}
	\begin{minipage}{0.9 \textwidth}{\it
Find $\mu \in \mathcal{P} (\mathbb{D} ([0, +\infty); [0,1]))$ endowed with its Borel $\sigma$-field and the canonical filtration, such that, with $\dif_t (\omega)= \omega_t$ the canonical process, for every $g \in \mathcal{C}^{\infty} ([0,1])$, the process:
\begin{equation} \label{eq:NLMP}
    \tag{NLMP}
    \begin{aligned}
	M^{(g)}_t: =&\; g (\dif_t)- g (\dif_0)- \int_0^t L_{\mu_s}g(\dif_s)\dd s, \quad (t \ge 0)
    \end{aligned}
\end{equation}
is a $\FF_t$-martingale under $\mu$ and, moreover, $\mu_0= \overline \mu_0$. }
\end{minipage}
\end{center}
An application of It\^o's formula shows that  $\nlLaw$, the law of $\nlDif$ solution of \eqref{eq:MV}, solves this nonlinear martingale problem.
\medskip

\begin{proposition}[Characterization] \label{prop:limits-of-non-linear-markov-chain-solve-NLMP}
	Under Assumption \ref{ass:on_b}, any weak limit point $P$ of $(P^{(N)})_{N \ge 1}$ solves the non-linear martingale problem \eqref{eq:NLMP}.
\end{proposition}

\begin{proposition}[Uniqueness]\label{prop:uniqueness_NLMP}
Under Assumption \ref{ass:on_b}, if the McKean--Vlasov stochastic differential equation \eqref{eq:MV} admits a unique weak solution $\nl X$, the \eqref{eq:NLMP} has a unique solution in $\mathcal{P}(\mathcal{C}([0, +\infty); [0,1]))$ (which is, of course, the law of the process $\nl X$).
\end{proposition} 
Now we are in position to prove the main result of this section. 
\begin{proof}[Proof of Theorem \ref{th:convergence_theorem}] As announced in the beginning of this section, by Proposition \ref{prop:tightness-scaled-non-linear-markov-chain}, there exist limit points for the sequence $(\Law (\nlZn))_{N \ge 1}$. Moreover, any of these limit points solves the non-linear martingale problem thanks to Proposition \ref{prop:limits-of-non-linear-markov-chain-solve-NLMP}. Nevertheless, according to Proposition \ref{prop:uniqueness_NLMP}, the non-linear martingale problem has a unique solution; which is in fact, the unique solution of the McKean--Vlasov equation. Summarizing, the the sequence $(\Law (\nlZn))_{N \ge 1}$ converges to the law of the solution of \eqref{eq:MV}. 
\end{proof}

\subsubsection{Tightness}\label{app:proofs-convergence-to-WFMV}
To prove the result, we rely on the well-known Rebolledo-Aldous criterion, which we recall next. We say that a family of processes $(V^{(N)}_t: t \ge 0)_{N \ge 1}$ satisfy the tightness condition $\bf [T_1]$ if for every $t \ge 0$ the family $(\Law (V^{(N)}_t))_{N \ge 1} $ is tight. We say that a family of processes $(V^{(N)}_t: t \ge 0)_{N \ge 1}$ ($V^{(N)}$ being a $\FF^{(N)}$-adapted process), satisfy the Aldous' condition $\bf {[A]}$ if for each $C > 0, \varepsilon > 0, \eta > 0$ there exists $\delta > 0, N' > 0$ such that for any sequence of stopping times bounded by $C$, $(\tau^{(N)})_{N \ge 1}$ ($\tau^{(N)}$ being a $\FF^{(N)}$-stopping time)  we have:
\begin{align*}
\sup_{N \ge N'} \sup_{\theta < \delta}\prob [\vert V^{(N)}_{\tau^{(N)}}- V^{(N)}_{\tau^{(N)}+ \theta}\vert > \eta ] < \varepsilon.
\end{align*}

The Rebolledo-Aldous criterion states that if for a sequence of semimartingales in $\mathbb{D}([0, \infty); [0,1])$, both, the sequence of predictable processes and the sequence of quadratic predictable processes associated to the martingale parts, satisfy the conditions $\bf[T_1]$ and $\bf [A]$, then the sequence of semimartingales is a tight sequence on $\mathbb{D}$ (see e.g. Corollary 2.3.3 at \cite{joffe1986}).

\begin{proof}[Proof of Proposition \ref{prop:tightness-scaled-non-linear-markov-chain}]
Let us start by computing the characteristics of $\nlZn$. Denote $\nl{\mu}^{(N)}_n= \Law (\nlZn_n)$, and:	
	\begin{align*}
		A^{(N)}_n= \dfrac{\sqrt {3}}{2N^{1/2}} \sum_{k=0}^{n-1} (\NLpfn (\nlZn_k, \nl{\mu}^{(N)}_k)-\nlZn_k), 
	\end{align*}
	\begin{align*}
		M^{(N)}_n =\nlZn_n -\nlZn_0 - A_n^{(N)}.
	\end{align*}
	It is straightforward to show that $M^{(N)}$ is a square-integrable martingale, with quadratic predictable process:
	\begin{align*}
		\langle M^{(N)} \rangle_n & = \dfrac{3}{N}\sum_{k=0}^{n-1} \Bigg ( \dfrac{1}{3} (\nlZn_k)^{2} (1-\NLpfn (\nlZn_k, \nl{\mu}^{(N)}_k)) \\
  & \qquad +\dfrac{1}{3}     (1-\nlZn_k)^2 \NLpfn (\nlZn_k, \nl{\mu}^{(N)}_{k})  - \dfrac{1}{4} (\NLpfn (\nlZn_k, \nl{\mu}^{(N)}_k)-\nlZn_k)^2      \Bigg  ) .  
	\end{align*}
Write $\tilde A_t^{(N)}= A^{(N)}_{\lfloor Nt\rfloor}$ and $\tilde K^{(N)}_t= \langle M^{(N)}\rangle_{\lfloor Nt\rfloor}$. We need to prove that the sequences of processes $(\tilde A^{(N)})_{N \ge 1}$ and $(\tilde K^{(N)}) _{N \ge 1} $ satisfy the conditions $\bf [T_1]$ and $\bf [A]$ above. 

Let's start with the tightness condition. For fixed $t \ge 0$, we have, for every $N \ge 1$:
\begin{align*}
\E [\vert \tilde A^{(N)}_t\vert ] \le \dfrac{1}{N} \E \bigg [\sum_{k=0}^{\lfloor Nt \rfloor -  1} \bigg\vert \NLb (\nlZn_k, \mu^{(N)}_k)\bigg\vert \bigg ] \le C t,
\end{align*}
for some constant $C'$ that depends on the bounds of $\NLb$. Analogously, the terms inside the sum at the definition of $\langle M^{(N)} \rangle$ are bounded, and again, for some constant $C''$, we have $\E [\vert \tilde K^{(N)}_t\vert ] \le C'' t $ for every $N \ge 1$. We deduce the tightness of the laws of the sequences $(\tilde A^{(N)}_t)_{N \ge 1}$ and  $(\tilde K^{(N)}_t)_{N \ge 1}$.

As for the the Aldous' condition, fix $C > 0, \varepsilon> 0, \eta >0$. For any sequence of stopping times $\tau^{(N)}$ and any $\theta> 0$:
\begin{align*}
\bigg \vert \tilde A^{(N)}_{\tau^{(N)}+ \theta}- A^{(N)}_{\tau^{(N)}+ \theta} \bigg \vert =\dfrac{1}{N} \Bigg \vert  \sum_{k= \lfloor N \tau^{(N)} \rfloor }^{\lfloor N \tau^{(N)}+ N\theta\rfloor -1} \NLb (\nlZn_k, \mu^{(N)}_k)\Bigg\vert,
\end{align*}
The sum contains at most $N \theta+1$ terms, uniformly bounded by some constant depending on $\NLb$, say $C'$. Thus:
\begin{align*}
\bigg \vert \tilde A^{(N)}_{\tau^{(N)}+ \theta}- \tilde A^{(N)}_{\tau^{(N)}+ \theta} \bigg \vert\ > \eta \Rightarrow C' (\theta+ 1/N) > \eta,
\end{align*}
and the right hand side fails to hold as soon as $\theta$ is small enough and $N$ is large enough. So, the sequence $(\tilde A^{(N)})_{N \ge 1}$ satisfies the condition $\bf [A]$. A completely analogous reasoning leads to the verification that this condition also holds for $(\tilde K^{(N)})_{N \ge 1}$.
\end{proof}

\subsubsection{Identification of the limit}
We have prove that the sequence of semimartingales $(\sNLchainN_{[0, T]}: N \ge 1)$ is tight. Now we proceed to show that any of its limit points solves the non-linear martingale problem.

\begin{proof}[Proof of Proposition \ref{prop:limits-of-non-linear-markov-chain-solve-NLMP}]

Let $P$ be any limit point of the sequence $(P^{(N)})$; without loss of generality, we can assume that the whole sequence converges weakly to $P$. As above,  $\dif_t(\omega)= \omega_t$ stands for the canonical process on the Skorokhod space. The functional $\bmath{\Psi}: \mathbb{D} ([0,+\infty); [0,1]) \mapsto \overline \R_{+}$ given by $\omega \mapsto \sup_{t \ge 0} \vert \omega_t- \omega_{t-}\vert $ is continuous for the Skorokhod topology, and under $P^{(N)}$, $\bmath{\Psi}(\dif) \le N^{-1/2}$. Thus,  by definition of weak convergence:
\begin{align*}
\int_{\mathbb{D}} \bmath{\Psi} (\omega) P (\dd \omega) = \lim_{N \to \infty} \int_{\mathbb{D}} \bmath{\Psi} (\omega) P^{(N)}(\dd \omega) =0, 
\end{align*}
and consequently, $P (\mathcal{C} ([0, +\infty); [0,1]))=1$. Write $P_s\coloneqq \Law_P (\dif_s)=P \circ \dif_s^{-1} \in \mathcal{P} ([0,1])$, $s \ge 0$. For a fixed test function $\varphi \in \mathcal{C}^{\infty}([0,1])$, define the functional $\mathcal{M}_t: \mathbb{D}([0, +\infty); [0,1]) \mapsto \R$ by:
\begin{align*}
	\mathcal{M}_t (\omega)\coloneqq \varphi (\omega_t)- \varphi (\omega_0) - \int_{0}^{t}L_{P_s} \varphi (\omega_s) \dd s.
\end{align*}
Observe that, although $\omega \mapsto \omega_s$ is not continuous for the Skorokhod topology, it is indeed continuous when restricted to the space of continuous paths. Thus, for each $t \in [0,1]$, $\mathcal{M}_t$ is $P$-a.s. continuous. Let $\varphi_1, \varphi_2, \ldots \varphi_k$ be a family of  continuous functions on $[0,1]$, $0 \le t_1 < t_2 < \ldots < t_k < s < t$. Consider the functional $\Xi: \mathbb{D}([0,+\infty); [0,1]) \mapsto \R_+$:
\begin{align*}
	\Xi (\omega)= \prod_{i=1}^{k}\varphi_k (\omega_{t_k})( \mathcal{M}_t (\omega)-\mathcal{M}_s (\omega)).
\end{align*}
Once again we note that the set of discontinuities of $\Xi$ lies outside the support of $P$.
	
Write $\E$ and $\E^{(N)}$ for expectation with respect to $P$ and $P^{(N)}$, respectively. By weak convergence and the continuous mapping theorem: 
\begin{align} \label{eq:weak_convergence}
		\E [\Xi (\dif)]= \lim_{N} \E^{(N)} [\Xi (\dif)]. 
\end{align}
	Hence, we just need to prove that the sequence $(\E^{(N)} (\Xi (\dif)))_N$ converges to $0$. 
    
    Let $\delta= (3/N)^{1/2}$, and put:
	\begin{align*}
		L^{(N)}_\mu \varphi(x)= (1-\NLpfn (x, \mu)) \E[\varphi(x-x\delta U)-\varphi(x)]+ \NLpfn (x, \mu) \E[\varphi(x+(1-x)\delta U)-\varphi(x)].
	\end{align*}
	Write:
	\begin{align*}
		\mathcal{M}^{(N)}_t &= \varphi (\omega_t)- \varphi (\omega_0)- \sum_{k=0}^{\lfloor Nt \rfloor-1} L^{(N)}_{P^{(N)}_{k/N}} \varphi (\omega_{k/N}); \\
		\mathcal{Q}^{(N)}_t &=\varphi (\omega_t)- \varphi (\omega_0)-  \sum_{k=0}^{\lfloor Nt \rfloor-1} L^{(N)}_{P_{k/N}} \varphi (\omega_{k/N});
	\end{align*}
	and define $\Xi^{(N)}$ and $\tilde \Xi^{(N)}$ as $\Xi$ but with $\mathcal{M}$ replaced by $\mathcal{M}^{(N)}$ and $\mathcal{Q}^{(N)}$, respectively. Clearly:
    \begin{equation}\label{eq:three_term}
	\begin{aligned} 
		\vert\E^{(N)}_{} [\Xi (\dif)]\vert &\le  \vert \E^{(N)}_{} [\Xi^{(N)}_{} (\dif)] +  \E^{(N)}_{}[\vert  \Xi^{(N)}_{}(\dif)-\tilde \Xi^{(N)}_{} (\dif)\vert ]\\
		&\quad+  \E^{(N)}_{} [\vert \tilde \Xi^{(N)}_{}(\dif)-\Xi (\dif)\vert ], 
	\end{aligned}
	\end{equation}
and observe that the first term vanishes since $\mathcal{M}^{(N)}$ is a martingale under $P^{(N)}$. So, we have to show that the remaining terms vanish in the limit. Now, as $N$ goes to infinity, through a Taylor expansion up to the second order it is easily checked that:
	\begin{align*}
		L^{(N)}_\mu\varphi (x) = \dfrac{1}{N} (\NLb(x, \mu)\varphi' (x)+ \dfrac{1}{2N} x(1-x) \varphi'' (x)) + o (1/N),
	\end{align*}
    where $o (1/N)$ is uniform on $x\in[0,1]$. Hence:
    \begin{align*}
	 		\vert \mathcal{M}_t - \mathcal{Q}^{(N)}_t \vert &= \left \vert \int_{0}^t L_{P_s} \varphi (\omega_s)\dd s - \sum_{k=0}^{\lfloor Nt \rfloor -1 } L^{(N)}_{P_{k/N}} \varphi (\omega_{k/N})\right \vert \\
    & \le \left \vert \int_{0}^t L_{P_s} \varphi (\omega_s)\dd s - \dfrac{1}{N}\sum_{k=0}^{\lfloor Nt \rfloor -1 } L_{P_{k/N}} \varphi (\omega_{k/N})\right \vert + o(1).
	 \end{align*}
Under $P^{(N)}$, we a.s. have $\omega_s= \omega_{k/N}$ for $s\in [k/N, (k+1)/N)$. Consequently, under $P^{(N)}$,
\begin{align*}
	 		\vert \mathcal{M}_t - \mathcal{Q}^{(N)}_t \vert &\le \left \vert \int_{0}^t L_{P_s} \varphi (\omega_s)\dd s - \sum_{k=0}^{\lfloor Nt \rfloor -1 } \int_{k/N}^{(k+1)/N}L_{P_{k/N}} \varphi (\omega_{s})\dd s\right \vert + o(1) \\
    & \le \sum_{k=0}^{\lfloor Nt\rfloor-1} \int_{k/N}^{(k+1)/N} \vert L_{P_{k/N}}\varphi (\omega_s)- L_{P_s} \varphi (\omega_s)\vert \dd s + o(1) \\
    & \le C \sum_{k=0}^{\lfloor Nt\rfloor-1} \int_{k/N}^{(k+1)/N} \Was (P_{k/n}, P_s)\dd s + o(1),
\end{align*}
  where the constant $C$ involves the Lipschitz constants of $\NLb$ and bounds on $\varphi'$. Since $s\mapsto P_s$ is continuous, in particular it is uniformly continuous on $[0,t]$. Hence, given $\varepsilon > 0$, there exists $N_0 > 0$ such that for every $N > N_0$, $0 \le k < \lfloor Nt\rfloor$ and $s \in [k/N, (k+1)/N)$, we can ensure $\Was (P_{k/N}, P_s)< \varepsilon$. Then, again under $P^{(N)}$ for $N > N_0$:
  \begin{align*}
	 		\vert \mathcal{M}_t - \mathcal{Q}^{(N)}_t \vert &\le C \varepsilon t + o(1).
	 \end{align*}
  By bounded convergence, we deduce that the third term in \eqref{eq:three_term} converges to $0$ as $N\to\infty$. Similarly, we have
  \begin{align*}
		\vert \mathcal{M}^{(N)}_t - \mathcal{Q}^{(N)}_t \vert & \le \dfrac{1}{N}  \sum_{k=0}^{\lfloor Nt\rfloor - 1}\left \vert L_{P_{k/N}} \varphi (\omega_{k/N})- L_{P^{(N)}_{k/N}} \varphi (\omega_{k/N}) \right \vert + o(1) \\
  & \le C \dfrac{1}{N}  \sum_{k=0}^{\lfloor Nt\rfloor - 1} \Was (P_{k/N}, P^{(N)}_{k/N}) + o(1),
\end{align*}
and, thus, under $P^{(N)}$,
\begin{align*}
		\vert \mathcal{M}^{(N)}_t - \mathcal{Q}^{(N)}_t \vert &\le C  \sum_{k=0}^{\lfloor Nt\rfloor - 1} \int_{k/N}^{(k+1)/N}\Was (P_{k/N}, P_s)\dd s \\
 & \quad  + C  \sum_{k=0}^{\lfloor Nt\rfloor - 1} \int_{k/N}^{(k+1)/N}\Was (P_{s}, P^{(N)}_s)\dd s + o(1) \\
 & = C  \sum_{k=0}^{\lfloor Nt\rfloor - 1} \int_{k/N}^{(k+1)/N}\Was (P_{k/N}, P_s)\dd s \\
 & \quad  + C  \int_{0}^{t}\Was (P_{s}, P^{(N)}_s)\dd s + o(1). 
\end{align*}
The same argument used previously shows that for any $\varepsilon> 0$ there exists $N_0$ large enough such that for $N \ge N_0$, the first sum in the last line above can be made smaller that $C \varepsilon t$. As for the integral term, since $P^{(N)} \to P$ weakly and the map $\omega \mapsto \omega_s$ is $P$-a.s. continuous, we obtain by the continuous mapping theorem that $P^{(N)}_s \to P_s$ weakly; since $P_s, P_s^{(N)}$ are elements of $\mathcal{P} ([0,1])$, we have that $\Was (P_s, P^{(N)}_{s})$ converge to $0$, and we deduce by Lebesgue dominated convergence theorem that the integral in the last line above converges to $0$. This (and a straightforward application of the bounded convergence theorem) shows that the second term in \eqref{eq:three_term} converge to $0$. 
\end{proof}
\subsubsection{Uniqueness of the limit}
Finally, we show here that set of limits points of  \((P^{(N)}: N \ge 1)\) consists in only one element which corresponds to the law of the solution of \eqref{eq:MV}.

\begin{proof}[Proof of Proposition \ref{prop:uniqueness_NLMP}]
Observe that any solution to the non-linear martingale problem solves the non-linear Fokker-Planck equation in the weak form:
\begin{equation} \label{eq:non_Linear_FP}
\begin{aligned}
	\dfrac{\dd}{\dd t} \inner{\mu_t}{\varphi} = \inner{\mu_t}{L_{\mu_t} \varphi}; \quad \mu_0= \overline \mu_0, \quad t \ge 0,
\end{aligned}
\end{equation}
for every $\varphi \in \mathcal{C}^\infty{}([0,1])$. Consequently, the uniqueness of the solution of the non-linear martingale problem will follow as soon as we prove that equation \eqref{eq:non_Linear_FP} admits a unique solution in $\mathcal{P}(\mathcal{C}([0, +\infty); [0,1]))$. The proof of this fact relies on arguments regarding the uniqueness of solutions for the ``linearized'' forms of both, the SDEs and the martingale problem parameterized by flows of probability measures; similar reasoning can be found in the proof of Theorem 2.5 of \cite{bossy2019}, and what follows can be regarded as a streamlined argument. 

Observe that, under our assumptions:
\begin{enumerate}
\item For a given continuous flow of probability measures, there exists a unique strong solution for the ``linearized'' SDE associated to our problem. More precisely, for any measure $ \nu \in \mathcal{P} (\mathcal{C}([0,+\infty] ; [0,1]))$, there exists a unique strong solution of the ``linearized'' equation:
\begin{align}\label{eq:linearized-sde}
\dd X^{(\nu)}_t = \NLb(X_t^{(\nu)}, \nu_t) \dd t + \sqrt{X_t^{(\nu)}(1-X_t^{(\nu)})} \dd W_t; \quad X^{(\mu)}_0 \sim \nu_0,\quad t \ge 0,
\end{align}
where $\nu_t$ stands for the image measure of $\nu$ under the coordinate map $(x_s)_{s \ge 0} \mapsto  x_t$. Indeed, notice that, under Assumption \ref{ass:on_b} the weak existence of the linearized equation is given by Skorokhod-Stroock-Varadhan Theorem (See Theorem 5.4.22 and Remark 5.4.23 in \cite{karatzas91}). Moreover pathwise uniqueness follows from Yamada-Watanabe Theorem (Proposition 5.3.20 and Corollary 5.3.23 in \cite{karatzas91}), hence there exists a unique strong solution for equation \eqref{eq:linearized-sde}. Furthermore, under assumption \ref{ass:on_b}, it is easy to see that such solution remains in $[0,1]$. 

\item For any fixed continuous flow of probability measures, the linearized martingale problem has a unique solution. By the linearized martingale problem we mean the following. For $\nu \in \mathcal{P}(\mathcal{C}([0, +\infty); [0,1]))$, and $x \in [0,1]$ put:
\begin{align*}
	b^{(\nu)} (x, t) \coloneqq \NLb(x, \nu_t),
\end{align*}
and for $t \ge 0$, define the operator $L^{(\nu)}_t$ acting on $\mathcal{C}^{2} ([0,1];\R)$ as:
\begin{align*}
	L^{(\nu)}_t g(x) \coloneqq b^{(\nu)}(x, t) g' (x)+ \dfrac{1}{2} x(1-x) g'' (x).
\end{align*}
\medskip
The linearized martingale problem is posed as follows: 
    \begin{center}
	\begin{minipage}{0.7 \textwidth}
find $\mu \in \mathcal{P} (\mathbb{D} ([0, +\infty); [0,1]))$ such that, $\mu_0= \nu_0$ and with $\dif_t (\omega)= \omega_t$ the canonical process, for every $g \in \mathcal{C}^\infty ([0,1])$, the process:
\begin{align} \label{eq:FROZEN_MP}
    \tag{FMP}
	M^{(\nu;g)}_t: = g (\dif_t)- g (\dif_0)- \int_{0}^{t} L^{(\nu)}_s g (\dif_s) \dd s, \quad (t \ge 0)
\end{align}
is a $\FF_t$-martingale under $\mu$. 
\end{minipage}
\end{center}
\medskip
\end{enumerate}

These conditions tell us that for a fixed flow $(\nu_t)$, there exists a unique solution $\mu^{\nu}$ for:
\begin{align} \label{eq:fixed_non_linear_FP}
	\dfrac{\dd}{\dd t} \inner{\mu_t}{\varphi} = \inner{\mu_t}{L_{\nu_t} \varphi}; \quad \mu_0= \overline \mu_0;
\end{align}
and that there exists a unique strong solution $\difNu$ for the SDE \eqref{eq:linearized-sde} for the flow $(\nu_t)$. Notice that by It\^o's Lemma, the law of $\difNu$ solves \eqref{eq:fixed_non_linear_FP}. Then, by uniqueness, $\mu^{\nu}$ is equal to the law of $\difNu$. 

Let us consider $\bar{\nu}$ any solution of \eqref{eq:NLMP}. Then, inserting $\bar{\nu}$ in \eqref{eq:fixed_non_linear_FP} we obtain $\mu^{\bar\nu}$ the unique solution of: 
\begin{equation}\label{eq:FKwithNonLinearFlow}
\begin{aligned} 
	\dfrac{\dd}{\dd t} \inner{\mu^{\bar\nu}_t}{\varphi} = \inner{\mu^{\bar\nu}_t}{L_{\bar\nu_t} \varphi}; \quad \mu_0= \overline \mu_0,
\end{aligned}    
\end{equation}
and inserting $\bar{\nu}$ in \eqref{eq:linearized-sde} we obtain $X^{\bar\nu}$.
 But $\bar{\nu}$ is also a solution of \eqref{eq:FKwithNonLinearFlow}, so we must have by uniqueness $\bar{\nu} =\mu^{\bar\nu}$, and then $\bar{\nu}$ is equal to the law of $X^{\bar\nu}$. Consequently, $X^{\bar\nu}$ is a solution of \eqref{eq:MV}, which has a unique weak solution; so $\bar{\nu}$ is equal to $\bar{\mu}$, the law of the $\nlDif$ process defined in \eqref{eq:MV}. Hence \eqref{eq:NLMP} has a unique solution.
\end{proof}

\subsubsection{On the existence and uniqueness of solutions for the McKean-Vlasov Wright-Fisher equation}

In Proposition~\ref{prop:limits-of-non-linear-markov-chain-solve-NLMP} we assume existence and uniqueness of solutions (in the sense of Definition~\ref{def:weak_solution}) to the McKean--Vlasov Wright--Fisher equation. In Appendix~\ref{ss:existence_MV}, we establish existence and uniqueness under mild assumptions that are satisfied, for example, in the case of non-linear Wright--Fisher diffusions with pure selection, where the selection rate depends on the instantaneous marginal law. The proof is based on relative entropy estimates between the laws of two solutions of the frozen Wright--Fisher diffusion (as in~\eqref{eq:linearized-sde}), and we believe that it is of independent interest. However, since this result is not essential for the remainder of the paper, and to improve the readability of the main text, we have chosen to present it in a separate section.

For the model studied in Section~\ref{sec:case_study}, existence and uniqueness can be obtained from \cite[Theorem~3.1]{huang2022path}.

\section{The deme level: propagation of chaos}\label{sec:prop_chaos_results}
In this section, we set up the two-level model that describes the evolution of the fractions of allele-$0$ pathogens within individuals across a large host population. As announced in Section~\ref{sec:overview-results}, the main results in this context are Theorems~\ref{propchaos} and~\ref{th:prop_chaos_2}. These theorems are propagation-of-chaos results that compare the pathogen dynamics in the full two-level model with the evolution of pathogens within a single individual governed by a non-linear dynamical system.

\subsection{The two-level model and a coupling for the non-linear dynamics}
We consider a population consisting of $M\in\N$ hosts, each one being infected by a pathogen. As in the one-level dynamics, the pathogen population consists of allele-$0$ and allele-$1$ individuals. In addition to $M$, the other parameters of the model are $\delta\in[0,1]$ and a measurable function $\NLpf:[0,1]\times\Ps\to[0,1]$. To describe the dynamics of the population it is convenient to define $\hat{\mu}_M(\bf z)$, for ${\bf z}\coloneqq(z_1,\ldots,z_M)\in[0,1]^M$, as the empirical measure of the points $z_1,\ldots,z_M$, i.e.
$$\hat{\mu}_M({\bf z})\coloneqq\frac{1}{M}\sum_{j=1}^M\delta_{z_j}.$$

We denote by $\psdi{M}_{n,i}$ the fraction of allele-$0$ pathogens in the $i$-th host, $i\in\{1,\ldots,M\}$, at generation $n\in\N$, and we set $\psd{M}_{n}\coloneqq(\psdi{M}_{n,1},\ldots,\psdi{M}_{n,M})$. The evolution of the system is then encoded by the $[0,1]^{M}$-valued Markov chain $\psd{M} \coloneqq (\psd{M}_{n})_{n \ge 0}$ evolving as follows. If $\psd{M}_n={\bf z}$, then $\psd{M}_{n+1}$ is distributed according to
\begin{align}\label{eq:product}
	P_{\delta,\NLpf(z_1,\hat{\mu}_M({\bf z}))}(z_1,\cdot) \otimes P_{\delta,\NLpf(z_2,\hat{\mu}_M({\bf z}))}(z_2,\cdot)\otimes \ldots\otimes P_{\delta,\NLpf(z_M,\hat{\mu}_M({\bf z}))}(z_M,\cdot) 
\end{align}
Note that the pathogen dynamics is correlated across hosts only through their empirical measure.  
\begin{remark}[A two-level coupling]\label{app:particle-system-coupling}
Let us now couple the pathogen-host system $\psd{M}$ described above to a system of $M$ independent non-linear chains driven by $\mathbb{Q}^{(\delta,\NLpf)}$, which we denote by ${\overline{\mathbf{Z}}}^{(\delta,\NLpf,M)}\coloneqq(\nl Z^{(\delta,\NLpf)}_{n,1},\ldots,\nl Z^{(\delta,\NLpf)}_{n,M})_{n\ge 0}$ .  

To do this, assume that we have constructed both systems up to generation $n$ and consider a collection $\{U^{(j)}, U_1^{(j)}, U_2^{(j)}, \, j \ge 1\}$ of $\unif$ random variables independent between them and independent from the evolution of both systems up to generation $n$. For each $i\in\{1,\ldots,M\}$ set 
$$p_i\coloneqq \NLpf(\psdi{M}_{n,i},\hat{\mu}_M(\psd{M}_n))\quad\textrm{and}\quad q_i\coloneqq \NLpf(\nl Z^{(\delta,\NLpf)}_{n,i},\mu_{n,i}),$$
where $\mu_{n,i}$ is the the law of $\nl Z^{(\delta,\NLpf)}_{n,i}$. Further, if $p_i\leq q_i$, we set
        \begin{align*}
            \psdi{M}_{n+1,i} &= \upx_{p_i,q_i,\delta}(\psdi{M}_{n,i},U^{(i)},U^{(i)}_1,U^{(i)}_2),\\
            \nl Z^{(\delta,\NLpf)}_{n+1,i} &= \upy_{q_i,\delta}(\nl Z^{(\delta,\NLpf)}_{n,i},U^{(i)},U^{(i)}_1),
        \end{align*}
otherwise, we set
        \begin{align*}
          \psdi{M}_{n+1,i} &= \upy_{p_i,\delta}(\psdi{M}_{n,i},U^{(i)},U^{(i)}_1),\\
          \nl Z^{(\delta,\NLpf)}_{n+1,i} &= \upx_{q_i,p_i,\delta}(\nl Z^{(\delta,\NLpf)}_{n,i},U^{(i)},U^{(i)}_1,U^{(i)}_2).
        \end{align*}
According to Lemma \ref{uple}, $(\nl Z^{(\delta,\NLpf)}_{n,i})_{n\ge 0}$ is, for each $i\in\{1,\ldots,M\}$, a non-linear chain driven by $\mathbb{Q}^{(\delta,\NLpf)}$. Moreover, $(\nl Z^{(\delta,\NLpf)}_{n,1})_{n\ge 0},\ldots,(\nl Z^{(\delta,\NLpf)}_{n,M})_{n\ge 0}$ are independent.  It also follows from Lemma \ref{uple} that the system $\psd{M}$ has the desired distribution and that
	\begin{align}\label{colsnl}
		\E \left[\vert \psdi{M}_{n+1,1}-  \nl Z^{(\delta,\NLpf)}_{n+1,1}\vert\right] &\le \bigg (1- \dfrac{\delta}{2} \bigg ) \E\left[\vert \psdi{M}_{n,1}-  \nl Z^{(\delta,\NLpf)}_{n,1}\vert\right]\nonumber\\
  &+ \dfrac{\delta}{2}\E \left[\vert \NLpf(\psdi{M}_{n,1},\hat{\mu}_M(\psd{M}_n))-  \NLpf(\nl Z^{(\delta,\NLpf)}_{n,1},\mu_{n,1})\vert \right].
	\end{align}
\end{remark}
\subsection{Propagation of chaos I}
The aim of this section is to provide a bound for the distance between the law, at generation $n$, of the state of a given particle in the pathogen-host system  and the law of an appropriate non-linear chain. Since in this section the parameter $\delta$ is fixed, we will drop it from the notation. 

A important ingredient in the proof is the following result, which is a direct consequence of \cite[Thm.1, p.709]{fournier2015}. It provides non-asymptotic estimates for the convergence in Wasserstein distance of the empirical laws (and thus, can be regarded as a quantitative version of Glivenko-Cantelli's theorem ). Similar results for the Wasserstein-2 distance in different dimensions can be found in \cite{carmona2019}.  

\begin{proposition}\label{prop:glivenkocantelli}
	There exist a constant $C_0 > 0$ such that, for any sequence of i.i.d. random variables $(\xi_i)_{i \ge 1}$ on $[0,1]$ with common law $\mu \in \mathcal{P}([0,1])$, we have
	\begin{align*}
		\E [\Was (\hat{\mu}_{M}(\mathbf{\xi}), \mu)] \le \frac{C_0}{\sqrt{M}},
	\end{align*} 
    where $\hat{\mu}_{M}(\mathbf{\xi})\coloneqq M^{-1}\sum_{i=1}^M\delta_{\xi_i}$ is the empirical law induced by a sample of size $M$ of the sequence $(\xi_i)_{i \ge 1}$.  
\end{proposition}

Now we proof Theorem \ref{propchaos}, which state the chaos propagation property for the discrete time chains, meaning the convergence of the truly Markovian particle system to the non linear Markov chain.

\begin{proof}[Proof of Theorem \ref{propchaos}]
	We use the coupling described in Remark \ref{app:particle-system-coupling} with the same initial value for both systems; the distribution of the initial value is $\mu_0^{\otimes M}$. Let $\mu_n$ be the distribution of $\nl Z^{(\NLpf)}_{n,1}$. 
By definition of the $\Was$-distance, we have
$$\Was\Big(\Law (\psodi{M}_{n,1}), \Law (\nl Z^{(\NLpf)}_n )\Big)\le\E \left[\vert \psodi{M}_{n,1}-  \nl Z^{(\NLpf)}_{n,1}\vert\right].$$
Furthermore, since $\NLpf$ is $(L_1,L_2)$-Lipschitz,  Eq. \eqref{colsnl} delivers
	\begin{align}\label{chaos_pre}
		\E \left[\vert \psodi{M}_{n,1}-  \nl Z^{(\NLpf)}_{n,1}\vert\right]\le& \bigg (1- \dfrac{\delta(1-L_1)}{2} \bigg ) \E\left[\vert \psodi{M}_{n-1,1}-  \nl Z^{(\NLpf)}_{n-1,1}\vert\right]\nonumber\\
  &+ \dfrac{\delta L_2}{2}\E \left[\Was (\hat{\mu}_M(\psod{M}_{n-1}), \mu_{n-1}) \right].
	\end{align}
Note now that
	\begin{align*}
		\Was (\hat{\mu}_M(\psod{M}_{n-1}), \mu_{n-1}) \le \Was (\hat{\mu}_M(\psod{M}_{n-1}),\hat{\mu}_M({\overline{\mathbf{Z}}}^{(\NLpf,M)}_{n-1})+ \Was(\hat{\mu}_M({\overline{\mathbf{Z}}}^{(\NLpf,M)}_{n-1}), \mu_{n-1}).
	\end{align*} 
	For the first term, we have
	\begin{align*}
		\Was (\hat{\mu}_M(\psod{M}_{n-1}),\hat{\mu}_M({\overline{\mathbf{Z}}}^{(\NLpf,M)}_{n-1}))
		&\le \dfrac{1}{M} \sum_{j=1}^{M} \vert \psodi{M}_{n-1,j}-  \nl Z^{(\NLpf)}_{n-1,j}\vert.
	\end{align*}
	Moreover, the distribution of $\vert \psodi{M}_{n-1,j}-  \nl Z^{(\NLpf)}_{n-1,j}\vert$ is the same for all $j$. Hence,	
	\begin{align*}
		\E\left[\Was (\hat{\mu}_M(\psod{M}_{n-1}),\hat{\mu}_M({\overline{\mathbf{Z}}}^{(\NLpf,M)}_{n-1})) \right]\le \E\left[\vert \psodi{M}_{n-1,1}-  \nl Z^{(\NLpf)}_{n-1,1}\vert \right].
	\end{align*}
	Moreover, using the independence of the non-linear chains and Proposition \ref{prop:glivenkocantelli}, we obtain
	\begin{align*}
		\E\left[\Was(\hat{\mu}_M({\overline{\mathbf{Z}}}^{(\NLpf,M)}_{n-1}), \mu_{n-1})\right] \le \dfrac{C_0}{\sqrt{M}}.
	\end{align*}
	Plugging the obtained bounds in \eqref{chaos_pre} yields
	\begin{align*}
		\E \left[\vert \psodi{M}_{n,1}-  \nl Z^{(\NLpf)}_{n,1}\vert\right] &\le \bigg (1- \dfrac{\delta(1-L_1-L_2)}{2} \bigg )\E \left[\vert \psodi{M}_{n-1,1}-  \nl Z^{(\NLpf)}_{n-1,1}\vert\right] +\dfrac{\delta L_2}{2}\dfrac{C_0}{\sqrt{ M}}.
	\end{align*}
        Iterating this inequality and using that both systems start with the same initial value we get
        \begin{align}\label{eq:critical}
		\E \left[\vert \psodi{M}_{n,1}-  \nl Z^{(\NLpf)}_{n,1}\vert\right]&\le 
        \dfrac{L_2}{1-L_1-L_2}\dfrac{C_0}{\sqrt{ M}}\left(1- \Bigg (1- \dfrac{\delta(1-L_1-L_2)}{2} \Bigg )^n\right).
	\end{align}
	The result follows since $L_1+L_2<1$.
\end{proof}
\begin{remark}[On the $(L_1,L_2)$-Lipschitz condition]\label{l1l2c}
The condition $L_1+L_2<1$ is only used in the proof of Theorem \ref{propchaos} in the iteration step leading to Eq. \eqref{eq:critical}. If $L_1+L_2=1$, the iteration steps yields
$$\E \left[\vert \psodi{M}_{n,1}-  \nl Z^{(\NLpf)}_{n,1}\vert\right]\le \frac{\delta L_2}{2}\frac{C_0 \,n}{\sqrt{M}}.$$
Similarly, if $L_1+L_2>1$, the iteration step yields
        \begin{align*}
		\E \left[\vert \psodi{M}_{n,1}-  \nl Z^{(\NLpf)}_{n,1}\vert\right]&\le 
        \dfrac{L_2}{L_1+L_2-1}\dfrac{C_0}{\sqrt{ M}}\left(\Bigg (1+ \dfrac{\delta(L_1+L_2-1)}{2}\Bigg )^n-1\right).
	\end{align*}
Hence, we obtain linear exponential bounds in $n$ for the corresponding $\Was$-distance. This contrasts with the uniform bound obtained in Theorem \ref{propchaos} under the assumption $L_1+L_2<1$. 
\end{remark}
\subsection{Propagation of chaos II: scaled processes}
In this section we prove Theorem \ref{th:prop_chaos_2} which is in spirit analogous to Theorem \ref{propchaos}, but for duly scaled processes in continuous time. Let us recall that, we associate to each $N\in\N$ a measurable function $\NLpfn:[0,1]\times\Ps\to[0,1]$ and we consider a strictly increasing function $M:\N\to\N$. For each $N\in\N$, we consider the host-pathogen system with $M(N)$ hosts described by the Markov chain $\psdN{M(N)}$ with $\delta_N\coloneqq \sqrt{3/N}$. Similarly, for each $N \in\N$, we consider the non-linear chain $\nl Z^{(\delta_N,\NLpf_N)}$ driven by $\mathbb{Q}^{(\delta_N,\NLpfn)}$. The scaled processes $ \sPSsM$ and $\sNLchainN$ were defined \eqref{def:systems} as \(
\sPSsM_t=  \psdN{M(N)}_{\lfloor Nt\rfloor}\) and \(\quad\sNLchainN_t= \nlZn_{\lfloor Nt\rfloor},\, t\geq 0\). Now, we have all the ingredients to state our second propagation of chaos result, which specialize to the class of functions used in Section \ref{section:MVWF}.

\begin{proof}[Proof of Theorem \ref{th:prop_chaos_2}]
Thanks to the assumption on $\NLb$, $\NLpfn$ is $(1+2L_1/\sqrt{3N},2L_2/\sqrt{3N})$-Lipschitz. Combining this with Remark \ref{l1l2c} , we obtain
\begin{align*}
\E \left[\vert \psodi{M}_{n,1}-  \nl Z^{(\NLpf)}_{n}\vert\right] & \le \dfrac{C_0L_2}{\sqrt{M(N)}} \dfrac{\Big(1+\frac{L_1+L_2}{N}\Big)^n-1}{L_1+L_2},
\end{align*}
and the result follows choosing \(n=\lfloor Nt \rfloor\).
\end{proof}

In the next uniform-in-time propagation of chaos result, we remove the condition imposed in Theorem \ref{th:prop_chaos_2} about the particular form of $\NLpfn$. Instead, we impose a stronger Lipschitz condition on $\NLpfn$ and an additional condition on the function $M$.

\begin{proposition}\label{prop:uniform_prop_chaos}
Assume that for each $N\in\N$ the function $\NLpfn$ is $(L_1(N),L_2(N))$-Lipschitz for some $L_1(N),L_2(N)>0$ with $L_1(N)+L_2(N)<1$. Let $\mu_0^{(N)}$ be the law of $\sNLchainN_0$ and assume that $\sPSsM_0\sim (\mu_0^{(N)})^{\otimes M(N)}$. Then For $N$ sufficiently large, we have
\begin{align*}
\Was (\Law (\sPSsM_{t,1}), \Law (\nl X^{(N)}_t)) \le C_0 \dfrac{L_2(N)}{1- L_1(N)-L_2(N)}\dfrac{1}{\sqrt{M(N)}}.
\end{align*}
In particular, under the additional assumption that
\begin{align*}
M(N) \gg \left (\dfrac{L_2(N)}{1-L_1(N)-L_2(N)}\right)^2, 
\end{align*}
we have 
\begin{align*}
\lim\limits_{N\to\infty}\sup_{t\geq 0}\Was (\Law (\sPSsM_{t,1}), \Law (\nl X^{(N)}_t))=0.
\end{align*}
\end{proposition}

\begin{proof}[Proof of Proposition \ref{prop:uniform_prop_chaos}] The second claim follows directly from the first one. The first one follows using \eqref{eq:critical} with $L_1=L_1(N)$, $L_2=L_2(N)$, $M=M(N)$ and $n=\lfloor Nt\rfloor$.
\end{proof}

\section{Self-stabilizing Wright--Fisher populations with mutation: a case study. } \label{sec:case_study}

Let us recall the McKean--Vlasov SDE from Section \ref{sec:overview-results}
\begin{equation}\tag{\ref{eq:wf_mut}}
\begin{aligned}
\dd\self_t &= \left(-\theta_1 \self_t + \theta_0(1-\self_t) - \gamma(\self_t-\E[\self_t])\right) \dd t  + \sqrt{\self_t (1-\self_t)}\dd W_t,\\
 \nl X_0 &\sim \overline \mu_0. 
\end{aligned}
\end{equation}
Here $\self_t$ can be interpreted as the fraction of the population of type $0$ at time $t$. The population is subject to mutations at rates $\theta_0 >0,\theta_1 > 0$, but also to a \emph{pushing towards the mean} effect modulated by $\gamma\in(0,1)$. The existence and uniqueness of solutions for \eqref{eq:wf_mut} is not covered by the results in Section \ref{ss:existence_MV}, but follow from \cite[Theorem 3.1]{huang2022path}. 

We will now prove Theorem \ref{prop:invariant-self} about the ergodic properties of the process $\self$ using the tools developed in the previous sections.

\begin{proof}[Proof of Theorem \ref{prop:invariant-self}] The proof consists of four main steps.

\noindent\textbf{Step one - Ergodicity of the approximating non-linear chains:}
Let us consider
$\delta_N\coloneqq \sqrt{3/N}$ and
\begin{align} \label{eq:pn}
     \NLpfn(x,\mu)&\coloneqq x + \frac{2}{\sqrt{3N}}\left(-\theta_1x+\theta_0(1-x) - \gamma(x-\mom{\mu}{1}) \right),\quad x\in[0,1].
\end{align}
On can show that $\NLpfn(x)\in[0,1]$ for all $(x,\mu)$ as long as $N\geq 4(\theta_0+\theta_1+\gamma)^2/3$. Hence the non-linear Markov chain  $\nlZn$ is well defined. 
Moreover, a straightforward computation shows that
\begin{align}\label{eq:lips_constants}
    \vert \NLpfn(x,\mu) - \NLpfn(y,\nu)\vert \leq \left(1-\frac{2(\theta_0+\theta_1+\gamma)}{\sqrt{3N}}\right)|x-y| + \frac{2\gamma}{\sqrt{3N}}\Was(\mu,\nu).
\end{align}
In other words, $\NLpfn$ is $(L_1(N),L_2(N))$-Lipschitz with $L_1(N)=1-2((\theta_0+\theta_1+\gamma)/\sqrt{3N}$ and $L_2(N)=2\gamma/\sqrt{3N}$. In particular, we have
$$
    L_1(N)+L_2(N) = 1-\frac{2(\theta_1+\theta_0)}{\sqrt{3N}}<1\;\quad\text{for all }N\geq 4(\theta_0+\theta_1)^2/3.
$$
Thus, thanks to Theorem \ref{th:non_linear_ergodic}, the non-linear Markov chain  $\nlZn$ admits a unique invariant distribution $\eta^{(N)}$ and
\begin{align}\label{eq:rates}
\Was (\Law(\nlZn_n), \eta^{(N)}) \le  \Bigg (  1- \dfrac{\theta_1+\theta_0}{N}\Bigg)^n \Was (\Law (\nlZn_0), \eta^{(N)}).
\end{align}

\noindent\textbf{Step two - Convergence of the sequence of invariant distributions:}
A general observation about non-linear Wright--Fisher chains is needed. If $\eta$ is the stationary distribution of $\NLchainP$, then the first moment of $\eta$ can be expressed as
    \begin{align}\label{eq:nl_expectation}
\mom{\eta}{1}= \E_\eta [\NLchainP_1] =\E_\eta [\NLpf(\NLchainP_0,\eta)].
    \end{align}
In our case, this yields
\begin{align*}
\mom{\eta^{(N)}}{1} &= \dfrac{\theta_0}{\theta_1+ \theta_0}\eqqcolon m.
\end{align*}
Thus, defining
$$\pf_{\delta_N} (x)\coloneqq x+\frac{2\delta_N}{3}\left(\bar{\theta}_0(1-x)+\bar{\theta_1}x\right),
$$
where $\bar{\theta}_0=\theta_0(1+\gamma/(\theta_0+\theta_1))$ and $\bar{\theta}_1=\theta_1(1+\gamma/(\theta_0+\theta_1))$, a straightforward computation shows that $\NLpfn(x,\eta^{(N)})=\pf_{\delta_N} (x)$. Consequently, the invariant measure of the non-linear chain $\nlZn$ coincides with the invariant measure of the Markov chain $Z^{(\delta_n,\pf_{\delta_N})}$. Therefore, according to Remark \ref{rem:moment_nutation}, the sequence $(\eta^{(N)})_{N \ge 1}$ converges weakly to the $\betadist(2\bar{\theta}_0,2\bar{\theta}_1)$ distribution.

\noindent\textbf{Step three - Existence and uniqueness of an invariant distribution for $\self$:}
It is well-known that the Wright--Fisher diffusion with mutation
\begin{align*}
    \dd X_t = \left(-\theta_1 X_t + \theta_0(1-X_t) \right) \dd t + \sqrt{X_t (1-X_t)}\dd W_t, 
\end{align*}
has a unique invariant distribution $\betadist (2 \theta_0, 2\theta_1)$ (see e.g. \cite{durrett2002}). Now, fix a $\Was$-Lipschitz function $h:\mathcal{P}([0,1]) \mapsto [0,1]$ and consider the family of SDE parameterized by $\mu \in \mathcal{P}(0,1)$:
\begin{equation}\label{eq:SDE_mu}
\begin{aligned}
    \dd\difMu_t =& \left(-(\theta_1 + \gamma (1-h(\mu))) \difMu_t + (\theta_0+ \gamma h(\mu))(1-\difMu_t) \right) \dd t \\
    &\quad + \sqrt{\difMu_t (1-\difMu_t)}\dd W_t. 
\end{aligned}
\end{equation}

Thus, according to the previous discussion, $\difMu$ has a unique invariant distribution $T\mu$, which is given by the $\betadist (2(\theta_0+ \gamma h(\mu)), 2 (\theta_1+ \gamma (1- h(\mu))))$ distribution. Using Proposition \ref{prop:was_beta}, it can be shown that 
\begin{align*}
\Lip(T) \le \dfrac{\gamma \Lip (h)}{\theta_1+\theta_0+ \gamma}.
\end{align*}
As a consequence of this result (and a fixed point argument that mimics the proof of theorem \ref{th:existence}), for every $\gamma \ge 0$, the McKean--Vlasov Wright--Fisher SDE \eqref{eq:wf_mut} admits a unique invariant distribution which is the $\betadist \left ( 2 \theta_0 (1+ \gamma/(\theta_1+ \theta_0)), 2\theta_1 (1+ \gamma/(\theta_1+\theta_0)) \right)$ distribution.

\noindent\textbf{Step four - Interpolation:}

For $N \geq N_0\coloneqq4(\theta_0+\theta_1+\gamma)^2/3$, we set $\sNLchainN_t\coloneqq\nlZn_{\lfloor Nt \rfloor}$, for $t \ge 0$. Assume that $\nlZn_0= \self_0$. Set $\eta^{(N)}_t= \Law (\sNLchainN_t)$ and $\eta_t= \Law (\self_t)$. Fix $t \ge 0$ and $\varepsilon > 0$. We have
\begin{align*}
\Was (\eta_t, \eta) \le \Was (\eta_t, \eta^{(N)}_t) + \Was (\eta^{(N)}_t, \eta^{(N)}) + \Was (\eta^{(N)}, \eta).
\end{align*}
We claim now that, by Theorem \ref{th:convergence_theorem}, there exists $N_1\ge N_0$ such that for every  $N \ge N_1$ the first term is smaller than $\varepsilon/4$ (the $N_1$ may depend on $t$). Indeed, for any $T \ge t$, we have proved above, that in the Skorokhod space the sequence of laws $P^{(N)}$ of the processes $(\sNLchainN_{[0, T]})$ converges weakly to the law $P$ of the process $\self_{[0, T]}$. Since the projection map $\pi_t: \mathbb{D} ([0, T]; [0,1]) \mapsto [0,1]$ is $P$-a.s. continuous, we have that $P^{(N)} \circ \pi_t^{-1} \to P \circ \pi_t^{-1}$ weakly. Since $\Was$ metrizes the weak convergence on $[0,1]$, the claim follows.  

Now, by Step two, there exists $N_2 \ge N_1$ such that for every $N \ge N_2$ the third term above is smaller than $\varepsilon/4$. As for the second term, in virtue of the estimate \eqref{eq:rates}, for some $N_3 \ge N_2$ and for every $N \ge N_3$:
\begin{align*}
\Was (\eta^{(N)}_t, \eta^{(N)}) &\le \exp(-(\theta_1+\theta_0)t)\Was (\eta^{(N)}_0, \eta^{(N)}) + \varepsilon/4\\
&\le \exp(-(\theta_1+\theta_0)t)\Was (\eta_0, \eta) + 2\varepsilon/4.
\end{align*}
The result follows since $\varepsilon$ was chosen arbitrarily. 
\end{proof}
\begin{remark}
 Observe that the mean of the invariant measure is $m\coloneqq\theta_0/(\theta_1+ \theta_0)$, and thus it is independent of $\gamma$. Moreover, since the second moment of the invariant distribution equals
$$m\dfrac{2(\theta_0+ \gamma m)+1}{2 (\theta_1+\theta_0+ \gamma)+1},$$ we conclude that the variance of the invariant distribution of the McKean--Vlasov Wright--Fisher SDE is decreasing in $\gamma\geq 0$. In particular, the non-linearity has, as the title of this section suggests, a stabilizing effect . See Fig. \ref{fig:fig2} for an illustration of the density of the invariant measure for different values of $\gamma$.

\end{remark}

We close this section with the proof of Theorem \ref{prop_chaos_sswf}, which states that any individual particle of the Markovian system $\sPSsM$ approximates the dynamics of the McKean--Vlasov Wright--Fisher diffusion \eqref{eq:wf_mut}. 

\begin{proof}[Proof of Theorem \ref{prop_chaos_sswf}]
Note first that
\begin{align*}
\Was (\Law((\sPSsM_{t, 1})), \Law (\self_t))\le&  \Was (\Law((\sPSsM_{t, 1})), \Law (\sNLchainN_t))\\
& +\Was (\Law (\sNLchainN_t),\Law(\self_t)).
\end{align*}
Hence, Proposition \ref{prop:uniform_prop_chaos} delivers
\begin{align*}
\limsup_{N\to\infty}\sup_{t\ge 0}\Was (\Law((\sPSsM_{t, 1})), \Law (\self_t))\le&\limsup_{N\to\infty}\sup_{t\ge 0}\Was (\Law (\sNLchainN_t),\Law(\self_t)).
\end{align*}
We claim now that for every $T>0$, we have 
\begin{align*}
\lim_{N \to \infty} \sup_{t \in [0, T]} \Was (\Law (\sNLchainN_t), \Law (\self_t))=0.
\end{align*}
Assume the  claim is true.  Clearly, for every $T \ge 0$:
\begin{align*}
\sup_{t \ge T} \Was (\Law (\sNLchainN_t), \Law (\self_t)) & \le \Was (\eta^{(N)}, \eta)+ \sup_{t \ge T}\Was (\Law (\sNLchainN_t), \eta^{(N)}) \\
& \qquad  \qquad \qquad+ \sup_{t \ge T} \Was (\eta, \self_t),  
\end{align*}
According to Remark \ref{rem:moment_nutation} the first term in the right-hand side converges to $0$ as $N\to\infty$. Moreover, thanks to \eqref{eq:rates} the second term can be made arbitrarily small by taking $T$ sufficiently large. Similarly, by Proposition \ref{prop:invariant-self} we can make the last term arbitrarily small for $T$ sufficiently large. Combining this with the claim yields the result. It remains to prove the claim.

Let $V^{(N)}_t$ be the piecewise-linear, continuous process given by:
\begin{align*}
V^{(N)}_t = \nlZn_{\lfloor Nt \rfloor} + (Nt- \lfloor Nt\rfloor) (\nlZn_{\lfloor Nt \rfloor+1}- \nlZn_{\lfloor Nt \rfloor}); \quad (t \ge 0). 
\end{align*}
Note  that $\vert V^{(N)}_t - \sNLchainN_t \vert \le N^{-1/2}$ for every $t \ge 0$. In particular,
\begin{align} \label{eq:ineq_3}
\Was (\Law (V^{(N)}_t), \Law (\sNLchainN_t)) \le N^{-1/2}.
\end{align}
Now, since the laws of the sequence $(\sNLchainN)_{N \ge 1}$ converges weakly to the law of $\self$, the same is true for the law of the processes $(V^{(N)})_{N \ge 1}$ (see for example \cite{ethierkurtz86}, Proposition 10.4, page 149), and moreover the convergence holds as probability measures on $\mathcal{C}([0, T];[0,1])$ (see \cite{ethierkurtz86}, Problem 25 at page 153). This space is Polish under the uniform norm, and of bounded diameter; so the Wasserstein distance $\Was (T; \cdot, \cdot)$ at display \eqref{def:was_C} metrizes the weak convergence on $\mathcal{P} (\mathcal{C}([0, T]; [0,1]))$ (see \cite{villani2009}, theorem 6.9, page 108). Since:
\begin{align*} 
\sup_{t \in [0, T]} \Was (\Law (V^{(N)}_t), \Law (\self_t)) \le \Was (T; \Law (V^{(N)}_{[0, T]}), \Law (\self_{[0, T]})),
\end{align*}
the left-hand goes to $0$ as $N \to \infty$. This, together with \eqref{eq:ineq_3}, shows the claim.
\end{proof}

\section{Closing remarks and future research}\label{section:discussion}
In this article we introduced a unified framework that recovers several
well-known models from population genetics and, in some cases,
reveals new moment dualities. We further extended our analysis to the
non-classical setting of non-linear Markov chains and showed that,
under suitable scaling, their dynamics converge to McKean--Vlasov
diffusions.

To illustrate the utility of our approach, we studied in detail
non-linear Markov chains associated with Wright--Fisher models whose
mutation rates depend on the empirical measure of the population, as
well as the limiting McKean--Vlasov diffusion.

The present framework raises several directions for further research.
A natural extension is to incorporate interactions at the
meta-community scale, thereby yielding a three-level model (see the
discussion and motivation in the Introduction). While the modeling
itself is straightforward within our setup, a full mathematical
analysis comparable to that carried out at the host and deme levels
presents new challenges; this is the topic of ongoing work.

The duality results in Section~\ref{section:host_level}
(Theorem~\ref{th:moment_duality_II}) were established in a purely
algebraic manner. In many classical models of population genetics,
moment duality reflects a coupling between forward and backward
dynamics via time reversal, often supported by a particle
representation of the forward process. In our situation, although a
particle picture exists for the backward process, its relation to the
forward dynamics remains unclear. Constructing such a coupling is part of our current research efforts.

Finally, a more conceptual challenge is to understand how the notion of
duality employed here may extend to non-linear Markov processes.
Although we are unable to provide concrete solutions at this stage, we
believe this to be a promising direction for future investigation.\\

\begin{acks}
We would like to thank the anonymous referee and the editor for their careful reading of the previous version and for the valuable comments and suggestions that substantially improved this article.

We are also grateful to Joaquín Fontbona for sharing his valuable insights, in particular regarding the existence of solutions to McKean--Vlasov equations with non-Lipschitz coefficients. L.V.\ thanks Ellen Baake and her group for kindly providing the facilities that made it possible to carry out part of this work.
\end{acks}

\ifArxiv
\noindent\textbf{Funding.}
F.C. was funded by the Deutsche Forschungsgemeinschaft (DFG, German Research Foundation) --- Project-ID 317210226 --- SFB 1283; C.J. was funded by Universidad de Santiago de Chile, DMCC --- Beca de Trabajo POC 2022-DMCCUSA215; H.O. was partially funded by ANID Exploraci\'on, project number 13220168-2023 \emph{Biological and quantum Open System Dynamics: evolution, innovation and mathematical foundations} and FONDECYT Regular Nº1242001; L.V. was funded by FONDECYT Iniciaci\'on, project number 11240158-2024 \emph{Adaptive behavior in stochastic population dynamics and non-linear Markov processes in ecoevolutionary modeling}.
\fi


\bibliographystyle{plain}

\bibliography{BIB_NLWF}

\begin{thebibliography}{10}

\bibitem{BCH17}
E.~Baake, F.~Cordero, and S.~Hummel.
\newblock A probabilistic view on the deterministic mutation--selection
  equation: dynamics, equilibria, and ancestry via individual lines of descent.
\newblock {\em J. Math. Biol.}, 77:795--820, 2018.

\bibitem{BEH23}
E.~Baake, L.~Esercito, and S.~Hummel.
\newblock Lines of descent in a moran model with frequency-dependent selection
  and mutation.
\newblock {\em Stochastic Processes and their Applications}, 160:409--457,
  2023.

\bibitem{BW18}
E.~Baake and A.~Wakolbinger.
\newblock Lines of descent under selection.
\newblock {\em Journal of Statistical Physics}, 172(1):156--174, Jul 2018.

\bibitem{benaimhurth2022}
M.~Ben\"aim and T.~Hurth.
\newblock {\em Markov chains on metric spaces: a short course}.
\newblock Springer Nature, Switzerland, 2022.

\bibitem{Ber09}
N.~Berestycki.
\newblock {\em Recent progress in coalescent theory}, volume~16 of {\em Ensaios
  Matem\'{a}ticos}.
\newblock Sociedade Brasileira de Matem\'{a}tica, Rio de Janeiro, 2009.

\bibitem{Bertoin2003}
J.~Bertoin and J-F. Le~Gall.
\newblock Stochastic flows associated to coalescent processes.
\newblock {\em Probab. Theory Relat. Fields}, 126(2):261--288, 2003.

\bibitem{Bertoin2005}
J.~Bertoin and J-F. Le~Gall.
\newblock Stochastic flows associated to coalescent processes. {II}.
  {S}tochastic differential equations.
\newblock {\em Ann. Inst. H. Poincar{\'e} Probab. Statist.}, 41(3):307--333,
  2005.

\bibitem{Bertoin2006}
J.~Bertoin and J-F. Le~Gall.
\newblock Stochastic flows associated to coalescent processes. {III}. {L}imit
  theorems.
\newblock {\em Illinois J. Math.}, 50(1-4):147--181, 2006.

\bibitem{billingsley1999}
P.~Billingsley.
\newblock {\em Convergence of probabilty measures}.
\newblock Wiley series in probability and statistics. John Wiley and Sons,
  1999.

\bibitem{bossy2019}
M.~Bossy, J.~Fontbona, and H.~Olivero.
\newblock Synchronization of stochastic mean field networks of hodgkin–huxley
  neurons with noisy channels.
\newblock {\em Journal of Mathematical Biology}, 78:1771--1820, 2019.

\bibitem{butkovsky2014}
O.~A. Butkovsky.
\newblock On ergodic properties of nonlinear markov chains and stochastic
  mckean--vlasov equations.
\newblock {\em Theory of Probability \& Its Applications}, 58(4):661--674,
  2014.

\bibitem{Cannings74}
C.~Cannings.
\newblock The latent roots of certain markov chains arising in genetics: A new
  approach, i. haploid models.
\newblock {\em Advances in Applied Probability}, 6(2):260--290, 1974.

\bibitem{carmona2019}
R.~Carmona and F.~Delarue.
\newblock {\em Probabilistic Theory of Mean Field Games with Applications, I}.
\newblock Probability Theory and Stochastic Modelling. Springer Cham, 2019.

\bibitem{CHS19}
F.~{Cordero}, S.~{Hummel}, and E.~{Schertzer}.
\newblock General selection models: Bernstein duality and minimal ancestral
  structures.
\newblock {\em Ann. Appl. Probab.}, 32:1499--1556, 2022.

\bibitem{DK99}
P.~Donnelly and T.~Kurtz.
\newblock Particle representations for measure-valued population models.
\newblock {\em Ann. Probab.}, 27:166--205, 1999.

\bibitem{durrett2002}
R.~Durrett.
\newblock {\em Probability models for DNA sequence evolution}.
\newblock Probability and its applications. Springer-Verlag, New York, 2002.

\bibitem{ethierkurtz86}
S.~Ethier and T.~Kurtz.
\newblock {\em Markov Processes: characterization and convergence}.
\newblock John Wiley and Sons., 1986.

\bibitem{fournier2015}
N.~Fournier and A.~Guillin.
\newblock On the rate of convergence in wasserstein distance of the empirical
  measure.
\newblock {\em Probab. Theory Relat. Fields}, 162:707--738, 2015.

\bibitem{Glad77}
K.~Gladstien.
\newblock Haploid populations subject to varying environment: The
  characteristic values and the rate of loss of alleles.
\newblock {\em SIAM Journal on Applied Mathematics}, 32(4):778--783, 1977.

\bibitem{Glad78}
K.~Gladstien.
\newblock The characteristic values and vectors for a class of stochastic
  matrices arising in genetics.
\newblock {\em SIAM Journal on Applied Mathematics}, 34(4):630--642, 1978.

\bibitem{GS18}
A.~{Gonz{\'a}lez Casanova} and D.~Span{\`o}.
\newblock Duality and fixation in {$\Xi$}-{W}right--{F}isher processes with
  frequency-dependent selection.
\newblock {\em Ann. Appl. Probab.}, 28:250--284, 2018.

\bibitem{Harris63}
T.~Harris.
\newblock {\em The theory of branching processes}, volume~6.
\newblock Springer Berlin, 1963.

\bibitem{hau23}
F.~Hausdorff.
\newblock Momentenprobleme f{\"u}r ein endliches intervall.
\newblock {\em Mathematische Zeitschrift}, 16:220--248, 1923.

\bibitem{HERRMANN20101215}
S.~Herrmann and J.~Tugaut.
\newblock Non-uniqueness of stationary measures for self-stabilizing processes.
\newblock {\em Stochastic Processes and their Applications}, 120(7):1215--1246,
  2010.

\bibitem{huang2022path}
X.~Huang and X.~Wang.
\newblock Path dependent mckean-vlasov sdes with h{\"o}lder continuous
  diffusion.
\newblock {\em Discrete \& Continuous Dynamical Systems-Series S}, 16(5), 2023.

\bibitem{IkedaWatanabe1989}
N.~Ikeda and S.~Watanabe.
\newblock {\em Stochastic Differential Equations and Diffusion Processes}.
\newblock North-Holland, Amsterdam, second edition, 1989.

\bibitem{jansen2014}
S.~Jansen and N.~Kurt.
\newblock On the notion(s) of duality for markov processes.
\newblock {\em Probability Surveys}, 11:59--120, 2014.

\bibitem{joffe1986}
A.~Joffe and M.~Metivier.
\newblock Weak convergence of sequences of semimartingales with applications to
  multitype branching processes.
\newblock {\em Advances in Applied probability}, 18(1):20--65, 1986.

\bibitem{kallenberg2002}
O.~Kallenberg.
\newblock {\em Foundations of Modern Probability}.
\newblock Springer-Verlag, New York, 2002.

\bibitem{karatzas91}
I.~Karatzas and S.~Shreve.
\newblock {\em Brownian motion and Stochastic Calculus}.
\newblock Springer Verlag, New York, 2nd. edition, 1991.

\bibitem{kolokoltsov2010nonlinear}
V.~Kolokoltsov.
\newblock {\em Nonlinear Markov processes and kinetic equations}, volume 182.
\newblock Cambridge University Press, 2010.

\bibitem{lindvall2002}
T.~Lindvall.
\newblock {\em Lectures on the coupling method}.
\newblock Dover Publications., 2002.

\bibitem{luo2014}
S.~Luo.
\newblock A unifying framework reveals key properties of multilevel selection.
\newblock {\em Journal of Theoretical Biology}, 341:41--52, 2014.

\bibitem{luo2017}
S.~Luo and J.~Mattingly.
\newblock Scaling limits of a model for selection at two scales.
\newblock {\em Nonlinearity}, 30(4):1682, mar 2017.

\bibitem{mckean66}
H.P. McKean.
\newblock A class of markov processes associated with non-linear parabolic
  equations.
\newblock {\em Proceedings of the National Academy of Sciences},
  56(6):1907--1911, 1966.

\bibitem{Martin99}
M.~M\"ohle.
\newblock The concept of duality and applications to markov processes arising
  in neutral population genetics models.
\newblock {\em Bernoulli}, 5(5):761--777, 1999.

\bibitem{neumann2023nonlinear}
B.A. Neumann.
\newblock Nonlinear markov chains with finite state space: invariant
  distributions and long-term behaviour.
\newblock {\em Journal of Applied Probability}, 60(1):30--44, 2023.

\bibitem{Pitman1999}
J.~Pitman.
\newblock Coalescents with multiple collisions.
\newblock {\em Ann. Probab.}, 27:1870--1902, 1999.

\bibitem{book:lipster-shiryayev-vol1}
A.~N.~Shiryayev R.~S.~Liptser.
\newblock {\em Statistics of Random Processes I: General Theory}.
\newblock Stochastic Modelling and Applied Probability №5. Springer, 1977.

\bibitem{book:lipster-shiryayev-vol2}
A.~N.~Shiryayev R.~S.~Liptser.
\newblock {\em Statistics of Random Processes II: Applications}.
\newblock Stochastic Modelling and Applied Probability №6. Springer, 1978.

\bibitem{revuzyor1999}
D.~Revuz and M.~Yor.
\newblock {\em Continuous martingales and Brownian motion}.
\newblock Springer Verlag, 3rd edition, 1999.

\bibitem{saburov2016ergodicity}
M.~Saburov.
\newblock Ergodicity of nonlinear markov operators on the finite dimensional
  space.
\newblock {\em Nonlinear Analysis: Theory, Methods \& Applications},
  143:105--119, 2016.

\bibitem{Sa99}
S.~Sagitov.
\newblock The general coalescent with asynchronous mergers of ancestral lines.
\newblock {\em J. Appl. Probab.}, 36:1116--1125, 1999.

\bibitem{stubbendieck2016}
R.~Stubbendieck, C.~Vargas-Bautista, and P.~Straight.
\newblock Bacterial communities: Interactions to scale.
\newblock {\em Front Microbiol.}, 2016.

\bibitem{tysbakov2009}
A.~Tysbakov.
\newblock {\em Introduction to nonparametric estimation}.
\newblock Springer series in statistics. Springer, 2009.

\bibitem{urban2008}
M.~Urban, M.~Leibold, P.~Amarasekare, L.~{De Meester}, R.~Gomulkiewicz,
  M.~Hochberg, C.~Klausmeier, N.~Loeuille, C.~{de Mazancourt}, J.~Norberg,
  J.~Pantel, S.~Strauss, M.~Vellend, and M.~Wade.
\newblock The evolutionary ecology of metacommunities.
\newblock {\em Trends in Ecology and Evolution}, 23(6):311--317, 2008.

\bibitem{videla2023}
L.~Videla.
\newblock A transport process on graphs and its limiting distributions.
\newblock {\em Journal of Applied Probability}, 60(1):341–357, 2023.

\bibitem{villani2009}
C.~Villani.
\newblock {\em Optimal transport: old and new}.
\newblock Grundlehren der mathematischen Wissenschaften. Springer, 2009.

\end{thebibliography}

\appendix

\section{Auxiliary results}\label{section:appendix}
\subsection{Further results for Section \ref{ss3.1}}
\label{subsection:auxiliary_results}

\begin{proposition}\label{prop:resumen}
	For any $ \delta \in (0,1], \p \in [0,1]$ it holds:
	\begin{enumerate}[label=\roman*),ref=\ref{prop:resumen} part (\roman*),start=1]
	\item $\beta_{\delta, \p}$ is absolutely continuous with respect to the Lebesgue measure and its density satisfies the integral equation:
\begin{align}\label{eq:integral_equation}
	g(z)= \dfrac{1-\p}{\delta} \int_{z}^{\frac{z}{1-\delta}\wedge 1} \dfrac{g(x)}{x} \dd x + \dfrac{\p}{\delta}  \int_{\frac{z-\delta}{1-\delta}\vee 0}^z \dfrac{g(x)}{1-x}\dd x.
\end{align}
	\item  \label{prop:moments}
	For every $k \ge 0$:
	\begin{align} \label{eq:moments}
		\mom{\beta_{\delta, p}}{k}\left ( 1- \dfrac{1}{\delta}\int_{0}^{\delta} (1-x)^k \dd x\right) = \dfrac{\p}{\delta}\sum_{j=0}^{k-1} \binom{k}{j}\mom{\beta_{\delta, p}}{j} \int_{0}^{\delta} (1-x)^{j}x^{k-j}\dd x.		
	\end{align} 
	In particular,
	\begin{align}\label{eq:1_2_moments}
		\mom{\beta_{\delta, p}}{1}= \p; \qquad \text{var}(\beta_{\delta, p}) = \dfrac{\delta\,\p(1-\p)}{3-\delta}. 
	\end{align}
	\item \label{prop:distance-invartiant} For all $\p,q\in[0,1]$:
	\begin{align} \label{eq:was_distance}
		\Was (\beta_{\delta, \p}, \beta_{\delta, q}) = \vert \p-q \vert.
	\end{align}
	\end{enumerate}	
	\end{proposition}

\begin{proof} 
\begin{enumerate}[label=\roman*)]
\item Let $h \ge 0$ be a bounded measurable function. Using the invariance of $\beta_{\delta,\p}$ under $P_{\delta,\p}$ and the definition of $P_{\delta,\p}$
	\begin{align*}
		\beta_{\delta, \p} h &= \beta_{\delta, \p} P_{\delta, \p} h \\
		&= \int_{[0,1]} \beta_{\delta, \p} (\dd x) \int_{[0,1]}  \Big((1-\p)  h (x(1-\delta y))+ \p h (x+(1-x)\delta y) \Big)\dd y\\
		&= \int_{[0,1]} \beta_{\delta, \p} (\dd x)\Bigg ( \dfrac{1-\p}{\delta x}\int_{x (1-\delta)}^x h (u)\dd u + \dfrac{\p}{\delta (1-x)}\int_{x}^{x+(1-x)\delta} h(u)\dd u \Bigg) . 
	\end{align*}
	If $h \equiv 0$ a.e. with respect to the Lebesgue measure, then the inner integrals vanish for every $x \in [0,1]$, and then $\beta_{\delta, \p}h=0$, which proves our claim. We postpone the derivation of the integral equation for the density of $\beta_{\delta,\p}$ until the end of the proof.

\item Let $f:[0,1] \mapsto [0,1]$ be the moment generating function of the distribution $\beta_{\delta, \p}$. Then, if $(X_n)_{n\geq0}$ is the chain with kernel $P_{\delta, \p}$ and initial distribution $\beta_{\delta, \p}$, for $t \in [0,1]$:
	\begin{align*}
		f(t)& = \E[e^{tX_0}] = \E [e^{tX_1}]  = (1-\p)\E [e^{tX_0 (1- \delta U)}] + \p \E [e^{t X_0 (1- \delta U)+ t\delta U}],
	\end{align*}
	where $U\sim \unif$ independent of $X_0$. Conditioning, we obtain:
	\begin{align*}
		f(t)&= \dfrac{(1-\p)}{\delta} \int_{0}^\delta f(t (1-x)) \dd x + \dfrac{\p}{\delta} \int_{0}^\delta f(t (1-x)) e^{tx} \dd x. 
	\end{align*}  
	Differentiating $k$ times:
	\begin{align*}
		f^{(k)} (t) & = \dfrac{(1-\p)}{\delta} \int_{0}^\delta (1-x)^{k}f^{(k)}(t (1-x)) \dd x \\
		& \quad + \dfrac{\p}{\delta} \int_{0}^\delta \sum_{j=0}^k \binom{k}{j}(1-x)^j x^{k-j}f^{(j)}(t (1-x)) e^{tx} \dd x. 
	\end{align*}
	Since $\mom{\beta_{\delta, p}}{j}= f^{(j)}(t) \vert_{t=0}$, we obtain:
	\begin{align*}
		   \mom{\beta_{\delta, p}}{k}= \mom{\beta_{\delta, p}}{k} \dfrac{1-\p}{\delta} \int_{0}^{\delta} (1-x)^{k} \dd x +   \sum_{j=0}^{k} \binom{k}{j}\mom{\beta_{\delta, p}}{j}\dfrac{\p}{\delta} \int_{0}^{\delta} (1-x)^{j} x^{k-j}\dd x. 
	\end{align*}
	The claim follows.
	\item From Proposition \ref{lemma:contraction}, part 1, we see that the $\vert p-q\vert$ is an upper bound for the Wasserstein distance between $\beta_{\delta, p}$ and $\beta_{\delta, q}$; then equality follows from \eqref{eq:1_2_moments}
\end{enumerate}
We now provide a rigorous derivation of the integral equation \eqref{eq:integral_equation}.  
Let $(X_n)$ be the Markov chain driven by $P_{\delta,p}$ and started from its stationary
distribution $\beta_{\delta,p}$. For ease of notation, set $X := X_0$ and $g := g_{\delta,p}$.
For $z \in (0,1)$, we obtain:
\begin{align*}
\prob[X_{1} \leq z] &= (1-p) \prob[x(1-\delta U)  \leq z] + p \prob[X(1-\delta U)+\delta U \leq z] \\
&= (1-p)\prob\bigg[ U \geq \frac{X-z}{\delta X}\bigg] + p \prob\bigg[U \leq \frac{z-X}{\delta(1-X)}\bigg]\\
&= (1-p)\bigg( \prob\bigg[ U \geq \frac{X-z}{\delta X}, X > z\bigg] + \prob\bigg[\underbrace{U \geq \frac{X-z}{\delta X}, X \leq z}_{\text{event A}}\bigg]\bigg ) \\
&\phantom{=}+   p \bigg( \prob\bigg[\underbrace{U \leq \frac{z-X}{\delta(1-X)}, X>z}_{\text{event B}}\bigg] + 
\prob\bigg[U \leq \frac{z-X}{\delta(1-X)}, X\leq z\bigg]  \bigg) 
\shortintertext{Observe that event A simplifies to  $\{X\leq z\}$ and event B has probability 0. So:}
\prob[X_{1} \leq z]&= (1-p)\prob\bigg[U \geq \frac{X-z}{\delta X}, X>z\bigg]+ p\prob\bigg[U \leq \frac{z-X}{\delta(1-X)},X\leq z\bigg] + (1-p)\prob(X \leq z). 
\end{align*}
Let $f_{(X,U)}(x,u)$ be the joint density of the random variables $U$ and $X$. Then, $f_{(X,U)}(x,u) = g(x)\mathbf{1}_{x, u \in [0,1]}$. So:
\begin{align*}
\prob\bigg[U \geq \frac{X-z}{\delta X}, X>z\bigg] &= \int_{z}^{\frac{z}{1-\delta} \wedge 1}\int_{\frac{x-z}{\delta x}}^{1} g(x) \dd u \dd x \\
&= \left( 1-\frac{1}{\delta} \right) \int_{z}^{\frac{z}{1-\delta} \wedge 1} g(x) \dd x  + \frac{z}{\delta} \int_{z}^{\frac{z}{1-\delta} \wedge 1}\frac{g(x)}{x} \dd x.
\end{align*}
Analogously:
\begin{align*}
\prob\bigg[U \leq \frac{z-X}{\delta(1-X)}, X\leq z\bigg] &=  \int_{0}^{\left(\frac{z-\delta}{1-\delta} \right) \vee 0}\int_{0}^{1} g(x) \dd u \dd x
+ \int_{\left(\frac{z-\delta}{1-\delta} \right) \vee 0}^{z}\int_{0}^{\frac{z-x}{\delta(1-x)}} g(x) \dd u \dd x \\ 
&=  \int_{0}^{\left(\frac{z-\delta}{1-\delta} \right) \vee 0} g(x) \dd x
+ \frac{1}{\delta}\int_{\left(\frac{z-\delta}{1-\delta} \right) \vee 0}^{z} g(x) \dd x
+ \frac{z-1}{\delta}\int_{\left(\frac{z-\delta}{1-\delta} \right) \vee 0}^{z} \frac{g(x)}{1-x} \dd x.
\end{align*}
Set $\wedge(z) \coloneqq  \left( \frac{z}{1-\delta} \right) \wedge 1$, $\vee(z) \coloneqq \left( \frac{z-\delta}{1-\delta} \right) \vee 0$ and $G(z) = \prob(X \leq z)$ (so that $G' = g$ $\mathrm{Lebesgue}$-a.e.) Joining the computations above, we arrive to the relation:
\begin{align*}
G(z) &= (1-p)\left(1-\frac{1}{\delta} \right) \int_{z}^{\wedge(z)}g(x)\dd x + \left(\frac{1-p}{\delta} \right)z\int_{z}^{\wedge(z)} \frac{g(x)}{x} \dd x\\
     &\phantom{=}+ p\int_{0}^{\vee(z)} g(x) \dd x +\frac{p}{\delta} \int_{\vee(z)}^{z} g(x) \dd x + \frac{p}{\delta}(z-1) \int_{\vee(z)}^{z}\frac{g(x)}{1-x}d x+ (1-p)G(z).
\end{align*}
Differentiating this expression, we get (omitting the argument $z$, for the sake of readability):
\begin{align}\label{eq:density}
g(z) &= (1-p)\left( 1-\frac{1}{\delta} \right) \left[ g(\wedge) \frac{\dd \wedge}{\dd z} - g \right] + 
z \left( \frac{1-p}{\delta}\right)\left[\frac{g(\wedge)}{\wedge} \frac{\dd \wedge}{\dd z} -\frac{g}{z}\right]+
\frac{1-p}{\delta}\int_{z}^{\wedge}\frac{g(x)}{x} \dd x \nonumber \\
&\phantom{=}+ p \left[ g(\vee)\frac{\dd \vee}{\dd z} \right] + 
\frac{p}{\delta}\left[ g-g(\vee) \frac{\dd \vee}{\dd  z}\right] +
\frac{p}{\delta}(z-1) \left[ \frac{g}{1-z} - \frac{g(\vee)}{1-\vee} \frac{\dd \vee}{\dd z}\right] \nonumber\\
&\phantom{=}+ \frac{p}{\delta}\int_{\vee}^{z}\frac{g(x)}{1-x} d x + (1-p)g. \nonumber\\
&= \frac{1-p}{\delta}\int_{z}^{\wedge} \frac{g(x)}{x} d x + \frac{p}{\delta}\int_{\vee}^{z} \frac{g(x)}{1-x} \dd x \nonumber\\
&\phantom{=}+ (1-p)g(\wedge)\frac{\dd \wedge}{\dd z} \left( 1-\frac{1}{\delta} + \frac{1}{\delta} \frac{z}{\wedge}\right)
+ p g(\vee) \frac{d \vee}{\dd z} \left( 1-\frac{1}{\delta} - \frac{1}{\delta} \frac{z-1}{1-\vee}\right).
\end{align}
Notice that:
\[
\frac{\dd \wedge}{\dd z}(z) = \left( \frac{1}{1-\delta} \right) 1_{[0,1-\delta)}(z)
\quad \text{and} \quad  
\frac{\dd \vee}{\dd z}(z) = \left( \frac{1}{1-\delta} \right) 1_{(\delta,1]}(z),
\]
and also:
\[
1-\frac{1}{\delta}+\frac{1}{\delta}\frac{z}{\wedge}  = \frac{1}{\delta}(x-(1-\delta)) 1_{[1-\delta,1]}(z)
\quad \text{and} \quad
1-\frac{1}{\delta} -\frac{1}{\delta}\frac{z-1}{1-\vee} =  -\frac{1}{\delta}(x-\delta) 1_{[0,\delta]}(z). 
\]
Consequently, the terms vanish at line \eqref{eq:density}. With this, we have verified the case \(\delta\in(0,1)\). The case \(\delta=1\) is analogous. 
\end{proof}

\subsection{Deferred Proofs for Section \ref{ss:ergodic}} \label{app:to:ss:ergodic}

\begin{proof} [Proof of Lemma \ref{lemma:unique_invariant}]

 Let $\mu,\nu,\mu^0,\mu^*\in\Ps$. 
 Let us introduce some notation. For any measurable function $\NLpf: [0,1] \times \mathcal{P} \mapsto [0,1]$ and $\mu\in \mathcal{P}$, we define the function $\pf_\mu: [0,1] \mapsto [0,1]$ via
$$\pf_\mu(x)\coloneqq \NLpf(x,\mu),\quad x\in[0,1].$$
Note that if $\NLpf$ is $(L_1,L_2)$-Lipschitz (recall Definition \ref{ass:Lipschitz}), then $\pf_\mu$ is $L_1$-Lipschitz.
 
 Construct the Markov chains $Z^{(\pf_\mu)}$ and $Z^{(\pf_\nu)}$ by means of the standard coupling (see Remark \ref{stcoupling}) from an arbitrary coupling of their initial values $Z^{(\pf_\nu)}_0\sim \mu^0$ and $Z^{(\pf_\mu)}_0\sim\mu^*$, respectively. As a consequence of Eq. \eqref{eq:pre_bound_00} and the fact that $\NLpf$ is $(L_1,L_2)$-Lipschitz, we conclude that
 	\begin{equation}\label{eq:bound-for-evolution-Q-chain}
	    \E [\vert Z^{(\pf_\mu)}_{n+1}-Z^{(\pf_\nu)}_{n+1} \vert]  \le \Bigg (  1- \dfrac{\delta (1-L_1)}{2}\Bigg)\E [\vert Z^{(\pf_\mu)}_n-Z^{(\pf_\nu)}_n\vert ] + \dfrac{\delta L_2}{2}\Was (\mu, \nu).
	\end{equation} 
 
 Iterating this inequality, we obtain
	\begin{align*}
	  \E [\vert Z^{(\pf_\mu)}_{n}-Z^{(\pf_\nu)}_{n} \vert] &\le \Bigg (  1- \dfrac{\delta (1-L_1)}{2}\Bigg)^n   \E [\vert Z^{(\pf_\mu)}_{0}-Z^{(\pf_\nu)}_{0} \vert] + \dfrac{ L_2}{1-L_1}\Was (\mu, \nu) . 
	\end{align*}	
	The above relation stands for any coupling of $(Z^{(\pf_\mu)}_{0},Z^{(\pf_\nu)}_{0})$. By considering a $\Was$-optimal coupling for the initial laws, we get
        \begin{equation}\label{eq:bound-for-iterarion-in-wasserstein}
            \begin{aligned}
    		\Was (\Law (Z^{(\pf_\mu)}_{n}), \Law (Z^{(\pf_\nu)}_{n})) &\le  \left (  1- \dfrac{\delta (1-L_1)}{2}\right)^n   \Was (\mu^0, \mu^*) + \dfrac{L_2}{1-L_1} \Was(\mu, \nu).
    	\end{aligned}
        \end{equation}
	
	In particular, if $\mu^0$ and $\mu^*$ are invariant distributions for $Z^{(\pf_\mu)}$, taking $\mu=\nu$ in the previous inequality yields 
	$$
	\Was (\mu^0, \mu^*) \le  \left (  1- \dfrac{\delta (1-L_1)}{2}\right)^n  \Was(\mu^0, \mu^*),
	$$
which in turn implies $\mu^0 = \mu^*$, i.e. the chain $Z^{(\pf_\mu)}$ admits at most one invariant distribution. Let us now prove the existence of an invariant distribution. Fix $\mu^0\in\Ps$, and set, for each $n\geq 0$, $\mu_n\coloneqq\Law (Z^{(\pf_\mu)}_{n})$ (in particular, $\mu_0=\mu^0$). Using with $\mu=\nu$ and $\mu^*=\mu_1$ \eqref{eq:bound-for-iterarion-in-wasserstein} we obtain
	\begin{align} \label{eq:crucial}
		\Was (\mu_n, \mu_{n+1} )\le  \Bigg (  1- \dfrac{\delta (1-L_1)}{2}\Bigg)^n\Was (\mu_0, \mu_1).
	\end{align}	
Thus the sequence $(\mu_n)_{n \ge 0}$ is a $\Was$-Cauchy sequence, and hence, it converges to a limiting law $\mu_\infty$. Now, for every continuous function $f: [0,1] \mapsto \R$, we have:
\begin{align*}
\mu_\infty f = \lim_{n \to \infty}\mu_{n} f = \lim_{n \to \infty} (\mu_{n-1} Q_\mu) f = \lim_{n \to \infty} \mu_{n-1} (Q_\mu f) = \mu_\infty (Q_\mu f) = (\mu_\infty Q_\mu) f, 
\end{align*}
where in the fourth equality we used that $Q_\mu f$ is continuous, which in turn follows from the form of $Q_\mu$ and the continuity $\p_\mu$. Thus the integrals of continuous functions with respect to the probability measures $\mu_\infty$ and $\mu_\infty Q_\mu$ coincide. Since this family is separating, we conclude that $\mu_\infty=\mu_\infty Q_\mu$. Hence, $\mu_\infty$ is invariant for $Q_\mu$. This ends the proof of the first statement.
	
Taking $\mu^0= T\mu$, $ \mu^*= T\nu$ in \eqref{eq:bound-for-iterarion-in-wasserstein} and letting $n \to \infty$ yields the second result. 
\end{proof}

\subsection{Wasserstein distance between Beta distributions}

The following result is used once in Section \ref{sec:case_study}. It should be standard, but we were unable to find a proof in the literature.

\begin{proposition}\label{prop:was_beta}
Let $a_1, a_2$, $b_1, b_2$ be positive numbers such that $a_1+a_2=b_1+b_2=M$ for some positive constant $M$. Then:
\begin{align*}
\Was (\betadist(a_1, a_2), \betadist(b_1, b_2) )= \dfrac{1}{M} \vert a_1-b_1 \vert.
\end{align*}
\end{proposition}
\begin{proof}
Since the family of $\betadist$ distributions is weakly continuous with respect to the parameters, it suffices to show the result for rational numbers $a_1, a_2, b_1, b_2$. So, let $N \in \mathbb{N}$ such that $a_i=r_i/N, b_i=s_i/N $. We have $M=(r_1+r_2)/N= (s_1+s_2)/N$. Let $L\coloneqq r_1+r_2$, and consider an i.i.d. family $(Z_i)_{i=1}^{L}$ of random variables with common law $\gammadist (1/N, 1)$. Then, $Z\coloneqq \sum_{k=1}^{L} Z_i \sim \gammadist (M, 1)$ and consequently $X\coloneqq \sum_{j=1}^{r_1}Z_j/Z \sim \betadist(a_1, a_2)$, $Y\coloneqq \sum_{j=1}^{s_1}Z_j/Z\sim \betadist (b_1, b_2)$.  Assuming $s_1 \ge r_1$, we have
\begin{align*}
\E [\vert X-Y \vert ] &= \E \left [ \dfrac{1}{Z}\sum_{j=r_1+1}^{s_1}Z_i  \right] = \E \Bigg[\betadist \left(\dfrac{s_1-r_1}{N}, \dfrac{L-(s_1-r_1)}{N}\right)\Bigg] = \dfrac{s_1-r_1}{NM}= \dfrac{b_1-a_1}{M}.
\end{align*}
Since $\vert \E [X]-\E [Y]\vert =(a_1-b_1)/M$, the claim follows.
\end{proof}

\section{Existence and uniqueness of solutions for the McKean--Vlasov Wright--Fisher (MVWF) equation}\label{ss:existence_MV}

In this Appendix we focus on simple conditions that guarantee that the MVWF equation:
\begin{align}\label{eq:MV_MV}
\dd{\nl X}_t= \NLb({\nl X}_t,\nl{\mu}_t)\dd t + \sqrt{{\nl X}_t(1-{\nl X}_t)}\,\dd W_t, \quad {\mu}_t= \Law (\nl X_t),\quad {\nl X}_0 \sim \overline \mu_0,
\end{align}
admits a unique weak solution on $[0, +\infty)$ (see definition \ref{def:weak_solution}). The Assumption \ref{ass:Lipschitz_bound} below will be enough for MVWF models with selection; under this assumption, the existence/uniqueness problem for the McKean--Vlasov equation will be addressed via an absolutely-continuous change of measure. The situation for the non-linear Wright--Fisher diffusion with pure mutation (which is necessary in order to work out the case study at Section \ref{sec:case_study}) is somewhat different, and we will say some words about this case when necessary. 

Fix $\mu_0 \in \mathcal{P}$. For $t \ge 0$,  let $\mathcal{P}_{t}\coloneqq\mathcal{P} (\mathcal{C}([0,t]; [0,1]))$ be the space of probability measures on the space $\mathcal{C} ([0, t]; [0,1])$ endowed with the Borel $\sigma$-field associated to the topology of the supremum norm such that the image measure under the coordinate map $(x_s)_{s \in [0, t]} \mapsto x_0$ equals $\mu_0$. We turn $\mathcal{P}_{t}$ into a metric space through the associated Wasserstein distance:
\begin{align} \label{def:was_C}
	\Was (t; \mu, \nu) = \inf_{\pi} \Bigg (\int_{\mathcal{C} ([0,t]; [0,1])^2} \bigg (\sup_{s \in [0, t]} \vert x_s-y_s\vert \bigg) \pi (\dd x \dd y) \Bigg) ,
\end{align}
where the infimum is taken over all the couplings $\pi$ of $(\mu, \nu)$. Observe that whenever $T \ge t$, the space $\mathcal{P}_t$ can be embedded into $\mathcal{P}_T$ and furthermore:
\begin{align*}
	\Was (t; \mu, \nu) = \inf_{\pi} \Bigg (\int_{\mathcal{C} ([0,T]; [0,1])^2} \bigg (\sup_{s \in [0, t]} \vert x_s-y_s\vert \bigg) \pi (\dd x \dd y) \Bigg),
\end{align*}
where now the infimum is taken over all the measures on $\mathcal{C} ([0, T]; [0,1])^2$ such that the restriction to $\mathcal{C}([0,t]; [0,1])^2$ has marginals $\mu$ and $\nu$.

Recall that we are working under the assumption \ref{ass:on_b}. Let $a(x)= \sqrt{x(1-x)}$. Consider the following condition.
\begin{assumption} \label{ass:Lipschitz_bound}
	For every $\mu, \nu \in \mathcal{P} (\mathcal{C} ([0, +\infty); [0,1]))$, and any $x \in [0,1]$:
	\begin{align*}
		\left \vert  \dfrac{b^{(\mu)}  (x, t)  - b^{(\nu)} (x, t)}{a (x)}  \right \vert \le C \Was (\mu_t, \nu_t),
	\end{align*}
 where we used the notation $b^{(\mu)}(x, t)=\NLb (x, \mu_t)$. 
\end{assumption}

\begin{theorem}\label{th:MV_existence_uniqueness}
    For every $\mu \in \mathcal{P}_t$ define $\Phi_t (\mu)\coloneqq \Law_\Pb(\difMu_{[0,t]}) \in \mathcal{P}_t$. Under Assumption \ref{ass:Lipschitz_bound}, the map $\mu \mapsto \Phi_t (\mu)$ admits a unique fixed point. In particular, Eq. \eqref{eq:MV} admits a unique weak solution on $t \ge 0$.
\end{theorem}
\begin{proof}
On a filtered probability space $(\Omega, \FF, (\FF_s), \prob)$, let $\difMu$ and $\difNu$ be strong solutions to equation \eqref{eq:linearized-sde} on $[0, t]$ for $\mu, \nu \in \mathcal{P}_t$, respectively, driven by the same Brownian motion $W$ and with identical initial distribution $\mu_0$. Consider the adapted process $(U_s: s \in [0, t])$ given by: 
	\begin{align*}
		U_s = \exp \Bigg\{    - \int_{0}^{s} \dfrac{b^{\nu} (\difNu_r, r)- b^{\mu} (\difNu_r, r)  }{a (\difNu_r)} \dd W_r - \dfrac{1}{2} \int_{0}^{s} \left (    \dfrac{b^{\nu} (\difNu_r, r)- b^{\mu} (\difNu_r, r)  }{a (\difNu_r)}\right)^2 \dd r   \Bigg\}.
	\end{align*}
	Assumption \ref{ass:Lipschitz_bound} ensures that $U$ is a $\FF_s$-martingale . Let $\Q$ be the probability measure defined on $\FF_t$ through  $\dd\Q/\dd\prob=  U_t$, and for $s \le t$ write $\Q_s$ for its restriction to $\FF_s$. Note that applying Girsanov's theorem (see \cite[Chap 4.]{IkedaWatanabe1989}, \cite[Sec. 6.3]{book:lipster-shiryayev-vol1} and \cite[Sec. 17.2,17.3]{book:lipster-shiryayev-vol2}), it follows that
    $$
    \Law_\Q(\difNu_{[0,t]}) = \Law_\Pb(\difMu_{[0,t]}). 
    $$
So, applying \eqref{eq:wass-TV-inequality}, we obtain
\begin{align*}
	\Was (t; \Phi_t (\mu), \Phi_t (\nu))&  \le d_{TV} (\Phi_t (\mu),\, \Phi_t(\nu))  = d_{TV} (\Law_\Pb(\difMu_{[0,t]}) ,\, \Law_\Pb(\difNu_{[0,t]}) ) \\
    &= d_{TV} (\Law_\Q(\difNu_{[0,t]}) ,\, \Law_\Pb(\difNu_{[0,t]}) )\le d_{TV} (\prob_t, \Q_t), \nonumber
\end{align*}
where in the last inequality we use that the total variation distance between the pushforward measures is dominated by the total variation distance of the original measures. Then by Pinsker's inequality (see e.g. \cite{tysbakov2009}, Lemma 2.5, at page 88), the last term in the right-hand side of the last inequality is smaller than the relative entropy between $\prob_t$ and $\Q_t$. Hence, for $t \ge 0$:
\begin{align*}
	\Was (t; \Phi_t (\mu), \Phi_t (\nu))
	&\le \sqrt {   2 \E_{\prob} \left [\log \dfrac{d \Q_t}{d \prob_t}\right]      } = \sqrt {   2 \E_{\prob} \left [\log(U_t) \right]      } .
	\end{align*}
	Applying It\^o's lemma, Assumption \ref{ass:Lipschitz_bound} yields
	\begin{align*}
 	2\E_{\prob} \left [\log(U_t) \right]  &\le   C  \int_{0}^{t} \Was^2 (\mu_s, \nu_s)\dd s,
	\end{align*}
from where it follows
\begin{align}\label{eq:was_contract}
	\Was^2 (t; \Phi_t (\mu), \Phi_t (\nu))
	& \le C \int_{0}^t \Was^2 (s; \mu, \nu) \dd s.
\end{align}
Fix $\mu^{(0; t)} \in \mathcal{P}_t$, and for $n \ge 0$ define $\mu^{(n+1; t)}\coloneqq \Phi_t (\mu^{(n; t)})$. Then, iterating the above relation we obtain:
\begin{align*}
	\Was^2 (t; \mu^{(n+1;t)}, \mu^{(n;t)}) \le \dfrac{C^{n}t^n}{n!} \Was^2 (t; \mu^{(1;t)}, \mu^{(0;t)}).
\end{align*}
Consequently, for every $t \ge 0$, the sequence $(\mu^{(n;t)})_{n \ge 1}$ is a Cauchy sequence on $\mathcal{P}_t$, and hence convergent to a limit $\mu^{(\infty; t)}$, say. Notice that from \eqref{eq:was_contract}, we have
\begin{align*}
	\Was^2 (t; \Phi_t (\mu^{(\infty; t)}), \mu^{(n+1;t)})
	& \le C t \Was^2 (t;\mu^{(\infty; t)}, \mu^{(n;t)}) ,
\end{align*}
taking limit $n\to\infty$, we obtain that $\mu^{(\infty; t)}$ is a fixed point for $\Phi_t$. Finally, if $\mu_1$, $\mu_2$ are fixed points for $\Phi_t$, \eqref{eq:was_contract} yields
$$
\Was^2 (t; \mu_1, \mu_2))
	 \le C \int_{0}^t \Was^2 (s; \mu_1, \mu_2) \dd s,
$$
and from Gronwall's lemma we conclude $\mu_1=\mu_2$.

To prove the second claim, a general remark is in order. Assume that we have a sequence $(\nu_N)_{ N \ge 1}$ of probability measures on the metric space $\mathcal{C} ([0, t]; [0,1])$ converging weakly to a probability measure $\nu^{(t)}$. If $(\GG_s)_{s \in [0,t]}$ is the canonical filtration, then the sequence of restrictions $(\nu_N \vert_{\GG_s})_{N \ge 1}$ converge weakly to the measure $\nu^{(t)} \vert_{\GG_s}$ for every $s \le t$. This remark applies to the above sequences $(\mu^{(n;t)})_{n \ge 1}$, and moreover, the uniqueness just proved yields
\begin{align*}
\mu^{(\infty; t)} \vert_{\GG_s} = \mu^{(\infty; s)}. 
\end{align*}
Thus, we can unambiguously define a probability measure on $\mathcal{C} ([0, +\infty); [0,1])$ via:
\begin{align*}
\mu^{(\infty)} := \mu^{(\infty; M)} \text{ on } \GG_M, M \ge 1.
\end{align*}
Define, finally, on a certain probability space $(\Omega', \FF', (\FF'_t), \prob')$ carrying a standard Brownian motion $W$ the process $\nl X$ as the unique solution of the SDE:
\begin{align*}
\dd \nl X_t = b^{(\mu^{(\infty)})} (\nl X_t, t)\dd t +\sqrt{\nl X_t (1- \nl X_t)}\dd W_t, \quad \nl X_0 \sim \mu_0.
\end{align*}
By construction, for every $t \ge 0$ we have $\Law (\nl X_t)= \mu^{(\infty)}_t$. This proves our claim.
\end{proof}

\end{document}